\documentclass{amsart}
\usepackage{amsthm}
\usepackage{amssymb}
\usepackage{amsmath}
\usepackage{color}
\usepackage{graphicx}
\usepackage{hyperref}
\usepackage{indentfirst}
\usepackage{enumitem}
\usepackage{caption} 
\usepackage{subcaption} 

\newtheorem{theorem}{Theorem}[section]
\newtheorem{lemma}[theorem]{Lemma}
\newtheorem{proposition}[theorem]{Proposition}

\theoremstyle{remark}
\newtheorem{remark}[theorem]{Remark}

\theoremstyle{definition}
\newtheorem{definition}[theorem]{Definition}

\newcommand{\N}{\mathbb{N}}
\newcommand{\Z}{\mathbb{Z}}
\newcommand{\R}{\mathbb{R}}
\newcommand{\C}{\mathbb{C}}
\newcommand{\D}{\mathbb{D}}
\newcommand{\scs}{\mathcal{J}(\lambda)}
\newcommand{\tu}{\tilde{u}}
\newcommand{\tv}{\tilde{v}}
\newcommand{\tw}{\tilde{w}}
\newcommand{\tj}{\tilde{J}}

\DeclareMathOperator{\ind}{ind}
\DeclareMathOperator{\wind}{wind}

\DeclareMathOperator{\sli}{sl}
\DeclareMathOperator{\expe}{exp}

\title{$3-2-1$ foliations for Reeb flows on the tight 3-sphere}

\author{Carolina Lemos de Oliveira}
\thanks{This study was financed by grant \texttt{\#}2016/10466-5, São Paulo Research Foundation (FAPESP), by the Coordenação de Aperfeiçoamento de Pessoal de Nível Superior - Brasil (CAPES) - Finance Code 001 and by Serrapilheira Institute through a grant awarded to Prof. Vinicius Ramos.}
\address{Instituto de Matemática e Estatística - Universidade do Estado do Rio de Janeiro,
	Rio de Janeiro - RJ, Brasil}
\email{carolina.lemos@ime.uerj.br}

\begin{document}
	
\subjclass[2020]{Primary 53D35; Secondary 37J46, 37J55}
\keywords{Hamiltonian dynamics, pseudo-holomorphic curves, Reeb flows}

\begin{abstract}
	We study the existence of $3-2-1$ foliations adapted to Reeb flows on the tight $3$-sphere. These foliations admit precisely three binding orbits whose Conley-Zehnder indices are $3$, $2$, and $1$, respectively. All regular leaves are disks and annuli asymptotic to the binding orbits. 
	Our main results provide sufficient conditions for the existence of $3-2-1$ foliations with prescribed binding orbits.
	We also exhibit a concrete Hamiltonian on $\R^4$ admitting $3-2-1$ foliations when restricted to suitable energy levels.
\end{abstract}

\maketitle

\tableofcontents

\section{Introduction}

In this paper, we use the theory of pseudo-holomorphic curves in symplectizations to study the existence of transverse foliations for Reeb flows on the tight $3$-sphere. These flows are equivalent to Hamiltonian flows on star-shaped energy surfaces in $\mathbb{R}^4$.

Some of the relevant results on the existence of global surfaces of section for Reeb flows on the tight $3$-sphere follow from the theory of pseudo-holomorphic curves in symplectizations, initiated by Hofer in \cite{hofer1993pseudoholomorphic}.
A global surface of section for a $3$-dimensional Reeb flow is an embedded surface with boundary whose interior is transverse to the flow and whose boundary consists of periodic orbits. The surface intersects every orbit, and a first return map describes the qualitative properties of the flow.
Hofer, Wysocki, and Zehnder showed in \cite{hwz1998} that a dynamically convex Reeb flow on the $3$-sphere admits disk-like global surface of section.
The disk is part of an $S^1$-family of global surfaces of section, forming an open book decomposition.
This result applies to Hamiltonian flows on strictly convex energy surfaces in $\R^4$.

A generalization of global surfaces of section and open books adapted to Reeb flows are the transverse foliations. They consist of a singular foliation whose singular set is a finite set of periodic orbits, called binding orbits. The regular leaves are surfaces asymptotic to the binding orbits and transverse to the flow.
Such foliations may imply the existence of other periodic orbits and homoclinics to a hyperbolic binding orbit.
One may obtain transverse foliations as projections of finite energy foliations in the symplectization. In \cite{hwz2003}, the authors generalize the results in \cite{hwz1998} for generic star-shaped energy surfaces in $\R^4$. They prove that a Hamiltonian flow restricted to a generic star-shaped energy surface admits a finite energy foliation on its symplectization.
There are several possible configurations of transverse foliations in \cite{hwz2003}. In any case, the binding orbits have Conley-Zehnder indices $3$, $2$, or $1$, and the regular leaves are punctured spheres. If a single periodic orbit forms the binding, then it has Conley-Zehnder index $3$, and the transverse foliation determines
an open book decomposition with disk-like pages, which are global surfaces of section for the flow.

In \cite{dePS2013, depaulo2019}, de Paulo and Salom\~ao provide concrete examples of Reeb flows on the tight $3$-sphere that admit transverse foliations other than open books. They study Hamiltonian flows close to a critical energy surface containing two strictly convex sphere-like subsets that touch each other at a saddle-center equilibrium point. For energies slightly above the critical value, they show that the flow admits a transverse foliation, called a $3-2-3$ foliation, with three binding orbits. The Lyapunoff orbit near the equilibrium is one of the binding orbits, and the other binding orbits have Conley-Zehnder index $3$. See also \cite{bramham2015periodic, colin2020existence, fish2018connected, wendl2008finite} for interesting results on the existence of transverse foliations.

A natural question that arises from the results in \cite{hwz1998} and \cite{hwz2003} is to find necessary and sufficient conditions for a finite set of closed orbits to be the binding of a transverse foliation.
Some recent studies answer this question in the case of disk-like global surfaces of section.
Characterization of the closed orbits that bound a disk-like global surface of section, assuming that the Reeb flow on $S^3$ is dynamically convex, is given in \cite{hryniewicz2012fast} and \cite{hryniewicz2014systems}. A simple orbit bounds a disk-like global surface of section if and only if it is unknotted and has self-linking number $-1$.
In \cite{hs2011}, the authors characterize the closed orbits that bound a disk-like global surface of section for a nondegenerate Reeb flow on the tight $3$-sphere. More precisely, a simple closed orbit $P$ bounds a disk-like global surface of section if and only if it is unknotted, has Conley-Zehnder index $\geq 3$, has self-linking number $-1$, and all index-$2$ orbits link with $P$.
In both cases, the global surface of section is a page of
an open book decomposition.

In the present work, we study the existence of transverse foliations adapted to Reeb flows on the tight $3$-sphere that admit a binding orbit with Conley-Zehnder index $1$. More precisely, we study a particular type of transverse foliation, called $3-2-1$ foliation, which has three binding orbits $P_3$, $P_2$, and $P_1$, whose Conley-Zehnder indices are $3$, $2$, and $1$, respectively. The rigid leaves are formed by a plane asymptotic to $P_2$, a pair of cylinders connecting $P_3$ to $P_2$ forming a $2$-torus, and a cylinder connecting $P_2$ to $P_1$. There is a family of planes asymptotic to $P_3$, breaking at each end into a cylinder connecting $P_3$ and $P_2$ and the plane asymptotic to $P_2$. There is also a family of cylinders connecting $P_3$ to $P_1$, which breaks at each end into a cylinder connecting $P_3$ and $P_2$ and the cylinder connecting $P_2$ and $P_1$.
Our main result provides sufficient conditions for a set of three closed orbits to be the binding orbits of a $3-2-1$ foliation.
We also exhibit a Hamiltonian on $\R^4$ whose flow restricted to suitable energy surfaces admits a $3-2-1$ foliation.

The $3-2-1$ foliations are one of the possible transverse foliations established in \cite{hwz2003}. We expect to use our methods to find concrete Hamiltonians admitting more general transverse foliations.

\subsection{Main results}
{Our objective is to study the existence of transverse foliations for Reeb flows associated to tight contact forms on $S^3$.} 
A \textit{transverse foliation} for a flow $\varphi^t$ on a closed oriented $3$-manifold $M$ consists of
	\begin{itemize}
		\item A finite set $\mathcal{P}$ of simple periodic orbits of $\varphi^t$, called \textit{binding orbits};
		\item A smooth foliation of $M\setminus \cup_{P\in \mathcal{P}}P$ by properly embedded surfaces. 
		Every leaf is transverse to $\varphi^t$ and has an orientation induced by $\varphi^t$ and $M$. 
		For every leaf $\dot{\Sigma}$ there exists a compact embedded surface $\Sigma\ \hookrightarrow M$ so that $\dot{\Sigma} = \Sigma\setminus \partial \Sigma$ and $\partial \Sigma$ is a union of connected components of $\cup_{P\in \mathcal{P}}P$. An end $z$ of $\dot{\Sigma}$ is called a
		puncture. 
		To each puncture $z$ there is an associated component $P_z\in \mathcal{P}$ of $\partial \Sigma$, called the asymptotic limit of $\dot{\Sigma}$ at $z$. A puncture $z$ of $\dot{\Sigma}$ is called positive if the orientation on $P_z$ induced by $\Sigma$ coincides with the orientation induced by $\varphi^t$. Otherwise, $z$ is called negative.
	\end{itemize}
This definition follows \cite{Hryniewicz2018} and is based on the finite energy foliations from \cite{hwz2003}.

Let $\lambda$ be a contact form on $S^3$, that is, $\lambda\wedge d\lambda$ is a volume form.
The \textit{Reeb vector field} $R_\lambda$ associated to $\lambda$ is uniquely determined by 
\begin{equation}
	i_{R_\lambda}d\lambda\equiv 0,~~~~~~i_{R_\lambda}\lambda\equiv 1.
\end{equation}
The flow $\{\varphi^t\}$ of $R_\lambda$ is called the \textit{Reeb flow} of $\lambda$.
The contact structure associated to $\lambda$ is the $2$-plane distribution 
$	\xi=\ker\lambda.$ 
We denote by 
$\pi:TS^3\to \xi$ 
the projection onto $\xi$ uniquely determined by $\ker \pi=\R R_\lambda$. 

An embedded disk $D\subset S^3$ satisfying $T\partial D\subset \xi \text{ and } T_pD\neq \xi_p,~\forall p\in \partial D$
is called an \textit{overtwisted disk}. 
The contact form $\lambda$ is \textit{tight} if the contact structure $\xi=\ker \lambda$ does not admit an overtwisted disk. 
Consider $\R^4$ with coordinates $(x_1,x_2,y_1,y_2)$. The Liouville $1$-form 
\begin{equation}\label{eq:liouville-form}
	\lambda_0=\frac{1}{2}\sum_{i=1}^{2}x_idy_i-y_idx_i
\end{equation}
restricts to a contact form on $S^3$.
By results of Bennequim \cite{bennequin-entrelacements} and Eliashberg \cite{eliashberg1992contact}, 
we know that, up to diffeomorphism, any tight contact form on $S^3$ is of the form $f\lambda_0|_{S^3}$ for some smooth function $f:S^3\to \R^+$. 
If $E:=\{z\sqrt{f(z)}|z\in S^3\}$ is a regular energy level of a Hamiltonian function on $(\R^4,\omega_0:=d\lambda_0)$, then the associated Hamiltonian flow restricted to $E$ is equivalent to the Reeb flow of $f\lambda_0|_{S^3}$.

We call a pair $P=(x,T)$, where $x:\R\to S^3$ is a periodic trajectory of $\dot{x}(t)=R_\lambda(x(t))$ and $T>0$ is a period of $x$, a \textit{Reeb orbit}. 
We identify $P=(x,T)$ with the element of $\dfrac{C^\infty(\R/\Z,S^3)}{\R/\Z}$ induced by the loop 
	$x_T:{\R}/{\Z}\to S^3,~~~x_T(t)=x(Tt),$ 
where the quotient is relative to the translations $t\mapsto x_T(t+c)$.
{By abuse of notation we sometimes write $x(\R)=P$.}
If $T$ is the minimal positive period of $x$, we call $P$ \textit{simple}. If $m\geq 1$ is an integer, the $m^{th}$ iterate of $P$ will be denoted by $P^m:=(x,mT)$.
We denote the set of Reeb orbits by $\mathcal{P}(\lambda)$.

The Reeb flow $\varphi^t$ preserves the contact structure $\xi$ and the maps $d\varphi^t:\xi_{x(0)}\to \xi_{x(t)}$ are $d\lambda$-symplectic.
We call the orbit $P=(x,T)$ \textit{nondegenerate} if $1$ is not an eigenvalue of $d\varphi^T|_{\xi_{x(0)}}$. If every orbit $P\in \mathcal{P}(\lambda)$ is nondegenerate, then the contact form $\lambda$ is called \textit{nondegenerate}.

A simple orbit $P=(x,T)\in \mathcal{P}(\lambda)$ is called \textit{unknotted} if $x(\R)$ is an unknot.
We say that a set of simple orbits $\cup_{i=1}^n P_i=(x_i,T_i)$ is an unlink if $\cup_{i=1}{x_i}(\R)$ is an unlink.
We say that two orbits $P=(x,T)$ and $\bar{P}=(\bar{x},\bar{T})$ are \textit{linked} if the linking number
${\rm lk}(x(\R),\bar{x}(\R))$ 
is nonzero.

Before stating our main results, we give some necessary definitions.

\begin{definition}[Strong transverse section]\label{de:strong-transverse-section}
	Let $\lambda$ be a contact form on $S^3$.
	Let $\Sigma\hookrightarrow S^3$ be a compact embedded surface such that $\dot{\Sigma}=\Sigma\setminus \partial \Sigma$ is transverse to the Reeb vector field $R_\lambda$ and $\partial\Sigma$ consists of a finite number of simple orbits in $\mathcal{P}(\lambda)$. 
	$\Sigma$ is called a \textit{strong transverse section} if every connected component of $\partial \Sigma$ 
	associated to an orbit $P=(x,T)$ has a neighborhood on $\Sigma$ parametrized by $\psi:(r_0,1]\times \R/\Z\to \Sigma$
	such that $\psi(1,t)=x_T(t)$, $\forall t\in \R/\Z$, and the section of $x_T^*\xi$ defined by 
	\begin{equation}\label{eq:equacao-defi-strong}
		\eta(t)=\pi\frac{\partial}{\partial r}\psi(r,t)\big|_{r=1}
	\end{equation}
	satisfies
	$$d\lambda(\eta(t),\mathcal{L}_{R_\lambda}\eta(t))\neq 0, ~\forall t\in \R/ \Z ,$$
	where $\mathcal{L}_{R_\lambda}$ is the Lie derivative with respect to $R_\lambda$. 
	See Remark \ref{rm:strong-transverse-section}.
\end{definition}

\begin{definition}[$3-2-1$ foliation]\label{de:3-2-1-foliation}
	Let $\lambda$ be a contact form on $S^3$. 
	A \textit{$3-2-1$ foliation} 
	adapted to $\lambda$ is a transverse foliation for the associated Reeb flow satisfying the following properties.
	The set $\mathcal{P}$ of binding orbits consists of three simple orbits $P_1$, $P_2$, and $P_3$ with Conley-Zehnder indices $1$, $2$, and $3$ respectively, self-linking number $-1$, and such that $P_1\cup P_2\cup P_3$ is an unlink. 
	The foliation $\mathcal{F}$ of $S^3\setminus \cup_{P\in \mathcal{P}}P$ is as follows:
	\begin{itemize}
		\item $\mathcal{F}$ contains a pair of cylinders $U_1$ and $U_2$, both asymptotic to $P_3$ at their positive punctures and $P_2$ at their negative punctures. $T:=U_1\cup U_2\cup P_2\cup P_3$ is homeomorphic to a torus and $T\setminus P_3$ is $C^1$-embedded. $T$ divides $S^3$ into two closed regions $\mathcal{R}_1$ and $\mathcal{R}_2$.
		\item $\mathcal{F}$ contains a disk $D\subset \mathcal{R}_1$ asymptotic to $P_2$ at its positive puncture and a cylinder $V \subset \mathcal{R}_2$ asymptotic to $P_2$ at its positive puncture and $P_1$ at its negative puncture. $D\cup P_2\cup V$ is a $C^1$-embedded disk with boundary $P_1$ and transverse to $T$.
		\item $\mathcal{F}$ contains a one-parameter family of disks $F_\tau\subset \mathcal{R}_1$, $\tau\in(0,1)$, all of them asymptotic to $P_3$ at their positive punctures, such that $\{F_\tau\}_{\tau\in (0,1)}\cup U_1\cup U_2\cup D$ foliate $\mathcal{R}_1\setminus(P_2\cup P_3)$.
		\item $\mathcal{F}$ contains a one-parameter family of cylinders $C_\tau\subset \mathcal{R}_2$, $\tau\in(0,1)$, all of them asymptotic to $P_3$ at their positive punctures and $P_1$ at their negative punctures, such that $\{C_\tau\}_{\tau\in(0,1)}\cup U_1\cup U_2\cup V$ foliate $\mathcal{R}_2\setminus (P_1\cup P_2\cup P_3)$.
		\item The closure of every leaf of $\mathcal{F}$ is a strong transverse section.
	\end{itemize}
	See figure \ref{fi:3-2-1-foliation}. The Conley-Zehnder index and the self-linking number are discussed in Section \ref{se:preliminaries}.
\end{definition}

\begin{figure}
	\centering
	\includegraphics[scale=0.9]{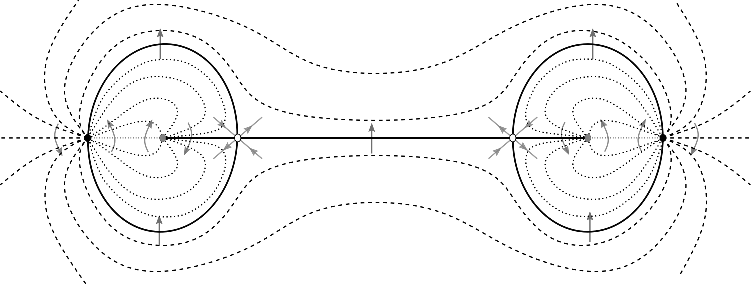}
	\caption{A section of a $3-2-1$ foliation. The black dots represent the orbit $P_3$ with Conley-Zehnder index $3$, the white dots represent the hyperbolic orbit $P_2$ with Conley-Zehnder index $2$, and the gray dots represent the orbit $P_1$ with Conley-Zehnder index $1$. 
	The bold curves represent the two rigid cylinders connecting $P_2$ and $P_3$, the rigid cylinder connecting $P_1$ and $P_2$, and the rigid disk with boundary $P_2$. The dashed curves represent the family of disks with boundary $P_3$. The dotted curves represent the family of cylinders connecting $P_1$ and $P_3$. The arrows indicate the direction of the Reeb vector field. The $3$-sphere is viewed as $\R^3\cup \{\infty\}$.}
	\label{fi:3-2-1-foliation}
\end{figure} 

A complex structure $J$ on $\xi$ is $d\lambda$-compatible if the bilinear form $d\lambda(\cdot,J\cdot)$ is a positive inner product on $\xi_p$, $\forall p\in S^3$. 
The space of $d\lambda$-compatible structures is nonempty and will be denoted by $\mathcal{J}(\lambda)$. 
For each $J\in \mathcal{J}(\lambda)$, the pair $(\lambda,J)$ determines a natural almost complex structure $\tj$ in the symplectization $\R\times S^3$ of $S^3$, given by \eqref{eq:J-til}. 
We consider $\tj$-holomorphic curves $\tu:S^2\setminus \Gamma\to \R\times S^3$, where the domain is the Riemann sphere with a finite set $\Gamma$ of punctures removed. 
Due to results of \cite{hofer1993pseudoholomorphic} and \cite{hofer1996properties1}, if $\tu$ satisfies an energy condition, it approaches closed orbits of $R_\lambda$ near the punctures. 
We postpone the precise definitions and statements to Section \ref{se:preliminaries}.
In what follows we use the notation $\tj=(\lambda,J)$.

Now we state the main results of this paper.
\begin{theorem}\label{theo:main-theorem}
	Let $\lambda$ be a tight contact form on $S^3$.
	Assume that there exist Reeb orbits $P_1=(x_1,T_1),~P_2=(x_2,T_2),~P_3=(x_3,T_3)\in \mathcal{P}(\lambda)$ that are nondegenerate, simple, and have Conley-Zehnder indices respectively $1$, $2$, and $3$. 
	Assume further that the orbits $P_1$, $P_2$, and $P_3$ are unknotted, $P_i$ and $P_j$ are not linked for $i\neq j$, $i,j\in\{1,2,3\}$, and the following conditions hold: 
	\begin{enumerate}[label=(\roman*)]
		\item Every orbit in $\mathcal{P}(\lambda)$ having period $\leq T_3$ is nondegenerate;
		\item $P_2$ is the unique orbit in $\mathcal{P}(\lambda)$ with Conley-Zehnder index $2$ and period less than $T_3$ that is not linked to $P_3$; 
		\item $P_1$ is the unique orbit in $\mathcal{P}(\lambda)$ with Conley-Zehnder index $1$ and period less than $T_2$ that is not linked to $P_2$; 
		\item There exists $J\in\mathcal{J}(\lambda)$ such that the almost complex structure $\tj=(\lambda,J)$ admits a finite energy plane $\tu:\C\to \R\times S^3$ asymptotic to $P_3$ at it positive puncture $z=\infty$;
		\item There is no $C^1$-embedding $\Psi:S^2\subset \R^3\to S^3$ such that $\Psi({S^1\times \{0\}})=x_2(\R)$ and 
		each hemisphere is a strong transverse section.
	\end{enumerate}	
	Then there exists a $3-2-1$ foliation adapted to $\lambda$ with binding orbits $P_1$, $P_2$, and $P_3$. Consequently, there exists at least one homoclinic orbit to $P_2$. 
\end{theorem}
{\begin{remark}
		We expect to weaken hypothesis \textit{(iv)} in Theorem \ref{theo:main-theorem} by assuming the existence of a disk with boundary $P_3$ and interior transverse to the Reeb vector field.
		It is expected that an analysis similar to that of \cite{depaulo2022,depaulo2019} will show that real-analytic flows admitting a $3-2-1$ foliation have positive topological entropy if no two branches of the stable and unstable manifold of $P_2$ coincide.
\end{remark}}

Condition \textit{(v)} in Theorem \ref{theo:main-theorem} is necessary to the existence of a $3-2-1$ foliation. This is the content of Proposition \ref{TH:NECESSIDADE-3} below.

\begin{proposition}\label{TH:NECESSIDADE-3}
	Assume that there exists a $3-2-1$ foliation adapted to the contact form $\lambda$ on $S^3$ and let $P_2=(x_2,T_2)$ be the binding orbit with Conley-Zehnder index $2$, as in Definition \ref{de:3-2-1-foliation}. Then there is no $C^1$-embedding $\psi:S^2\to S^3$ such that $\psi({S^1\times \{0\}})=x_2(\R)$ 
	and each hemisphere is a strong transverse section.
\end{proposition}


The following theorem gives another set of sufficient conditions for the existence of a $3-2-1$ foliation. Some hypotheses are more restrictive than the hypotheses of Theorem \ref{theo:main-theorem}, but we do not assume any non-degeneracy condition for the contact form.
\begin{theorem}\label{pr:example}
	Let $\lambda$ be a tight contact form on $S^3$.
	Assume that there exist Reeb orbits $P_1=(x_1,T_1),~P_2=(x_2,T_2),~P_3=(x_3,T_3)\in \mathcal{P}(\lambda)$ that are nondegenerate, simple, and have Conley-Zehnder indices respectively $1$, $2$, and $3$. 
	Assume further that the orbits $P_1$, $P_2$, and $P_3$ are unknotted, $P_i$ and $P_j$ are not linked for $i\neq j$, $i,j\in\{1,2,3\}$, and the following conditions hold: 
	\begin{enumerate}[label=(\roman*)]
		\item $T_1<T_2<T_3<2T_1$;
		\item If $P=(x,T)\in\mathcal{P}(\lambda)$ satisfies $P\neq P_3,~T\leq T_3$ \text{ and } ${\rm lk}(P,P_3)=0$, then $P\in\{P_1,P_2\}$.
		\item There exists $J\in\mathcal{J}(\lambda)$ such that the almost complex structure $\tj=(\lambda,J)$ admits a finite energy plane $\tu:\C\to \R\times S^3$ asymptotic to $P_3$ at its positive puncture $z=\infty$ and a finite energy cylinder 
		$\tw:\C\setminus \{0\}\to \R\times S^3$ asymptotic to $P_3$ at its positive puncture $z=\infty$ and $P_1$ at its negative puncture $z=0$;
		\item There exists no $C^1$-embedding $\Psi:S^2\subset \R^3\to S^3$ such that $\Psi({S^1\times \{0\}})=x_2(\R)$ and 
		each hemisphere is a strong transverse section.
	\end{enumerate}	
	Then there exists a $3-2-1$ foliation adapted to $\lambda$ with binding orbits $P_1$, $P_2$, and $P_3$. Consequently, there exists at least one homoclinic orbit to $P_2$. 
\end{theorem}	

	An interesting application of Theorem \ref{pr:example} is in the study of bifurcations of finite energy foliations. 
Consider a one-parameter family of Reeb flows on the tight $3$-sphere admitting adapted open books with disk-like pages. Hypotheses (i) and (ii) can be checked if the orbits $P_1$, $P_2$, and $P_3$ bifurcate from the binding orbit at some parameter value. 
See Remark \ref{re:example-bifrucation} for an example of this phenomenon.
The study of more general bifurcations of finite energy foliations is a work in progress of the author, P. Salomão, and A. Schneider. When one of the binding orbits of a finite energy foliation bifurcates into a finite set of binding orbits, it is expected that the transverse foliation bifurcates accordingly and the new Reeb orbits become part of the binding set.

 
\subsection{An example of Reeb flow admitting a $3-2-1$ foliation}

Consider $\R^4$ with coordinates $(x_1,y_1,x_2,y_2)$ and equipped with the canonical symplectic form 
$\omega_0=\sum_{i=1}^2dx_i\wedge dy_i.$
Consider the Hamiltonian function $H=H_\epsilon:\R^4\to \R$ defined by
 \begin{equation}\label{eq:hamiltoniana}
 	H(x_1,y_1,x_2,y_2)=H_1(x_1,y_1)+H_2(x_2,y_2),
 \end{equation} 
\begin{equation}
	\begin{aligned}
		H_1(x_1,y_1)&=\frac{x_1^2+y_1^2}{2}\\
		H_2(x_2,y_2)&=(x_2^2+y_2^2)^2-\epsilon(x^2_2+y_2^2)x_2-\frac{\epsilon}{2}(x_2^2+y_2^2)
	\end{aligned}
\end{equation}
If $\epsilon>0$, then $H_2$ has exactly three critical points: a saddle $(p_2,0)$, a local maximum $(p_1,0)$ and a minimum $(p_3,0)$, where $p_2<p_1=0<p_3$. See Figure \ref{fi:niveis-de-energia}.

\begin{figure}
	\centering
	\includegraphics[width=0.65\textwidth]{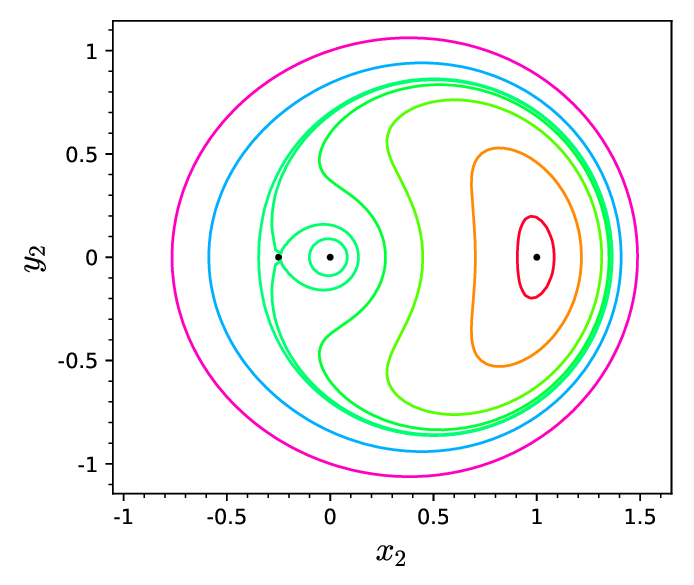}
	\caption{Energy levels of $H_2$ for $\epsilon=1$.} 
	\label{fi:niveis-de-energia}
\end{figure}

The energy level $S:=H^{-1}\left(\frac{1}{2}\right)$ is star-shaped:  
If $z=(x_1,y_1,x_2,y_2)\in S$ and $Y$ is the radial vector field $Y(z)=\frac{1}{2}z$, then, for sufficiently small $\epsilon$, we have
\begin{equation*}
	\begin{aligned}
		dH_z(Y)&=\frac{x_1^2+y_1^2}{2}+2(x_2^2+y_2^2)^2-\epsilon(x_2^2+y_2^2)x-\frac{\epsilon}{2}(x_2^2+y_2^2)-\epsilon(x_2^2+y_2^2)x_2\\
		&=H(z)+(x_2^2+y_2^2)^2-\epsilon(x_2^2+y_2^2)x_2\\
		&>0
	\end{aligned}
\end{equation*}
It follows that $S$ is diffeomorphic to $S^3$ and $\lambda:=\lambda_0|_S$ is a contact form on $S$. 
The Hamiltonian vector field associated to $H:\R^4\to \R$ is given by 
\begin{equation}\label{eq:hamiltonian-vf}
	X_{H}(x_1,y_1,x_2,y_2)=(-y_1,x_1,-P(x_2,y_2),Q(x_2,y_2)),
\end{equation}
where 
\begin{equation}
	\begin{aligned}
		Q(x_2,y_2)&=\dfrac{\partial H_2(x_2,y_2)}{\partial x_2}=4x_2(x_2^2+y_2^2)-3\epsilon x_2^2-\epsilon y_2^2 -\epsilon x_2 \\
		P(x_2,y_2)&=\dfrac{\partial H_2(x_2,y_2)}{\partial y_2 }=4y_2(x_2^2+y_2^2)-2\epsilon x_2y_2-\epsilon y_2.
	\end{aligned}
\end{equation}
The Reeb vector field $R_\lambda$ is given by 
$	R_\lambda(z)=h(z)X_{H}(z),~~\text{for }z=(x_1,y_1,x_2,y_2)\in S,$ 
where
\begin{equation}
	h(z)=\lambda_z(X_{H}(z))^{-1}=2(x_1^2+y_1^2+x_2Q(x_2,y_2)+y_2P(x_2,y_2))^{-1}.
\end{equation}	
Consider the Reeb orbits $P_1=(\gamma_1, T_1)$, $P_2=(\gamma_2, T_2)$ and $P_3=(\gamma_3, T_3)$ where 
\begin{equation}\label{eq:defi-orbitas}
	\gamma_i(t)=\left(r_i\cos\left(\frac{2}{r_i^2}t\right),r_i\sin\left(\frac{2}{r_i^2}t\right),p_i,0\right), ~~\text{for }	r_i=\sqrt{1-2H_2(p_i,0)},
\end{equation}
 
\begin{equation}\label{eq:periodo-exemplo}
	T_i=\pi r_i^2.
\end{equation}
Note that $H_2(p_3,0)<H_2(p_2,0)<H_2(p_1,0)$, which implies that
\begin{equation}\label{eq:t1-t2-t3}
	T_1<T_2<T_3.
\end{equation}
Since $H_2(p_i,0)={O}(\epsilon^2)$, for each $i=1,2,3$, we have 
\begin{equation}\label{eq:t3menort1}
	T_3<2T_1,
\end{equation}
for sufficiently small $\epsilon$.
It is easy to see that $P_1\cup P_2 \cup P_3$ is a trivial link.

\begin{proposition}\label{pr:exemplo-prova}
	For sufficiently small $\epsilon$, the contact form $\lambda=\lambda_0|_S$, where $S=H^{-1}(\frac{1}{2})$, and the Reeb orbits $P_1$, $P_2$, and $P_3$, defined by \eqref{eq:defi-orbitas}-\eqref{eq:periodo-exemplo}, satisfy the hypotheses of Theorem \ref{pr:example}. Therefore, there exists a $3-2-1$ foliation adapted to $\lambda$ with binding orbits $P_1$, $P_2$, and $P_3$. 
\end{proposition}

\begin{figure}
	\centering
	\includegraphics[width=0.65\textwidth]{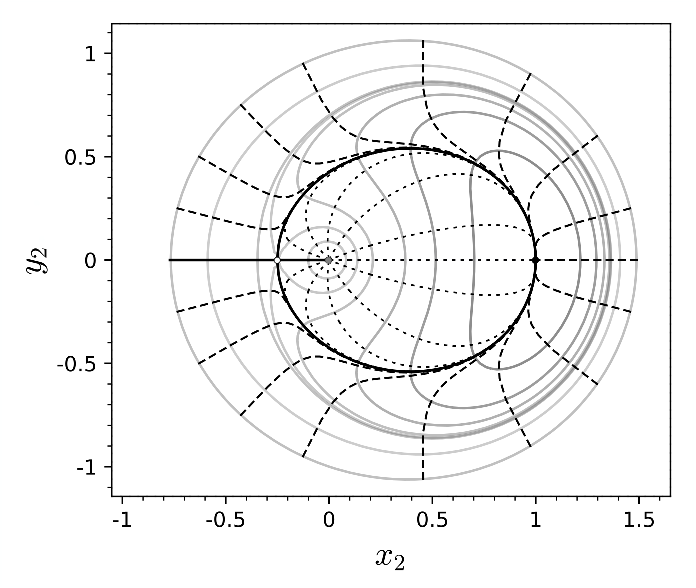}
	\caption{Projection of a $3-2-1$ foliation onto the plane $(x_2,y_2)$.}
\end{figure}

\begin{remark}\label{re:example-bifrucation}
	If $\epsilon<0$ and $|\epsilon|$ is sufficiently small, $S=H^{-1}\left(\frac{1}{2}\right)$ is strictly convex, which implies (see \cite{hwz1998}) that $\lambda=\lambda_0|_S$ is dynamically convex. 
	The simple Reeb orbit $P$ with image $H_1^{-1}\left(\frac{1}{2}\right)\times \{0\}$ has Conley-Zehnder index $3$, self-linking number $-1$, and is unknotted. 
	It follows from the main statement of \cite{hryniewicz2014systems} that $P$ is binding of an open book decomposition with disk-like pages adapted to $\lambda$, where each page is a global surface of section for the associated Reeb flow.
	The orbits $P_1$, $P_2$ and $P_3$ bifurcate from $P$ when $\epsilon=0$.
\end{remark}


\subsection{Outline of the main arguments}

The paper is organized as follows. 
In Section \ref{se:preliminaries}, we review some facts about contact geometry, Conley-Zehnder indices, and pseudo-holomorphic curves in symplectizations, which will be used throughout the paper.
The proof of 
Theorem \ref{theo:main-theorem} is split into Propositions \ref{pr:foliation-solid-torus}, \ref{pr:a-cylinder-asymptotic-p1-p2}, \ref{pr:a-family-of-cylinders}, and \ref{pr:self-linking-1}, proved in Sections \ref{se:foliating-gamma1=0}, \ref{se:cilindro-p2-p1}, and \ref{se:familia-cilindros-folheacao}, respectively. Proposition \ref{TH:NECESSIDADE-3} is proved in Section \ref{se:prova-necessidade}, Theorem \ref{pr:example} in Section \ref{se:proposition-pre-exemple}, and Proposition \ref{pr:exemplo-prova} in Section \ref{se:proof-prop-exemplo}. 

In the following, we sketch the main steps of the proof of Theorems \ref{theo:main-theorem} and \ref{pr:example}.
By hypothesis \textit{(iv)} of Theorem \ref{theo:main-theorem}, there exists an $\R$-invariant almost complex structure $\tj=(\lambda,J)$ admitting a finite energy plane 
asymptotic to $P_3$. 
By the results of \cite{hofer1995properties2} and the Fredholm theory developed in \cite{hofer1999properties3}, after a quotient by the natural $\R$-action, the plane lives in a one-dimensional family 
$\tu_\tau=(a_\tau,u_\tau):\C\to \R\times S^3,~~\tau\in (0,1)$,
where each projection $u_\tau:\C\to S^3$ is an embedding transverse to the Reeb vector field. 
  In Proposition \ref{pr:foliation-solid-torus}, using the linking hypotheses and bubbling-off analysis, we show that the family breaks in both ends into two-level holomorphic buildings, each consisting of a rigid cylinder from $P_3$ to $P_2$ and a rigid plane asymptotic to $P_2$. 
We use hypothesis \textit{(v)} and the intersection theory developed in \cite{siefring2011intersection} to show that the rigid plane to $P_2$ is common to the two ends. 
\begin{figure}
	\centering
	\includegraphics[scale=0.6]{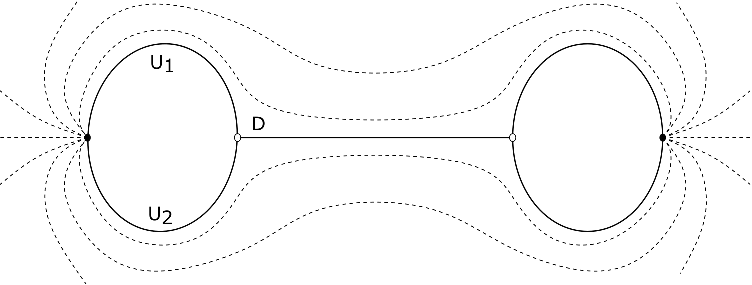}
	\caption{
	The pair of rigid cylinders $U_1=u_r(\C\setminus\{0\})$, $U_2=u_r'(\C\setminus \{0\})$, the rigid disk $D=u_q(\C)$ and the family of disks $F_\tau=u_\tau(\C)$, $\tau\in(0,1)$, foliating the region $\mathcal{R}_1$.}
\label{fi:solid-torus}
\end{figure}
The projection of the one-parameter family of planes $\tu_\tau$, the two rigid cylinders $\tu_r=(a_r,u_r),\tu'_r=(a'_r,u'_r):\C\setminus \{0\}\to \R\times S^3$, and the rigid plane $\tu_q=(a_q,u_q):\C\to \R\times S^3$ to $S^3$ determine a foliation of a closed region $\mathcal{R}_1$ homeomorphic to a solid torus. See figure \ref{fi:solid-torus}.

The next step of the proof is Proposition \ref{pr:a-cylinder-asymptotic-p1-p2}, where we obtain a rigid cylinder $\tv_r=(b_r,v_r):\C\setminus\{0\}\to \R\times S^3$ from $P_2$ to $P_1$.
We consider a non-cylindrical symplectic cobordism between $(S^3,\lambda)$ and $(S^3,\lambda_E)$, where $\lambda_E$ is a dynamically convex contact form, adapting the ideas of \cite{hwz1998}. 
We define an non $\R$-invariant almost complex structure $\bar{J}$ on the cobordism such that $\bar{J}=\tilde{J}$ on $[2,+\infty)\times S^3$.
After an $\R$-translation, the plane $\tu_q$ is $\bar{J}$-holomorphic. 
By the Fredholm theory of \cite{hofer1999properties3}, the plane lives in a one-dimensional family of $\bar{J}$-holomorphic planes asymptotic to $P_2$.
Using hypothesis \textit{(v)} and the SFT compactness theorem, we show that the family 
breaks into a two-level holomorphic building, where the first level is pseudo-holomorphic for the original $\R$-invariant almost complex structure. Using the linking hypotheses and intersection theory, we show that the first level is a cylinder from $P_2$ to $P_1$ and projects onto an embedded cylinder in the complement of the region $\mathcal{R}_1$.

\begin{figure}
	\centering
	\includegraphics[scale=0.65]{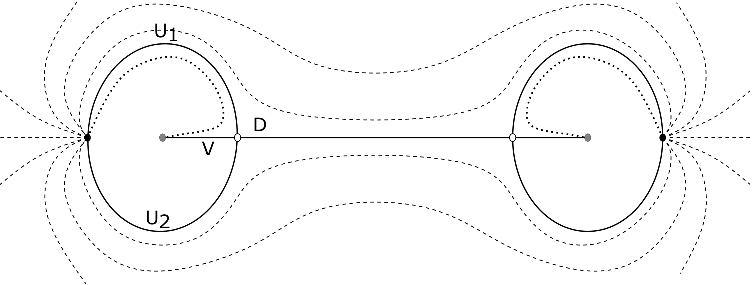}
	\caption{The foliation of $\mathcal{R}_1$, the cylinder $V=v_r(\C\setminus\{0\})$ connecting the orbits $P_2$ and $P_1$, and a cylinder $C_\tau=w_\tau(\C\setminus\{0\})$ connecting $P_3$ and $P_1$.}
\end{figure}

Next, we glue the cylinders $\tu_r$ and $\tv_r$ to obtain an one-dimensional family $\tw_\tau=(c_\tau,w_\tau):\C\setminus\{0\}\to \R\times S^3$ of $\tj$-holomorphic cylinders from $P_3$ to $P_1$. Using hypotheses \textit{(ii)} we prove that this family breaks into a two-level building formed by the cylinders $\tu'_r$ and $\tv_r$. In Proposition \ref{pr:a-family-of-cylinders} we prove that the cylinders $C_\tau=w_\tau(\C\setminus\{0\})$ complete the foliation. 


We show that the orbits $P_1$, $P_2$ and $P_3$ have self-linking number $-1$ in Proposition \ref{pr:self-linking-1}.  This follows for the orbits $P_2$ and $P_3$ since these orbits bound disks transverse to the Reeb flow. For the orbit $P_1$ we produce a disk gluing the cylinder $\bar{V}$ and the disk $\bar{D}$.

The existence of the $3-2-1$ foliation and the arguments in {\cite[Proposition 7.5]{hwz2003}} imply the existence of at least one homoclinic orbit to $P_2$. 
For completeness, we sketch a proof of the existence of a homoclinic here. Our argument follows \cite[\S 2]{dePS2013} and \cite[\S 4]{depaulo2019}.

Consider the one-parameter family of disks $\{F_\tau\}$, $\tau\in (0,1)$, the one-parameter family of cylinders $\{C_\tau\}$, $\tau\in (0,1)$, and the cylinders $U_1,U_2$ as in Definition \ref{de:3-2-1-foliation}. The local unstable manifold $W_{loc}^-(P_2)$ intersects $F_\tau$ transversally, for $\tau$ close to $0$, in an embedded circle 
bounding an embedded closed disk $B_{F,\tau,0}\subset F_\tau$ with $d\lambda$-area $T_2$. 
The local stable manifold $W^+_{loc}(P_2)$ intersects $F_{\tau'}$ transversally, for $\tau'$ close to $1$, in an embedded circle bounding an embedded closed disk $B^+_{\tau'}\subset F_{\tau'}$ with $d\lambda$-area $T_2$. All points in $F_{\tau'}\setminus B^+_{\tau'}$ correspond to trajectories that exit $\mathcal{R}_1$ through the cylinder $U_2$.
In the same way, $W^+_{loc}(P_2)$ intersects $C_{\tau'}$, for $\tau'$ close to $1$, in an embedded circle $S^+_{\tau'}$ such that $P_1\cup S^+_{\tau'}$ is the boundary of a closed region $R^+_{\tau'}$ with $d\lambda$-area $T_2-T_1>0$. All points in $C_{\tau'}\setminus R^+_{\tau'}$ correspond to trajectories entering $\mathcal{R}_1$ through the cylinder $U_1$.

The existence of a $3-2-1$ adapted to $\lambda$ implies that the forward flow sends $B_{F,\tau,0}$, into a disk $B_{F,1}$ inside $F_{\tau'}$, for $\tau'$ close to $1$. If $B_{F,1}$ intersects $B^+_{\tau'}$, then, since both disks have the same area, their boundaries also intersect, and a homoclinic to $P_2$ exists. Otherwise, $B_{F,1}\subset \left(F_{\tau'}\setminus B^+_{\tau'}\right)$ and the forward flow sends $B_{F,1}$ 
into a disk $B_{C,\tau,1}\subset C_\tau$, for $\tau$ close to $0$, with $d\lambda$-area $T_2$.
The forward flow sends $B_{C,\tau, 1}\subset C_{\tau}$ into a disk $B_{C,2}$ inside $C_{\tau'}$. If $B_{C,2}$ intersects $R_{\tau'}^+$, then, since the area of $R_{\tau'}^+$ is $T_2-T_1<T_2$, their boundaries also intersect and there exists a homoclinic to $P_2$. 
Otherwise, $B_{C,2}$ is contained in $C_{\tau'}\setminus R_{\tau'}^+$ and the forward flow sends $B_{C,2}$ into a 
disk $B_{F,\tau,2}\subset F_\tau$, which has area $T_2$ and is disjoint from $B_{F,\tau,0}$.
Proceeding this way, we construct  disks $B_{F,\tau,2n}\subset F_\tau$, $B_{C,\tau,2n+1}\subset C_\tau$, for $\tau$ close to $0$, and $B_{F,2n+1}\subset F_{\tau'}$, $B_{C,2n+2}\subset C_{\tau'}$, for $\tau'$ close to $1$. 
This procedure must end at some point since the disks $B_{F,\tau,2n}$ and  $B_{C,\tau,2n+1}$ are disjoint and have the same area $T_2$, while the areas of $F_\tau$ and $C_\tau$ are $T_3$ and $T_3-T_1$, respectively. 
This implies that either $B_{F,2n+1}$ intersects $B^+_{\tau'}$ or $B_{C,2n+2}$ intersects $R_{\tau'}^+$, for some integer $n\geq 0$. In any case, a homoclinic to $P_2$ must exist. 

To prove Theorem \ref{pr:example}, we use hypothesis \textit{(iii)} to obtain a one-dimensional family of pseudo-holomorphic cylinders from $P_3$ to $P_1$. Using bubbling-off analysis and hypotheses  \textit{(i)} and \textit{(ii)}, we show that this family breaks in both ends into a two-level holomorphic building consisting of a rigid cylinder from $P_3$ to $P_2$ and a rigid cylinder from $P_2$ to $P_1$.

\section{Preliminaries}\label{se:preliminaries}

In this section,  $\lambda$ is a contact form on a $3$-manifold $M$, and $\xi=\ker \lambda$ is the associated contact structure.

\subsection{The Conley-Zehnder index}\label{se:Conley-Zehnder-index}

Let $\varphi^t$ be the Reeb flow associated to the contact form $\lambda$.
The bilinear form $d\lambda$ turns $\xi=\ker \lambda$ into a symplectic vector bundle and the linearized flow 
$d\varphi^t_x:\xi_x\to \xi_{\varphi^t(x)}$ is symplectic with respect to $d\lambda$.

Let $P=(x,T)\in \mathcal{P}(\lambda)$ be a nondegenerate Reeb orbit. 
Let $\Psi:x_T^*\xi\to\R/\Z\times \R^2$ be a trivialization of the symplectic vector bundle $(x_T^*\xi,d\lambda)$ and consider the arc of symplectic matrices $\Phi\in \C^\infty([0,1],Sp(1))$ defined by 
$\Phi(t)=\Psi_t\circ d\varphi^{Tt}|_{\xi_{x(0)}}\circ \Psi_0^{-1}~.$ 
Let $z\in \C\setminus \{0\}$ and let $\theta:[0,1]\to \R$ be a continuous argument for $z(t):=\Phi(t)z$, that is, $e^{2\pi\theta(t)}=\frac{z(t)}{|z(t)|},~\forall ~0\leq t\leq 1$.
Define the \textit{winding number} of $z(t)=\Phi(t)z$ by 
\begin{equation}\label{eq:winding-cz}
	\Delta(z)=\theta(1)-\theta(0)\in \R
\end{equation}
and the \textit{winding interval} of the arc $\Phi$ by 
$I(\Phi)=\{\Delta(z)|z\in \C\setminus \{0\}\}~.$
The interval $I(\Phi)$ is compact and its length is strictly smaller than $\frac{1}{2}$. Since $P$ is nondegenerate, we have $\partial I(\Phi)\cap \Z=\emptyset$, see {\cite[\S 2]{hryniewicz2014systems}}. 
{Thus, the winding interval either lies between two consecutive integers or contains precisely one integer. 
The Conley-Zehnder index of the orbit $P$ relative to the trivialization $\Psi$ is defined by
\begin{equation}\label{eq:defi-cz-index-nondeg}
	\mu(P,\Psi)=\left\{
	\begin{array}{lr}
		2k+1, \text{ if }I(\Phi)\subset (k,k+1)\\
		2k, \text{ if } k\in I(\Phi)
	\end{array}
	\right.
\end{equation}
This index only depends on the homotopy class
of the trivialization $\Psi$.

Throughout the paper, we only deal with the tight contact structure on $S^3$, which is a trivial symplectic bundle. 
In this case, the Conley-Zehnder index is independent of the choice of global symplectic trivialization of $\xi$. 
We define the Conley-Zehnder index of the Reeb orbit $P$ by 
\begin{equation}
	\mu(P):=\mu(P,\Psi),
\end{equation}
for any global symplectic trivialization $\Psi:\xi\to S^3\times \R^2$.


	We say that 
	the Reeb orbit $P=(x,T)$ is  \textit{positive hyperbolic} if $\sigma(\Phi(1))\subset (0,\infty)\setminus \{1\}$, \textit{negative hyperbolic} if $\sigma(\Phi(1))\subset (-\infty,0)\setminus \{-1\}$, or \textit{elliptic} if the eigenvalues are in $S^1\setminus \{1\}$. 
	An orbit is positive hyperbolic if and only if it has even Conley-Zehnder index. 


The following lemma is a consequence of the properties of the Conley-Zehnder index proved in \cite{hwz2003}.

\begin{lemma}\label{le:properties-cz-index}
	Let $\lambda$ be a tight contact form on $S^3$ and let $P=(x,T)$ be a nondegenerate Reeb orbit. Let $k\geq 1$ be an integer such that for every $l\in \{1,\cdots, k\}$, the orbit $P^l=(x,lT)$ is nondegenerate. The following assertions hold.
	\begin{itemize}
		\item [(1)]If $\mu(P^k)=1$, then $\mu(P^l)=1$, $\forall l\in \{1,\cdots, k\}$;
		\item [(2)] $\mu(P^k)\leq 0$ $\iff$ $\mu(P^l)\leq 0$, $\forall l\in \{1,\cdots, k\}$; 
		\item [(3)] If $P$ is a hyperbolic orbit, then $\mu(P^l)=l\mu(P)$, $\forall l\in \{1,\cdots, k\}$;
		\item[(4)]If $\mu(P^k)=2$, then $k\in \{1,2\}$ and $P$ is hyperbolic. If $k=2$, then $\mu(P)=1$.
	\end{itemize}
\end{lemma}

\subsection{The asymptotic operator}\label{ch:asymptotic operator}


Fix $P=(x,T)\in \mathcal{P}(\lambda)$  and let $h:S^1\to TM$ be a vector field along $x_T:S^1=\R/\Z\to M$. 
The Lie derivative $\mathcal{L}_{R_\lambda}h$ of $h$ with respect to $R_\lambda$ is defined by
\begin{equation}\label{eq:lei-derivative-section}
\mathcal{L}_{R_\lambda}h(t)=\dfrac{d}{ds}\bigg|_{s=0}d\varphi^{-s}(x(Tt+s))h\left(t+\frac{s}{T}\right)~,
\end{equation}
where $\varphi^t$ is the flow of $R_\lambda$.
Let $\nabla$ be a symmetric connection on $TM$. We can use $dx_T(t)\partial _t=TR_\lambda(x_T(t))$ to write
\begin{equation}\label{eq:asymptotic-operator-independent}
T\mathcal{L}_{R_\lambda}h=\mathcal{L}_{TR_\lambda}h=\nabla_{TR_\lambda}h-\nabla_hTR_\lambda=\nabla_th-T\nabla_hR_\lambda,
\end{equation}
where 
$\nabla_t$ is the covariant derivative along $x_T$.
We conclude that the differential operator $\nabla_t\cdot - T \nabla_{\cdot}R_\lambda$ maps sections of $x_T^*\xi$ to sections of $x_T^*\xi$ and is independent of the choice of symmetric connection.


Choosing a $d\lambda$-compatible complex structure $J\in \mathcal{J}(\lambda)$, we associate to the orbit $P=(x,T)$ the unbounded differential operator
\begin{equation}\label{eq:operador-assintotico-extensao}
\begin{aligned}
A_{P,J}:\mathcal{D}(A_{P,J})=W^{1,2}(S^1,x^*_T\xi)\subset L^2(S^1,x_T^*\xi)&\to L^2(S^1,x_T^*\xi)\\
\eta&\mapsto -J(\nabla_t\eta - T\nabla_{\eta}R_\lambda)~.
\end{aligned}
\end{equation}
	The operator $A_{P,J}$ defined by \eqref{eq:operador-assintotico-extensao}
	is called the \textit{asymptotic operator} associated to the orbit $P$ and the complex structure $J$.
	In any unitary  trivialization $\Psi$ of $(x_T^*\xi,d\lambda,J)$, the operator $A_{P,J}$ takes the form 
	$L_S:=-J_0\frac{d}{dt}-S(t),$ 
	where $S(t)$ is a path of symmetric matrices given by $S(t)=-J_0\dot{\Phi}(t)\Phi(t)^{-1}$ and $\Phi(t)$ is the linearized flow restricted to $\xi$ in the trivialization $\Psi$.
If $\eta(t)$ is an eigensection of $A_{P,J}$ with corresponding eigenvalue $\lambda\in\R$,
then $n(t):=\Psi\circ \eta(t)$ 
satisfies $n(t)\neq 0$ for all $t\in S^1$.
It follows that $\eta(t)$ has a well defined winding number given by 
\begin{equation}\label{rm:winding-number def}
\wind(\eta,\Psi):=\deg \left(t\mapsto\frac{n(t)}{\|n(t)\|}\right).
\end{equation}
This definition just depends on the homotopy class of the trivialization $\Psi$. 
The following properties about this winding number are proved in \cite{hofer1995properties2}.
\begin{proposition}\cite{hofer1995properties2}\label{pr:properties-asymptotic-operator}
	The unbounded operator $A_{P,J}$ has discrete real spectrum accumulating only at $\pm \infty$. 
	Let $\Psi$ be a unitary trivialization of $x_T^*\xi$. Then
	\begin{itemize}
		\item Given nonzero eigensections $\eta_1(t)$ and $\eta_2(t)$ associated to the same eigenvalue $\lambda\in \sigma(A_{P,J})$, we have 
		$\wind(\eta_1,\Psi)=\wind(\eta_2,\Psi)$, 
		so that we can define $\wind(\lambda,\Psi)=\wind(\eta,\Psi)$, for any eigensection $\eta$ associated to $\lambda$. 
		\item If $\lambda\neq \mu\in \sigma(A_{P,J})$ satisfy $\wind(\lambda,\Psi)=\wind(\mu,\Psi)$ and $\eta_\lambda(t), \eta_\mu(t)$ are non-vanishing $\lambda, \mu$-eigensections, respectively, then $\eta_\lambda(t),\eta_\mu(t)$ are pointwise linearly independent.
		\item Given $k\in \Z$, there exists precisely two eigenvalues $\lambda, \mu\in \sigma(A_{P,J})$, counting multiplicities, such that $\wind(\lambda,\Psi)=\wind(\mu,\Psi)=k$
		\item If  $\lambda, \mu\in \sigma(A_{P,J})$ and $\lambda\leq \mu$, then $\wind(\lambda,\Psi)\leq \wind(\mu,\Psi)$.
		\item $0\notin\sigma(A_{P,J})$ if and only if the orbit $P=(x,T)$ is nondegenerate.
	\end{itemize}
\end{proposition}

Define
\begin{equation}\label{eq:mair-autovalor-neg}
\nu^{neg}_P=\max\{\nu<0|\nu \text{ is an eigenvalue of }A_P\}
\end{equation}
\begin{equation}\label{eq:menos-autovalor-pos}
\nu^{pos}_P=\min\{\nu\geq0|\nu \text{ is an eigenvalue of }A_P\}
\end{equation}
and given a trivialization $\Psi$ of $x_T^*\xi$, define
$p=\wind(\nu^{pos}_P,\Psi)-\wind(\nu^{neg}_P,\Psi)\in \{0,1\}.$ 
\begin{definition}
	The (generalized) Conley-Zehnder index of the orbit $P$ relative to the unitary trivialization $\Psi$ is defined by
	\begin{equation}\label{eq:conley-zehnder-wind}
	\tilde{\mu}(P,\Psi)=2\wind(\nu^{neg}_P,\Psi)+p.
	\end{equation}
\end{definition}
It is proved in {\cite[Theorem 3.10]{hofer1995properties2}} that for any nondegenerate orbit $P\in \mathcal{P}(\lambda)$,
$\tilde{\mu}(P,\Psi)=\mu(P,\Psi)~,$
where $\mu(P,\Psi)$ is the Conley-Zehnder index defined by \eqref{eq:defi-cz-index-nondeg}.

\subsection{Self-linking number}

The \textit{self-linking number} $\sli(L)$ of a trivial knot $L\subset M$ transverse to $\xi$ is defined as follows. Consider $M$ oriented by $\lambda\wedge d\lambda$.  
Let $D\subset M$ be an embedded disk satisfying $\partial D=L$
and let $Z$ be a smooth nonvanishing section of $\xi|_D$. The section $Z$ is used to slightly perturb $L$ to another knot $L_\epsilon =\{\expe_x(\epsilon Z_x)|x\in L\}$ transverse to $\xi$ and $D$, where $\expe$ is any exponential map. A choice of orientation for $L$ induces orientations of $D$ and $L_\epsilon$.
The self-linking number  of $L$ is defined by
	$L_\epsilon\cdot D\in \Z,$ 
where $L_\epsilon\cdot D$ is the oriented intersection number of $L_\epsilon$ and $D$.
If, for instance, $\xi$ is trivial, this definition is independent of the choices of  $Z$ and $D$. 
If $P=(x,T)$ is an unknotted Reeb orbit, we define its self-linking number by $\sli(P)=\sli(x(\R))$. 


\subsection[Finite energy surfaces]{Finite energy surfaces}\label{ch:finite-energy-surfaces}

The \textit{symplectization} of the contact manifold $(M,\lambda)$ is the symplectic manifold $(\R\times S^3, d(e^a\lambda))$, where $a$ is the coordinate on $\R$. 
Given a complex structure $J\in \mathcal{J}(\lambda)$, we consider the almost-complex structure $\tj=(\lambda,J)$ on $\R\times M$ defined by
\begin{equation}\label{eq:J-til}
\tilde{J} \partial_a =R_\lambda,~~~\tilde{J}|_\xi=J,
\end{equation}
where we see $R_\lambda$ and $\xi$ as $\R$-invariant objects on $\R\times M$.
It is easy to check that the almost complex structure $\tilde{J}$ defined by \eqref{eq:J-til} is $d(e^a\lambda)$-compatible.

Let $(S,j)$ be a closed Riemann surface and let $\Gamma\subset S$ be a finite set.
Let $\tu:S\setminus \Gamma\to \R\times M$ be a $\tj$-holomorphic map, that is, $\tu$ is smooth and satisfies the Cauchy-Riemann equation 
$$\bar{\partial}_{\tj}(\tu)=\frac{1}{2}\left(\tu+\tj(\tu)\circ d\tu\circ j\right)=0.$$
The \textit{Hofer energy} $E(\tu)$ of $\tu$ is defined by
$$E(\tu)=\sup_{\phi\in \Sigma}\int_{S\setminus \Gamma}\tilde{u}^*d(\phi\lambda)~,$$
where $\Sigma=\{\phi\in C^\infty(\R,[0,1])|\phi'\geq 0)\}$.

\begin{definition}
	The $\tj$-holomorphic map $\tilde{u}:S\setminus \Gamma\to \R\times M$ is called a \textit{finite energy surface} if it satisfies 
	$0<E(\tilde{u})<+\infty$.
\end{definition}
The elements of $\Gamma$ are called \textit{punctures}.
 Let $z\in \Gamma$ be a puncture and take a holomorphic chart $\varphi:(U,0)\to (\varphi(U),z)$ centered at $z$. We call $(s,t)\simeq \varphi(e^{-2\pi(s+it)})$ \textit{positive exponential coordinates} and $(s,t)\simeq \varphi(e^{2\pi(s+it)})$  \textit{negative exponential coordinates} around $z$. 
Set $\tu(s,t)=\tu\circ \varphi(e^{-2\pi(s+it)})$, for $s>>1$.
Write $\tu=(a,u)$.
Using Stokes Theorem, one can prove that the limit
\begin{equation}\label{eq:mass}
m(z)=\lim_{s\to +\infty} \int_{\{s\}\times S^1}u^*\lambda 
\end{equation}
exists.
The puncture $z$ is called \textit{removable} if $m=0$, \textit{positive} if $m>0$ and \textit{negative} if $m<0$. By an application of Gromov's removable singularity theorem \cite{gromov85}, one can prove that $\tu$ can be smoothly extended to a removable puncture.
Thus, in the following we assume that all punctures are positive or negative and use the notation $\Gamma=\Gamma^+\cup\Gamma^-$  to distinguish positive and negative punctures. 
If $\tu:S\setminus \Gamma \to \R\times M$ is a finite energy surface, then $\Gamma^+\neq \emptyset$. 

\begin{theorem}\cite{hofer1993pseudoholomorphic,hofer1996properties1}\label{th:asymptotic-limit}
	Let $\tu=(a,u):S\setminus \Gamma\to \R\times M$ be a finite energy surface. Assume $z$ is non removable and let $\epsilon=\pm 1$ be the sign $z$. Fix a sequence $s_n\to +\infty$. Then there exists a nonconstant trajectory of the Reeb flow $x:\R\to M$ with period $T>0$ and a subsequence $s_{n_k}$ 
	such that 
	$$\lim_{k\to +\infty}\{t\mapsto u(s_{n_k},t)\}=\{t\mapsto x(\epsilon Tt)\}$$
	in the $C^\infty(S^1,M)$ topology. 
	If the orbit $(x,T)$ is nondegenerate, then
	$$\lim_{s\to +\infty}\{t\mapsto u(s,t)\}=\{t\mapsto x(\epsilon Tt)\}~.$$
\end{theorem}
We say that the periodic orbit $P=(x,T)$ given by Theorem \ref{th:asymptotic-limit} is \textit{an asymptotic limit} of $\tu$ at the puncture $z$. 
If $P=(x,T)$ is nondegenerate, it is called \textit{the asymptotic limit} of $\tu$ at $z$.


The \textit{$d\lambda$-area} of  a $\tilde{J}$-holomorphic curve $\tu:S\setminus \Gamma \to \R\times M$ is given by the formula
$$A(\tu)=\int_{S\setminus\Gamma}\tu^*d\lambda=\int_{S\setminus \Gamma}u^*d\lambda~.$$
One can check that $A(\tu)\geq 0$ and $A(\tu)=0$ if and only if $\pi\cdot du\equiv0$.
{The following theorem concerning $\tj$-holomorphic curves with vanishing $d\lambda$-area 
will be useful throughout the paper.}
\begin{theorem}{\cite[Theorem 6.11]{hofer1995properties2}}\label{theo:vanishing-dlambda-energy}
	Let $\tu=(a,u):\C\setminus \Gamma\to \R\times M$ be a finite energy sphere, where $\Gamma\subset \C$ is the finite set of negative punctures and $\infty$ is the unique positive puncture. If $\pi\cdot du\equiv 0$, then there exists a nonconstant polynomial $p:\C\to \C$ and a Reeb orbit $P=(x,T)\in \mathcal{P}(\lambda)$ such that 
	$p^{-1}(0)=\Gamma$  and $\tu=F_P\circ p$ , 
	where $F_P:\C\setminus \{0\}\to\R\times M$ is defined by $F_P(z=e^{2\pi(s+it)})=(Ts,x(Tt))$.
\end{theorem}


\subsection{Asymptotic behavior}

	A \textit{Martinet's tube} for a simple orbit $P=(x,T)\in \mathcal{P}(\lambda)$ is a pair $(U,\psi)$, where $U$ is a neighborhood of $x(\R)$ in $M$ and $\psi:U\to S^1\times B$ is a diffeomorphism (here $B\subset \R^2$ is an open ball centered at the origin) satisfying
	\begin{itemize}
		\item There exists $f:S^1\times B\to \R^+$ such that $f|_{S^1\times \{0\}}\equiv T$, $df|_{S^1\times\{0\}}\equiv 0$ and $\psi^*(f(d\theta+x_1dx_2))=\lambda$, where $\theta$ is the coordinate on $S^1$ and $(x_1,x_2)$ are coordinates on $\R^2$;
		\item $\psi(x_T(t))=(t,0,0)$.
	\end{itemize}
	The coordinates $(\theta,x_1,x_2)$ are called \textit{Martinet's coordinates}.
The existence of such Martinet's tubes is proved in \cite{hofer1996properties1} for any simple orbit $P\in \mathcal{P}(\lambda)$.


A more precise description of the asymptotic behavior of a finite energy surface is given in \cite{hofer1996properties1}. 
Let $\tu=(a,u):S\setminus \Gamma\to \R\times M$ be a finite energy surface asymptotic to a nondegenerate orbit $P=(x,T)$ at the positive puncture $z_0\in \Gamma$ and let  $(s,t)$ be positive exponential coordinates near $z_0$.
Let $k$ be a positive integer such that $T=kT_{min}$, where $T_{min}$ is the least positive period of $x$ and let $(\theta,z)=(\theta,x_1,x_2)$ be Martinet's coordinates in a neighborhood $U\subset M$ of $P_{min}=(x,T_{min})$. 
Then there exists $s_0\in \R^+$ such that, for $(s,t)\in [s_0,+\infty)\times \R$, the map $u(s,t)$ can be represented in Martinet's coordinates by 
$$ u(s,t)=(\theta(s,t),z(s,t))\in \R\times \R^2,$$
where $\theta$ is seen as a map on the universal cover $\R$ of $S^1=\R/\Z$ satisfying $\theta(s,t+1)=\theta(s,t)+k$ and $z$ is $1$-periodic in $t$.
\begin{theorem}[\cite{hofer1996properties1}]\label{theo:asymptotic behavior} 
	If $A(\tu)>0$, there are constants $A_{ij}\in \R^+$, $a_0,\theta_0\in\R$, a function $R:\R\times S^1\to \R^2$ and 
	an eigensection $\eta(t)$ of the asymptotic operator $A_{P,J}$ \eqref{eq:operador-assintotico-extensao}, associated to a negative eigenvalue $\alpha\in \sigma(A_{P,J})$, such that
	\begin{equation}
	\begin{aligned}
	|\partial_s^i\partial_t^j(a(s,t)-(Ts+a_0))|\leq A_{ij}e^{-r_0s}\\
	|\partial_s^i\partial_t^j(\theta(s,t)-(kt+\theta_0)|\leq A_{ij}e^{-r_0s}\\
	z(s,t)=e^{\int_{s_0}^s\alpha(r)dr}(e(t)+R(s,t))\\
	|\partial_s^i\partial_t^jR(s,t)|,|\partial_s^i\partial_t^j(\alpha(s)-\alpha)|\leq A_{ij}e^{-r_0s}
	\end{aligned}
	\end{equation}
	for all large $s$ and $i,j\in \N$. Here 
	$e:S^1\to \R^2$ represents the eigensection $\eta(t)$ in the coordinates induced by $\psi$, and $\alpha:[s_0,\infty)\to \R$ is a smooth function such that $\alpha(s)\to \alpha, \text{as } s\to \infty$.
	
	A similar statement holds if $z_0$ is a negative puncture. In this case, we use negative exponential coordinates near $z_0$, $e^{-r_0s}$ is replaced by $e^{r_0s}$ and the eigenvalue $\alpha$ of $A_{P,J}$ is positive.
\end{theorem}
The eigenvalue $\alpha$ and the eigensection $\eta(t)$, as in Theorem \ref{theo:asymptotic behavior}, will be referred to as the \textit{asymptotic eigenvalue} and \textit{asymptotic eigensection} of $\tu$ at the puncture $z_0$.

\begin{remark} \label{rm:strong-transverse-section}
	Let $\tu=(a,u):S\setminus \Gamma\to \R\times S^3$ be a finite energy surface that is asymptotic to nondegenerate simple orbits at all of its punctures and such that $\pi\cdot du$ is not identically zero. By Theorem \ref{theo:asymptotic behavior}, the surface $\Sigma=\overline{u(S\setminus \Gamma)}$ is a strong transverse section according to Definition \ref{de:strong-transverse-section}. 
	Indeed, for every puncture $z\in \Gamma$, the asymptotic eigensection $\eta(t)$ of $\tu$ at $z$ satisfies
		$$d\lambda(\eta(t),\mathcal{L}_{R_\lambda}\eta(t))=\frac{1}{T}d\lambda(\eta(t),JA_{P,J}\eta(t))=\frac{1}{T}\alpha d\lambda(\eta(t),J\eta(t))\neq 0, ~\forall t\in \R/\Z.$$
\end{remark}

\subsection{Algebraic invariants}
Let $\tu=(a,u):S\setminus \Gamma\to \R\times M$ be a $\tilde{J}$-holomorphic finite energy surface. Assume that $\pi\cdot du$ is not identically zero and that 
$\tu$ has nondegenerate asymptotic limits at all of its punctures. 
In \cite{hofer1995properties2}, it is proved that the set where $\pi\cdot du$ vanishes is finite, and it is defined a local degree associated to each zero of $\pi\cdot du$, which is always positive. 
The integer 
\begin{equation}
\wind_\pi(\tu)\geq 0
\end{equation}
is defined as the sum of such local degrees over all zeros of $\pi\cdot du$.


Consider a unitary trivialization $\Psi:(u^*\xi,d\lambda, J)\to (S\setminus\Gamma)\times \R^2$. 
For $z\in \Gamma$, fix positive cylindrical coordinates $(s,t)$ at $z$ and define
\begin{equation}\label{eq:definicao-wind-infty}
\wind_\infty(\tu,z,\Psi)=\lim_{s\to +\infty}\wind \left(t\mapsto \pi\cdot \partial_s u(s,\epsilon_z t),\Psi|_{u(s,\epsilon_z\cdot)^*\xi}\right)\in \Z
\end{equation}
where $\epsilon_z$ is the sign of the puncture $z$.
The winding number on the right is defined as in \eqref{rm:winding-number def}, and is independent of the choice of $J$ and of the holomorphic chart.
This limit is well defined since $\pi \cdot \partial_su(s,t)$ does not vanish for $s$ sufficiently large.
The asymptotic winding number of $\tu$ is defined by 
\begin{equation}\label{eq:definition-wind-infty}
\wind_\infty(\tu)=\sum_{z\in \Gamma^+}\wind_\infty(\tu,z,\Psi)-\sum_{z\in \Gamma^-}\wind_\infty(\tu,z,\Psi)~.
\end{equation}
It is proved in \cite{hofer1995properties2} that this sum does not depend on the chosen trivialization $\Psi$.  

%

\begin{remark}\label{rm:trivializacao-global-fixada}
	Throughout the paper we will only deal with the tight contact structure on $S^3$, which is a trivial symplectic bundle. 
	Consider a global symplectic trivialization
	$\Psi:\xi\to S^3\times \R^2~.$
	Then $\wind_\infty(\tu,z,\Psi)$ is the winding number of the asymptotic eigensection given by Theorem \ref{theo:asymptotic behavior}, with respect to the trivialization $\Psi$. 	
	Moreover, $\wind_\infty(\tu,z,\Psi)$ does not depend on the chosen global symplectic trivialization
	and we denote $\wind_\infty(\tu,z,\Psi)$ by $\wind_\infty(\tu,z)$.
\end{remark}
It is also proved in \cite{hofer1995properties2} that the invariants $\wind_\pi(\tu)$ and $\wind_\infty(\tu)$ satisfy 
\begin{equation}\label{eq:wind_pi-wind_infty}
	\wind_\pi(\tu)=\wind_\infty(\tu)-\chi(S)+\#\Gamma,
\end{equation}
where $\chi(S)$ is the Euler characteristic of $S$.

The following lemma is a consequence of Proposition \ref{pr:properties-asymptotic-operator} and Remark \ref{rm:trivializacao-global-fixada}. 
\begin{lemma}\label{le:wind-infty-estimate}
	Let $\tu:S\setminus \Gamma\to \R\times M$ be a finite energy surface and let $z\in \Gamma$ be a puncture with asymptotic limit $P=(x,T)$. Assume that $P$ is nondegenerate and define $\nu^{pos}_P$ and $\nu^{neg}_P$ by \eqref{eq:mair-autovalor-neg} and \eqref{eq:menos-autovalor-pos} respectively. If $\pi\cdot du$ does not vanish identically, then
	\begin{itemize}
		\item[(1)] $\wind_\infty(\tu,z)\leq \wind(\nu^{neg}_P)$ if $z$ is a positive puncture.
		\item[(2)] $\wind_\infty(\tu,z)\geq \wind(\nu^{pos}_P)$ if $z$ is a negative puncture.
	\end{itemize}	
\end{lemma}

%


\subsection{Fredholm theory}
	Let $\tu=(a,u):S\setminus \Gamma\to \R\times M$ be a $\tilde{J}$-holomorphic finite energy surface and assume that $\tu$ has nondegenerate asymptotic limits at all of its punctures.
	Let $\Psi:(u^*\xi,d\lambda)\to (S\setminus \Gamma)\times \R^2$ be a symplectic trivialization.
	Fix a puncture $z\in \Gamma$ and let $P_z=(x,T)$ be the asymptotic limit of $\tu$ at $z$.
	The trivialization $\Psi$ induces a homotopy class of oriented trivializations $[\Psi_z]$ of $x_T^*\xi$. 
		The Conley-Zehnder index of $\tu$ is defined by
		\begin{equation}
		\mu(\tilde{u})=\sum_{z\in \Gamma^+}\mu(P_z,\Psi_z)-\sum_{z\in \Gamma^-}\mu(P_z,\Psi_z)~.
		\end{equation}
		It is proved in \cite{hofer1995properties2} that this sum does not depend on the chosen trivialization $\Psi$. 

	An \textit{unparametrized} finite energy surface is an equivalence class $[(\tu,(S,j),\Gamma)]$, where $\tu:(S\setminus \Gamma,j)\to (\R\times M,\tj)$ is a $\tj$-holomorphic finite energy surface and $\Gamma$ is an ordered set. 
	The equivalence class is defined as follows: $(\tu',(S',j'),\Gamma')$ is equivalent to $(\tu,(S,j),\Gamma)$ if there exists a biholomorphism $\phi:(S,j)\to (S',j')$ such that $\tu'=\tu\circ \phi$, $\phi(\Gamma)=\Gamma'$ and $\phi$ preserves the ordering of $\Gamma$ and $\Gamma'$.
	To shorten notation we usually denote an unparametrized surface by $[\tu]$. 
		
	In 
	\cite{dragnev2004fredholm} it is proved that the set of unparametrized finite energy surfaces in the neighborhood
	of $[(\tu,(S,j),\Gamma)]$, where $\tu$ is a somewhere injective finite energy surface, 
	is described by a nonlinear Fredholm
	equation having Fredholm index equal to
	\begin{equation}\label{eq:formula-indice-fredholm}
	\ind(\tilde{u}):=\mu(\tilde{u})-\chi(S)+\#\Gamma.
	\end{equation}
	This generalizes the result for embedded finite energy surfaces proved in \cite{hofer1999properties3}.

	\begin{theorem}\label{theo:fredholm-estimate}\cite{dragnev2004fredholm}
		There exists a dense subset $\mathcal{J}_{reg}\subset \{\tj|J\in \mathcal{J}(\lambda)\}$ such that, if $\tilde{u}=(a,u):S\setminus \Gamma\to \R\times M$ is a somewhere injective finite energy surface which is pseudo-holomorphic with respect to  ${\tj}\in \mathcal{J}_{reg}$ and has nondegenerate asymptotic limits at all of its punctures, then 
		$$0\leq \ind(\tilde{u})=\mu(\tilde{u})-\chi(S)+\#\Gamma~.$$
		If $\pi\circ du$ is not identically zero, then
		$$1\leq \ind(\tilde{u})=\mu(\tilde{u})-\chi(S)+\#\Gamma~.$$
	\end{theorem}	
	

\subsection[Intersection Theory]{Intersection theory of punctured pseudo-holomorphic curves}

In this section we state some results concerning the intersection theory of punctured pseudo-holomorphic curves from \cite{siefring2011intersection}. 
A nice exposition of these results is given in \cite{fish2018connected}. 
Here we use the same convention as \cite{fish2018connected} for computing Conley-Zehnder indices and the asymptotic winding number $\wind_\infty$, which differs from that of \cite{siefring2011intersection}. 
}
We assume that all the finite energy curves considered have nondegenerate asymptotic limits. 

In the following, $\tu=(a,u):(S\setminus \Gamma,j)\to \R\times M$ and $\tv=(b,v):(S'\setminus \Gamma',j')\to \R\times M$ are finite energy $\tj$-holomorphic curves 
and  $\Psi$ denotes a choice of trivialization of the contact structure along all simple periodic orbits with covers appearing as asymptotic limits of $\tu$ or $\tv$. 
\begin{theorem}{\cite[Theorem 2.4]{siefring2011intersection}}\label{th:2.4-siefring}
	Assume that no component of $\tu$ or $\tv$ lies in an orbit cylinder and that the projected curves $u$ and $v$ do not have identical image on any component of their domains. Then the following are equivalent:
	\begin{itemize}
		\item[(1)] The projected curves $u$ and $v$ do not intersect;
		\item[(2)] All of the following hold:
		\begin{itemize}
			\item The map $u$ does not intersect any of the positive asymptotic limits of $\tv$;
			\item The map $v$ does not intersect any of the negative asymptotic limits of $\tu$;
			\item If $P$ is a periodic orbit so that, at $z\in\Gamma$, $\tu$ is asymptotic $P^{m_z}$ and, at $w\in \Gamma'$, $\tv$ is asymptotic to $P^{m_w}$, then: If $z$ and $w$ are both positive or both negative punctures, we have 
			$\frac{\wind_\infty(\tu,z,\Psi)}{m_z}\geq \frac{\wind_\infty(\tv,w,\Psi)}{m_w}~.$
			If $z$ is a negative puncture and $w$ is a positive puncture, we have
			$$\frac{\wind_\infty(\tu,z,\Psi)}{m_z}=\frac{\lfloor\mu(P^{m_z},\Psi)/2\rfloor}{m_z}=\frac{-\lfloor-\mu(P^{m_w},\Psi)/2\rfloor}{m_w}=\frac{\wind_\infty(\tv,w,\Psi)}{m_w}$$
		\end{itemize} 
		\item [(3)]All of the following hold:
		\begin{itemize}
			\item The map $u$ does not intersect any of the asymptotic limits of $\tv$.
			\item The map $v$ does not intersect any of the asymptotic limits of $\tu$.
			\item If $P$ is a periodic orbit so that, at $z\in \Gamma$, $\tu$ is asymptotic to $P^{m_z}$ and, at $w\in \Gamma'$, $\tv$ is asymptotic to $P^{m_w}$, then $\frac{\wind_\infty(\tu,z,\Psi)}{m_z}=\frac{\wind_\infty(\tv,w,\Psi)}{m_w}$.
		\end{itemize}
	\end{itemize}  
\end{theorem}

A set of necessary and sufficient conditions for the projection of a curve to be embedded are given in Theorem 2.6 of \cite{siefring2011intersection}.
The following statement is a direct consequence of {\cite[Theorem 2.6]{siefring2011intersection}}  that is enough for our purposes. 

\begin{theorem}{\cite{siefring2011intersection}}\label{th:siefring-2.6}
	Assume that $\tu$ is somewhere injective, connected 
	and does not have image contained in an orbit cylinder. 
	Assume further that $\tu$ is not asymptotic to a covering of the same simple orbit at two distinct punctures. 
	Then the following are equivalent:
	\begin{enumerate}
		\item The projected map $u:S\setminus \Gamma\to M$ is an embedding.
		\item $u$ does not intersect any of its asymptotic limits.
	\item All of the following hold:
	\begin{itemize}
		\item The map $\tu$ is an embedding.
		\item The projected map $u$ is an immersion which is everywhere transverse to the Reeb vector field.
		\item For each $z\in \Gamma$, we have  $\gcd(m_z,\wind_\infty(\tu,z,\Psi))=1$. 
	\end{itemize} 
\end{enumerate}
\end{theorem}

Next we recall the notion of two pseudo-holomorphic curves approaching an orbit in the same (or opposite) direction. 
The definition given here follows \cite{fish2018connected}. It applies to any nondegenerate orbit and is stricter than the definition from \cite{siefring2011intersection}.

\begin{definition}[\cite{siefring2011intersection}, \cite{fish2018connected}]\label{de:opposite-directions}
	Assume that $\tu$ and $\tv$ are asymptotic to the same Reeb orbit $P$ at certain punctures $z\in \Gamma$ and $w\in\Gamma'$ with the same sign. We say that  $\tu$ and $\tv$ approach $P$ in the \textit{same direction} at these punctures if $\eta_u=c\eta_v$ for some $c>0$ , where $\eta_u$ and $\eta_v$ are the asymptotic eigensections of $\tu$ at $z$  and of $\tv$ at $w$ respectively. In case $\eta_u=c\eta_v$ for $c<0$, we say that $\tu$ and $\tv$ approach $P$ in \textit{opposite directions}.
\end{definition}

The following theorem of \cite{fish2018connected} is a consequence of results on \cite{siefring2011intersection} concerning the intersection properties of pseudo-holomorphic curves approaching an even orbit in the same direction. 

\begin{theorem}{\cite[Theorem 3.16]{fish2018connected}}\label{th:2.5-siefring}
	Assume that $\tu$ and $\tv$ are asymptotic to an even Reeb orbit $P$ at punctures $z\in\Gamma$ and $w\in \Gamma'$ with the same sign. Assume that $\tu$ and $\tv$ have extremal winding number, that is, 
	$$\wind_\infty(\tu,z,\Psi)=\wind_\infty(\tv,w,\Psi)=\dfrac{\mu(P,\Psi)}{2}$$
	and that $\tu$ and $\tv$ approach $P$ in the same direction. Then the projections $u$ and $v$ intersect. 
\end{theorem}

\section[Foliating a solid torus]{Foliating a solid torus}\label{se:foliating-gamma1=0}

In this section, we start the proof of Theorem \ref{theo:main-theorem} with the following statement.
\begin{proposition}\label{pr:foliation-solid-torus}
	Let $\lambda$ be a tight contact form on $S^3$ satisfying the hypotheses of Theorem \ref{theo:main-theorem}.
	Then there exists a family of $\tj$-holomorphic planes $\{\tu_\tau=(a_\tau,u_\tau):\C\to \R\times S^3\}_{\tau\in (0,1)}$, all of them asymptotic to $P_3$, a pair of $\tj$-holomorphic cylinders $\tu_r=(a_r,u_r),\tu_r'=(a_r',v_r'):\C\setminus \{0\}\to \R\times S^3$, both asymptotic to $P_3$ at their positive punctures $z=+\infty$ and $P_2$ at their negative punctures $z=0$, and a finite energy plane $\tu_q=(a_q,u_q):\C\to \R\times S^3$ asymptotic to $P_2$. The projections $u_\tau, u_r, u_r', u_q$ are embeddings transverse to the Reeb vector field. 
	Define $F_\tau=u_\tau(\C)$, $U_1=u_r(\C\setminus \{0\})$, $U_2=u_r'(\C\setminus \{0\})$ and $D=u_q(\C)$. Then the surface $T:=P_2\cup P_3\cup U_1\cup U_2$ is homeomorphic to a torus and $T\setminus P_3$  is $C^1$-embedded. The union of the family $\{F_\tau\}_{\tau\in(0,1)}$ with $U_1$, $U_2$ and $D$ determine a smooth foliation of $\mathcal{R}_1\setminus (P_2\cup P_3)$, where $\mathcal{R}_1\subset S^3$ is a closed region  with boundary $T$ and homeomorphic to a solid torus. 
\end{proposition}	
By hypothesis,  
there exists a finite energy $\tj$-holomorphic plane
\begin{equation}\label{eq:plano-util}
	\tu=(a,u):\C\to \R\times S^3
\end{equation}
asymptotic to the orbit $P_3$. 
Both the $\tj$-holomorphic plane $\tu$ and the projection $u:\C\to S^3$ are embeddings.
This is a consequence of the following result, which is a particular case of Theorem 1.3 of \cite{hofer1995properties2}.
\begin{theorem}[{\cite{hofer1995properties2}}]\label{th:theorem-1.3-hwz-prop2}
	Consider $S^3$ equipped with a tight contact form $\lambda$. 
	Assume $\tu=(a,u):\C\to \R\times S^3$ is a finite energy plane asymptotic to an unknotted orbit $P$. 
	If $\mu(P)\leq 3$, then $u(\C)\cap P=\emptyset$ and $u:\C\to S^3\setminus P$ is an embedding.
\end{theorem}

	The finite energy plane $\tu$ is automatically Fredholm regular. If $J'\in \mathcal{J}_{reg}$ is a $C^\infty$-small perturbation of $J$, where $\mathcal{J}_{reg}$ is the dense set obtained by Theorem \ref{theo:fredholm-estimate}, then we can find a $\tilde{J}'$-holomorphic plane $\tu'$ asymptotic to $P_3$ as a $C^\infty$ small perturbation of $\tu$. It follows that we can revert the notation back to $\tu$ and $\tj$ and assume that $\tj\in \mathcal{J}_{reg}$. This assumption will be necessary since we will use Theorem \ref{theo:fredholm-estimate} in the proof of Theorem \ref{theo:main-theorem}.

\subsection{A family of planes asymptotic to $\boldsymbol{P_3}$}\label{se:a-family-of-planes}
{The following is a particular case of Theorem 4.5.44 of \cite{wendl2005thesis}. 
It generalizes Theorem 1.5 of \cite{hofer1999properties3}.
}
\begin{theorem}[\cite{wendl2005thesis}]\label{th:family-of-cylinders-wendl}
	Let $\lambda$ be a contact form on a closed $3$-manifold $M$ and let $J\in\mathcal{J}(\xi=\ker\lambda,d\lambda)$.
	Let $\tu=(a,u):S^2\setminus \Gamma\to \R\times M$ is an embedded $\tilde{J}$-holomorphic finite energy sphere such that  every asymptotic limit is nondegenerate, simple and has odd Conley-Zehnder index. Suppose that  $\ind(\tu)=2$. Then there exists a number $\delta>0$ and an embedding 
	\begin{equation*}
	\begin{aligned}
	\tilde{F}:\R\times (-\delta,\delta)\times S^2\setminus \Gamma&\to \R\times M\\
	(\sigma,\tau,z)&\mapsto (a_\tau(z)+\sigma,u_\tau(z))
	\end{aligned}
	\end{equation*}
	such that
	\begin{itemize}
		\item For $\sigma\in \R$ and $\tau\in (-\delta,\delta)$, the maps $\tu_{(\sigma,\tau)}:=\tilde{F}(\sigma,\tau,\cdot)$ are (up to parametrization) embedded $\tilde{J}$-holomorphic finite energy spheres and $\tu_{(0,0)}=\tu$.
		\item The map $F(\tau,z)=u_\tau(z)$ is an embedding $(-\delta, \delta)\times S^2\setminus \Gamma\to M$ and its image never intersects the asymptotic limits.
		In particular, the maps $u_\tau:S^2\setminus \Gamma\to M$ are embeddings for each $\tau\in (-\delta, \delta)$, with mutually disjoint  images which do not intersect their asymptotic limits.
		\item  For any sequence $\tv_k:S^2\setminus \Gamma\to \R\times M$ such that for each puncture in $\Gamma$, $\tv_k$ has the same asymptotic limit as $\tu$, with the same sign, and $\tv_k\to \tu$ in $C^\infty_{loc}(S^2\setminus \Gamma)$, there is a sequence $(\sigma_k,\tau_k)\to (0,0)\in \R\times (-\delta,\delta)$ such that $\tv_k=\tu_{(\sigma_k,\tau_k)}\circ \varphi_k$ for some sequence of biholomorphisms  $\varphi_k:S^2\to S^2$ and $k$ sufficiently large.
	\end{itemize}
\end{theorem}
\begin{remark}
	In Theorem \ref{th:family-of-cylinders-wendl}, we do not require $\tj\in\mathcal{J}_{reg}$.
\end{remark}

%

Applying Theorem \ref{th:family-of-cylinders-wendl} to the finite energy plane \eqref{eq:plano-util} 
we obtain a maximal one-parameter family of finite energy planes 
\begin{equation}\label{eq:familia-de-planos}
	\tilde{u}_\tau=(a_\tau, u_\tau):\C\to \R\times S^3,~~\tau\in (\tau_-,\tau_+)
\end{equation}
asymptotic to the orbit $P_3$.
The family \eqref{eq:familia-de-planos} satisfies 
	$u_{\tau_1}(\C)\cap u_{\tau_2}(\C)=\emptyset, ~\forall \tau_1\neq\tau_2.$ 
Indeed, suppose that there exist $\tau_1\neq \tau_2$ such that $u_{\tau_1}(\C)\cap u_{\tau_2}(\C)\neq \emptyset$. Then $u_{\tau_1}(\C)= u_{\tau_2}(\C)$ (see Theorem \ref{th:2.4-siefring} or {\cite[Theorem 1.4]{hofer1995properties2}}) and we would obtain an $S^1$-family of embedded planes that provides an open book decomposition
of $S^3$, with binding $P_3$ and disk-like pages.
It follows from equation \eqref{eq:wind_pi-wind_infty}, Lemma \ref{le:wind-infty-estimate} and formula \eqref{eq:conley-zehnder-wind} that  
$\wind_\pi(\tu_\tau)=0, ~\forall \tau\in (\tau_-,\tau_+).$
As a consequence of the definition of $\wind_\pi$, we conclude that 
$u_\tau$ is transverse to $R_\lambda$, $\forall \tau\in (\tau_-,\tau_+)$. 
This implies that every Reeb orbit in $S^3\setminus P_3$ is linked to $P_3$, which contradicts the existence of $P_1$ and $P_2$. 

Now we describe how the family $\{\tu_\tau\}$ breaks as $\tau \to \tau_{\pm}$. We assume that $\tau_-=0$ and $\tau_+=1$ and that $\tau$ strictly increases in the direction of $R_\lambda$.


\begin{proposition}\label{theo:bubbling-off-two-vertices}
	Consider a sequence $\tilde{u}_n:=\tilde{u}_{\tau_n}$ in the family \eqref{eq:familia-de-planos} satisfying $\tau_n\to 0^+$. 
	Then there exists a $\tilde{J}$-holomorphic finite energy cylinder $\tu_r:\C\setminus \{0\}\to \R\times S^3$, which is asymptotic to $P_3$ at the positive puncture $z=\infty$ and to $P_2$ at the negative puncture $z=0$, and a finite energy $\tilde{J}$-holomorphic plane $\tu_q:\C\to\R\times S^3$ asymptotic to $P_2$ at $z=\infty$, such that,	after suitable reparametrizations and $\R$-translations of $\tu_n$, the following hold
	\begin{enumerate}[label=(\roman*)]
		\item up to a subsequence, 
		$\tu_n\to \tu_r$ in $C^\infty_{loc}(\C\setminus \{0\})$ as $n\to \infty$.
		\item There exist sequences $\delta_n\to 0^+$, $z_n\in \C$ and $c_n\in \R$ such that, up to a subsequence,
		$\tu_n(z_n+\delta_n \cdot)+c_n\to \tu_q$ 	in $C^\infty_{loc}(\C)$ as $n\to \infty$.
	\end{enumerate} 
 Here $(a,x)+c:=(a+c,x),~\forall (a,x)\in \R\times S^3,~c\in \R$.
{A similar statement holds for any sequence $\tu_{\tau_n}$ satisfying $\tau_n\to 1^-$. In this case we change the notation from $\tu_r$ and $\tu_q$ to $\tu_r'$ and $\tu_q'$ respectively. }
\end{proposition}
The proof of proposition \ref{theo:bubbling-off-two-vertices} is left to Subsection \ref{se:proof-bubbling-planos}. Subsection \ref{se:bubbling-planos} below consists of preliminary results. 

\subsection{Bubbling-off analysis}\label{se:bubbling-planos}

By hypotheses, every orbit in $\mathcal{P}(\lambda)$ having period $\leq T_3$ is nondegenerate.
Consequently, there exists only a finite number of such orbits. 
Define 
$\sigma(T_3)$ as any real number satisfying
\begin{equation}\label{eq:definicao-sigmac}
	0<\sigma(T_3)<\min \{T, |T-T'|: T\neq T'\text{ periods }, T,T' \leq T_3\}~.
\end{equation}

Most of the material in \S\ref{se:germinating-seuqences} and \S\ref{se:soft-rescaling} is adapted from \cite{hls2015}.

\subsubsection{Germinating sequences}\label{se:germinating-seuqences}
%
Now  we fix  $J\in \mathcal{J}(\lambda)$ and consider a sequence of $\tilde{J}$-holomorphic curves
$\tv_n=(b_n, v_n):B_{R_n}(0)\subset \C\to \R\times S^3$
satisfying
\begin{align}\label{eq:geminating-sequence-def1}
	R_n\to \infty ,~~ R_n\in (0,+\infty]\\
	E(\tv_n)\leq T_3,~\forall n \label{eq:geminating-sequence-def2}\\
	\int_{B_{R_n}(0)\setminus\D}v_n^*d\lambda\leq \sigma(T_3),~\forall n \label{eq:germinating-sequence-def3}\\
	\{b_n(2)\} \text{ is uniformily bounded}\label{eq:germinating-sequence-def4}.
\end{align}

\begin{definition}\label{de:germinating-sequence}
	A sequence $\tv_n$ of $\tj$-holomorphic curves satisfying \eqref{eq:geminating-sequence-def1}-\eqref{eq:germinating-sequence-def4} is called a \textit{germinating sequence}.
\end{definition}

\begin{proposition}\label{pr:limit-germinating-sequence}
	There exists a finite set $\Gamma\subset \D$, a $\tj$-holomorphic map $\tv=(b,v):\C\setminus \Gamma\to \R\times S^3$ and a subsequence of $\tv_n$, still denoted by $\tv_n$, such that 
	$\tv_n\to \tv~~\text{in } C^\infty_{loc}(\C\setminus \Gamma,\R\times S^3)$
	and $E(\tv)\leq T_3$.
\end{proposition}

\begin{proof}
	Let $\Gamma_0\subset \C$ be the set of points $z\in \C$ such that there exists a subsequence $\tv_{n_j}$ and a sequence $z_j\in B_{R_{n_j}}(0)$ with $z_j\to z$ and
	$|d\tv_{n_j}(z_j)|\to \infty,~j\to \infty~.$ 
	If $\Gamma_0=\emptyset$, then by \eqref{eq:germinating-sequence-def4} and usual elliptic estimates, see {\cite[Chapter 4]{mcduff2004j}}, we find a $\tj$-holomorphic map $\tv:\C\to \R\times S^3$ such  that, up to a subsequence, $\tv_n\to \tv$ in $C^\infty_{loc}(\C,\R\times S^3)$. In this case, $\Gamma=\emptyset$.
	
	Now assume $\Gamma_0\neq\emptyset$ and let $z_0\in \Gamma_0$. It follows from results in {\cite[\S3.2]{hofer1993pseudoholomorphic}} that 
	there exists a period $0<T_0\leq T_3$ and sequences $r_j^0\to 0^+$, $n_j^0\to \infty$ and $z_j^0\to z_0$ such that 
	$$\lim_{j\to \infty} \int_{B_{r_j}(z_j^0)} v_{n_j^0}^*d\lambda\geq T_0~.$$

	Consider $\tv_{n^0_j}$ as the new sequence $\tv_n$. 
	Now let $\Gamma_1\subset \C\setminus \{z_0\}$ be the set of points $z_1\neq z_0$ such that there exists a subsequence $\tv_{n_j}$ and sequence $z_j\in B_{R_{n_j}}(0)$ with  
	$z \to z_1$ and $|d\tv_{n_j}(z_j)|\to \infty$.
	As before, if $\Gamma_1=\emptyset$, we have a $\tj$-holomorphic map $\tv:\C\setminus \{z_0\}\to \R\times S^3$ such that, up to a subsequence, $\tv_n\to \tv$ in $C^\infty_{loc}(\C\setminus \{z_0\},\R\times S^3)$. 
	In this case, we define $\Gamma=\Gamma_0=\{z_0\}$. 
	If $\Gamma_1\neq \emptyset$ and $z_1\in \Gamma_1$, there exist a period $0<T_1\leq T_3$ and sequences $r_j^1\to 0$, $n_j^1\to \infty$ and $z_j^1\to z_1$ such that
	$$\lim_{j\to \infty} \int_{B_{r_j}(z_j^1)} v_{n^1_j}^*d\lambda\geq T_1~.$$ 
	Considering $\tv_{n_j^1}$ as the new sequence $\tv_n$, define $\Gamma_2\subset \C\setminus\{z_0, z_1\}$ as before.
	Repeating this argument, let $z_i\in \Gamma_i\subset \C\setminus\{z_0,\dots, z_{i-1}\}$.
	Note that 
	\begin{equation*}
		\begin{aligned}
			C&\geq \lim_{n\to \infty}E(\tv_n)\geq \lim_{n\to \infty}\int_{\C}v_n^*d\lambda
			&\geq \sum_{l=0}^{i}\lim_{j\to \infty} \int_{B_{r_j^l}(z_j^l)}v_{n_j^l}^*d\lambda\geq T_0+\dots+T_i.	
		\end{aligned}
	\end{equation*}
	It follows that there exists $i_0$ such that $\Gamma_{i_0}\neq0$ and $\Gamma_i=\emptyset$ for $i>i_0$.
	We end up with a finite set $\Gamma=\{z_0,\dots,z_{i_0}\}$ and a $\tj$-holomorphic map 
	$\tv:\C\setminus \Gamma\to \R\times S^3$ 
	such that, up to a subsequence,
	$\tv_n\to \tv \text{ in }C^\infty_{loc}~.$ 
	 
	It follows from \eqref{eq:germinating-sequence-def3} that $\Gamma\subset \D$.
	The inequality $E(\tv)\leq T_3$ follows from \eqref{eq:geminating-sequence-def2} and Fatou's Lemma.
\end{proof}

\begin{definition}\label{de:limit-germinating-sequence}
	A $\tj$-holomorphic map $\tv:\C\setminus\Gamma\to \R\times S^3$ as in Proposition \ref{pr:limit-germinating-sequence} is called a limit of the germinating sequence $\tv_n$.
\end{definition}

If $\Gamma\neq \emptyset$, then $\tv$ is nonconstant. In this case, all the punctures $z=z_i\in \Gamma$ are negative and $\infty$ is a positive puncture. 
To prove this, define, for any $\epsilon>0$,
\begin{equation}\label{eq:def-mepisilonz}
	m_\epsilon(z):=\int_{\partial B_\epsilon(z)}v^*\lambda~,
\end{equation}
where $\partial B_\epsilon(z)$ is oriented counterclockwise. Then
$$m_\epsilon(z)=\int_{\partial B_\epsilon(z)} v^*\lambda=\lim_{n\to \infty}\int_{\partial B_\epsilon(z)}v_n^*\lambda=\lim_{n\to \infty}\int_{B_\epsilon(z)}v_n^*d\lambda~.$$
For $j$ large, $B_{r_j^i}(z_j^i)$, defined as in the proof of  Proposition \ref{pr:limit-germinating-sequence}, is contained in $B_\epsilon(z)$. 
It follows that 
$$m_\epsilon(z)=\lim_{n\to \infty}\int_{B_\epsilon(z)}v_n^*d\lambda \geq \lim_{j\to \infty}\int_{B_{r_j}(z_j^i)} v_{n_j^i}^*d\lambda\geq T_i>0~.$$
This implies that $\tv$ is nonconstant and the puncture $z$ is negative.
Moreover, as a consequence of $0<E(\tv)<\infty$, we know that $\tv$ has at least one positive puncture. Thus, $\infty$ is  a positive puncture. 

\subsubsection{Soft-rescaling near a negative puncture}\label{se:soft-rescaling}
Assume $\Gamma\neq \emptyset$ and let $\tv=(b,v):\C\setminus \Gamma\to \R \times M$ be  a limit of the germinating sequence $\tv_n=(b_n,v_n)$. 
Let $z\in \Gamma$. 
We define the mass $m(z)$ of $z$ by
\begin{equation}\label{eq:int-maior-sigma-c}
	m(z)=\lim_{\epsilon\to 0^+} m_\epsilon(z)=\lim_{\epsilon\to 0^+}\int_{\partial B_\epsilon(z)} v^*\lambda =T_z>\sigma(T_3)>0~, 
\end{equation}
where $T_z$ is the period of the asymptotic limit of $\tv$ at $z$.
Since $m_\epsilon(z)$ is a nondecreasing function of $\epsilon$, we can fix $\epsilon$ small enough so that 
\begin{equation}\label{eq:m-mepsilon}
	0\leq m_\epsilon(z)-m(z)\leq\frac{\sigma(T_3)}{2}.
\end{equation}
Choose sequences $z_n\in \overline{B_\epsilon(z)}$ and $0<\delta_n<\epsilon,~\forall n$, so that 
\begin{align}
	b_n(z_n)\leq b_n(\zeta),~\forall \zeta \in B_\epsilon(z) \label{eq:min-bn},\\
	\int_{B_\epsilon(z)\setminus B_{\delta_n}(z_n)} v_n^*d\lambda =\sigma(T_3) \label{eq:int-sigma-c}.
\end{align}
{Since $z$ is a negative puncture, \eqref{eq:min-bn} implies that $z_n\to z$. Hence the existence of $\delta_n$ as in \eqref{eq:int-sigma-c} follows from \eqref{eq:int-maior-sigma-c}.}
We claim that $\liminf \delta_n=0$. Otherwise, we choose $0<\epsilon'<\liminf \delta_n\leq \epsilon$. 
From \eqref{eq:m-mepsilon}, we get the contradiction 
\begin{equation*}
	\begin{aligned}
		\frac{\sigma(T_3)}{2}&\geq m_\epsilon(z) -m(z)\geq m_\epsilon(z)-m_{\epsilon'}(z)\\
		& =\lim_{n\to \infty} \int_{B_\epsilon(z)\setminus B_{\epsilon'}(z)}v_n^* d\lambda \\
		&\geq\lim_{n\to \infty}\int_{B_\epsilon(z)\setminus B_{\delta_n}(z_n)}v_n^*d\lambda=\sigma(T_3)~.
	\end{aligned}
\end{equation*}
Thus, we can assume that $\delta_n\to 0$.

Now take any sequence $R_n\to +\infty$ satisfying
$\delta_nR_n<\frac{\epsilon}{2}~$
and define the sequence of $\tj$-holomorphic maps $\tilde{w}_n=(c_n,w_n):B_{R_n}(0)\to \R\times S^3$ by
\begin{equation}\label{eq:definition-wn}
	\tilde{w}_n(\zeta)=(b_n(z_n+\delta_n\zeta)-b_n(z_n+2\delta_n), v_n(z_n+\delta_n\zeta))~.
\end{equation}
It follows from \eqref{eq:int-sigma-c} that
$$ \int_{B_{R_n}(0)\setminus \D}w_n^*d\lambda\leq \sigma(T_3),~\forall n~.$$
Moreover, by the definition of $\tilde{w}_n$, $E(\tilde{w}_n)\leq E(\tv_n)\leq T_3$ and $\tilde{w}_n(2)\in \{0\}\times M$. Thus, $\tilde{w}_n$ is a germinating sequence.

Let $\tilde{w}=(c,w):\C\setminus \Gamma'\to \R\times S^3$ be a limit of $\tilde{w}_n$, as in Proposition \ref{pr:limit-germinating-sequence}. If $\Gamma'\neq \emptyset$, then $\tilde{w}$ is not constant. 
If $\Gamma'=\emptyset$, then
\begin{equation}\label{eq:int-disco-w}
	\begin{aligned}
		\int_\D w^*d\lambda&=\lim_{n\to\infty}\int_\D w_n^*d\lambda= \lim_{n\to \infty}\int_{B_{\delta_n}(z_n)}v_n^*d\lambda\\
		&=\lim_{n\to \infty}\left(\int_{B_\epsilon(z)}v_n^*d\lambda-\int_{B_\epsilon(z)\setminus B_{\delta_n}(z_n)}v_n^*d\lambda\right)\\
		&=m_\epsilon(z)-\sigma(T_3)\\
		&\geq T_z-\sigma(T_3)>0~.
	\end{aligned}
\end{equation}
Thus $\tilde{w}$ is nonconstant as well. 
From Fatou's Lemma we get $0<E(\tilde{w})\leq T_3$.
This also implies that the periods of the asymptotic limits of $\tilde{w}$ are bounded by $T_3$.  


\begin{proposition}\label{pr:same-orbits-tree}
	The asymptotic limit $P_\infty$ of $\tilde{w}$ at $\infty$ coincides with the asymptotic limit $P_z$ of $\tilde{v}$ at the negative puncture $z\in \Gamma$.
\end{proposition}

To prove Proposition \ref{pr:same-orbits-tree}, we need the following lemma, which is a restatement of Lemma 4.9 from \cite{hwz2003}.

\begin{lemma}{\cite{hwz2003}}\label{le:cylinders-with-small-area}
	Consider a constant $e>0$ and let $\sigma(T_3)$ be defined by \eqref{eq:definicao-sigmac}.
	Identifying $S^1=\R/\Z$, let $\mathcal{W}\subset C^\infty(S^1,S^3)$ be an open neighborhood of the set of periodic orbits $P=(x,T)\in \mathcal{P}(\lambda)$ with $T\leq T_3$, viewed as maps $x_T:S^1\to S^3$, $x_T(t)=x(Tt)$. We assume that $\mathcal{W}$ is $S^1$-invariant, meaning that $y(\cdot+c)\in \mathcal{W}\Leftrightarrow y\in \mathcal{W}, \forall c\in S^1$, and that each of the connected components of $\mathcal{W}$ contains at most one periodic orbit modulo $S^1$-reparametrizations.  	
	Then there exists a constant $h>0$ such that the following holds. 
	If $\tu=(a,u):[r,R]\times S^1 \to \R\times S^3$ is a $\tj$-holomorphic cylinder satisfying 
	$$E(\tu)\leq T_3,~~~\int_{[r,R]\times S^1}u^*d\lambda\leq \sigma(T_3),~~~~\int_{\{r\}\times S^1}u^*\lambda\geq e~~~~\text{and}~~ r+h\leq R-h,$$
	then each loop $t\in S^1\to u(s,t)$ is contained in $\mathcal{W}$ for all $s\in [r+h,R-h]$.
\end{lemma}

\begin{proof}[Proof of Proposition \ref{pr:same-orbits-tree}]
	Let  $\mathcal{W}\subset C^\infty(S^1,S^3)$ be 
	as in the statement of Lemma \ref{le:cylinders-with-small-area}. 
	Let $\mathcal{W}_\infty$ and $\mathcal{W}_z$ be connected components of $\mathcal{W}$ containing $P_\infty$ and $P_z$ respectively.
	Since $\tv_n \to \tv$, we can choose $0<\epsilon_0<\epsilon$ small enough so that, if $0<\rho\leq \epsilon_0$ is fixed, then the loop
	$$t\in S^1\mapsto v_n(z_n+\rho e^{i2\pi t})$$
	belongs to $\mathcal{W}_z$ for large $n$.
	Since $\tw_n\to \tw$, we can choose $R_0>1$ large enough so that, if $R\geq R_0$ is fixed, then  the loop
	$$t\in S^1\mapsto w_n(Re^{i2\pi t})=v_n(z_n+\delta_nRe^{i2\pi t})$$
	belongs to $\mathcal{W}_\infty$ for large $n$.
	By \eqref{eq:int-maior-sigma-c} and \eqref{eq:int-sigma-c}, we can show that
	\begin{equation}\label{eq:def-e-cylinders}
		e:=\liminf \int_{\partial B_{\delta_n R_0(z_n)}}v_n^*\lambda>0.
	\end{equation}
	Consider, for each $n$, the $\tj$-holomorphic cylinder $\tilde{C}_n:\left[\frac{\ln R_o\delta_n}{2\pi},\frac{\ln \epsilon_0}{2\pi}\right]\times S^1\to \R\times S^3$, defined by $\tilde{C}_n(s,t)=\tv_n(z_n+e^{2\pi(s+it)})$.
	It follows from \eqref{eq:int-sigma-c} that 
	\begin{equation}\label{eq:leq-sigma-cylinders}
		\int_{\left[\frac{\ln R_o\delta_n}{2\pi},\frac{\ln \epsilon_0}{2\pi}\right]\times S^1}C_n^*d\lambda\leq \sigma(T_3)
	\end{equation}
	for large $n$.
	Using \eqref{eq:def-e-cylinders} and \eqref{eq:leq-sigma-cylinders} and applying Lemma \ref{le:cylinders-with-small-area}, we find $h>0$ so that the loop
	$$t\mapsto C_n(s,t)$$
	is contained in $\mathcal{W}$ for all $s\in \left[\frac{\ln R_0\delta_n}{2\pi}+h,\frac{\ln \epsilon_0}{2\pi}-h\right]$ and large $n$.
	But 
	$$C_n\left(\frac{\ln \epsilon_0}{2\pi}-h,t\right)=v_n(z_n+\epsilon_0 e^{-2\pi h}e^{2\pi it})\in \mathcal{W}_z,~\text{ for all }n\text{ large, and}$$
	$$C_n\left(\frac{\ln R_0\delta_n}{2\pi}+h,t\right)=v_n(z_n+R_0\delta_ne^{2\pi h}e^{2\pi t})\in \mathcal{W}_\infty,~\text{for all } n \text{ large}.$$ 
	Thus, $\mathcal{W}_\infty=\mathcal{W}_z$ and $P_\infty=P_z$.
\end{proof}

\begin{proposition}\label{rmk:either-or}
	Either  
	\begin{itemize}
		\item $\int_{\C\setminus \Gamma'}w^*d\lambda >0$ or
		\item $\int_{\C\setminus \Gamma'}w^*d\lambda =0$ and $\#\Gamma' \geq 2$.
	\end{itemize}	
\end{proposition}

\begin{proof}
	If $\Gamma'\neq \emptyset$, then $0\in \Gamma'$. This fact follows from $c_n(0)=\inf c_n(B_{R_n})$ and the fact that the punctures in $\Gamma'$ are negative.
	Arguing by contradiction, assume
	$$\int_{\C\setminus \Gamma'}w^*d\lambda=0 \text{ and } \#\Gamma'=1~.$$
	Thus, $\Gamma'=\{0\}$. 
	Using Theorem \ref{theo:vanishing-dlambda-energy}, we conclude that 
	$w(\zeta=e^{2\pi(s+it)})=x_z(T_zt)$, 
	where $P_z=(x_z,T_z)$ is the asymptotic limit of $\tv$ at $z$. 
	Thus, we have the contradiction
	\begin{equation*}
		\begin{aligned}
			m(z)&=T_z=\int_{\partial \D}w^*\lambda=\lim_{n\to \infty}\int_{\partial\D}w_n^*\lambda=\lim_{n\to \infty}\int_{\partial B_\epsilon(z)}v_n^*\lambda-\sigma(T_3)\\
			&=\int_{\partial B_\epsilon(z)}v^*\lambda-\sigma(T_3) \leq m(z)-\frac{\sigma(T_3)}{2}.
		\end{aligned}
	\end{equation*} 
	Here we have used \eqref{eq:m-mepsilon}, \eqref{eq:int-sigma-c} and \eqref{eq:int-disco-w}.
\end{proof}

\subsubsection{Some index estimates}
\begin{lemma}\label{le:wind=1,0}
	Let $\tu=(a,u):\C\setminus\Gamma \to \R\times S^3$ be a finite energy surface such that every puncture in $\Gamma$ is negative, $\pi\circ du$ does not vanish identically and for every asymptotic limit $P$ of $\tu$, $P$ is nondegenerate and $\mu(P)\in \{1,2,3\}$. Then $\wind_\pi(\tu)=0$ and for all $z\in \Gamma\cup\{\infty\}$, $\wind_\infty(\tu,z)=1$.
\end{lemma}
\begin{proof}
	We follow \cite[Prop 5.9]{hwz2003}.
	By equation \eqref{eq:conley-zehnder-wind}, for every asymptotic limit $P$ of $\tu$, one of the following options hold
	\begin{equation}
		\begin{aligned}
			\mu(P)&=1~\Rightarrow\wind(\nu^{neg}_P)=0,~\wind(\nu_P^{pos})=1\\
			\mu(P)&=2~\Rightarrow\wind(\nu^{neg}_P)=1,~\wind(\nu_P^{pos})=1\\
			\mu(P)&=3~\Rightarrow\wind(\nu^{neg}_P)=1,~\wind(\nu_P^{pos})=2.
		\end{aligned}
	\end{equation}
	Then, by Lemma \ref{le:wind-infty-estimate}, if $P$ is the asymptotic limit of $\tu$ at $z$ we have
	\begin{equation}\label{eq:estimativa-wind-lema}
		\begin{aligned}
			\wind_\infty(\tu,z)&\leq \wind(\nu^{neg}_P)\leq 1, ~\text{if }z=\infty,\\
			\wind_\infty(\tu,z)&\geq \wind(\nu^{pos}_P))\geq 1,~\text{if }z\in \Gamma.
		\end{aligned}
	\end{equation}
	It follows that $\wind_\infty(\tu)\leq 1-\#\Gamma$. Using equation \eqref{eq:wind_pi-wind_infty}, we obtain
	$$0\leq \wind_\pi(\tu)= \wind_\infty(\tu)-1+\#\Gamma\leq0~.$$
	Thus $\wind_\pi=0$ and $\wind_\infty(\tu)=1-\#\Gamma$.
	Using \eqref{eq:estimativa-wind-lema}, we conclude
	$$0\geq \wind_\infty(\tu,\infty)-1=\sum_{z\in \Gamma}\wind_\infty(\tu,z)-\#\Gamma=\sum_{z\in \Gamma}\wind_\infty(\tu,z)-1\geq 0~.$$
	We conclude that $\wind_\infty(\tu,z)=1$, for all $z\in \Gamma\cup \{\infty\}$.
\end{proof}

The following estimate is proved in \cite{hs2011}.

\begin{lemma}{\cite[Lemma 3.9]{hs2011}}\label{le:indice-maior-igual-a-2}
	Let $\tv=(b,v)$ be a nonconstant limit of a sequence $\{\tv_n=(b_n,v_n)\}$ satisfying \eqref{eq:geminating-sequence-def1}-\eqref{eq:germinating-sequence-def4}.
	Then the asymptotic limit $P_\infty$ at the (unique) positive puncture of $\tv$ satisfies $\mu(P_\infty)\geq 2$.	
\end{lemma}

\subsection{Proof of Proposition \ref{theo:bubbling-off-two-vertices}}\label{se:proof-bubbling-planos}
	After a reparametrization of $\tu_n$ and an translation in the $\R$-direction, we can assume that $\{\tu_n\}$ is a germinating sequence as defined in \eqref{eq:geminating-sequence-def1}-\eqref{eq:germinating-sequence-def4}.
	Indeed, we can take  sequences $z_n\in \C$ and  $0<\delta_n$, such that $a_n(z_n)=\inf a_n(\C)$ and $\int_{\C\setminus B_{\delta_n(z_n)}}u_n^*d\lambda=\sigma(T_3)$, and define
	$$\tv_n(z)=(b_n(z),v_n(z))=(a_n(z_n+\delta_nz)-a_n(z_n+2\delta_n),u_n(z_n+\delta_nz))~.$$
	Then $b_n(2)=0, \forall n$, and 
	\begin{equation}\label{eq:u_n-e-seq-germinante}
	\int_{\C\setminus \D}v_n^*d\lambda=\int_{\C\setminus B_{\delta_n}(z_n)}u_n^*d\lambda=\sigma(T_3).
	\end{equation}

	Let $\tu_r=(a_r,u_r):	\C\setminus \Gamma_r\to \R\times S^3$ be a limit of $\tu_n$ as defined in \ref{de:limit-germinating-sequence}. 
	We claim that $P_3$ is the asymptotic limit of $\tu_r$ at $\infty$. 
	Indeed, let $\mathcal{W}\subset C^\infty(S^1,S^3)$ be as in the statement of Lemma \ref{le:cylinders-with-small-area}.
	Using the normalization condition \eqref{eq:germinating-sequence-def3}, we can apply Lemma \ref{le:cylinders-with-small-area} and find $R_0>1$ such that for $R\geq R_0$, the loops $t\mapsto u_r(Re^{i2\pi t})$ and $\{t\mapsto u_n(Re^{i2\pi t})\}, n\in \N$ belong to $\mathcal{W}$. 
	For fixed $R$, the sequence of loops $t\mapsto u_n(Re^{i2\pi t})$ converges to $t\mapsto u_r(Re^{i2\pi t})$ in $C^\infty(S^1,S^3)$, so that for large $n$ and $R\ge R_0$, $t\mapsto u_r(Re^{i2\pi t})$ and $t\mapsto u_n(Re^{i2\pi t})$ belong to the same connected component of $\mathcal{W}$.
	This implies that $P_3$ is the asymptotic limit of $\tu_r$ at $\infty$.
	
	Note that $\#\Gamma_r\neq 0$ and $0\in \Gamma_r$. 
	Indeed, if $\Gamma_r=\emptyset$, then $\tu_r:\C\to \R\times S^3$ would satisfy the hypotheses of Theorem \ref{th:family-of-cylinders-wendl}, which contradicts the fact that the family \eqref{eq:familia-de-planos} is maximal.
	Since $a_n(0)=\inf a_n(\C)$ and the punctures in $\Gamma_r$ are negative, we have $0\in \Gamma_r$.

	
	We claim that $\int_{\C\setminus \{\Gamma_r\}} u_r^*d\lambda>0$.  Suppose, by contradiction, that $\pi\cdot du_r\equiv 0$. 
	Using Theorem \ref{theo:vanishing-dlambda-energy} and the fact that $\tu_r$ is asymptotic to the simple orbit $P_3$ at its positive puncture $P_3$, we conclude that $\Gamma_r=\{0\}$ and $u_r(\zeta=e^{2\pi(s+it)})=x_3(T_3t)$.
	This leads to the contradiction
	\begin{equation*}
		T_3=\int_{\partial \D}u_r^*\lambda=\lim_{n\to \infty}\int_{\D}u_n^*d\lambda=T_3-\sigma(T_3).
	\end{equation*} 
	Here we have used \eqref{eq:u_n-e-seq-germinante}.
	
	Now we prove that $\Gamma_r=\{0\}$. 
	Since $P_3$ is simple, we know that $\tu_r$ is somewhere injective. By Theorem \ref{theo:fredholm-estimate}, we have the estimate
	$$1\leq \ind(\tu_r)=3-\sum_{z\in \Gamma_r}\mu(P_z)-2+\#\Gamma_r+1,$$
	where $P_z$ is the asymptotic limit of $\tu_r$ at $z\in \Gamma_r$.
	By Lemma \ref{le:indice-maior-igual-a-2}, we have $\mu(P_z)\geq 2$, $\forall z\in \Gamma_r$. Thus,  the only possibility is $\Gamma=\{0\}$ and $\mu(P_0)=2$, where $P_0=(x_0,T_0)$ be the asymptotic limit of $\tu_r$ at $0$.

	Our next claim is that $\tu_r$ is asymptotic to $P_2$ at $0$. 
	Since, for each $n$, $u_n$ is an embedding whose image does not intersect $P_3$, it follows that the image of any loop under $u_n$ is not linked to $P_3$. This implies that the image of any loop under $u_r$ is not linked to $P_3$ and consequently that $P_0$ is not linked to $P_3$. Moreover, 
	$T_3-T_0=\int_{\C\setminus \{0\}} u_r^*d\lambda>0$. 
	We conclude that $P_0=P_2$.
	
	As proved in \S\ref{se:soft-rescaling}, there exist sequences $\delta_n\to 0^+$, $z_n\to 0$ and $R_n\to \infty$
	such that  the sequence of $\tj$-holomorphic maps $\tilde{w}_n=(c_n,w_n):B_{R_n}(0)\to \R\times S^3$ defined by
	$$\tw_n:=(a_n(z_n+\delta_n\cdot)-a_n(z_n+2\delta_n), u_n(z_n+\delta_n\cdot)).$$
	is a germinating sequence.
	Let $\tu_q=(a_q,u_q):\C\setminus \Gamma_q\to \R\times S^3$ be a limit of the sequence $\tw_n$.
	Then $\tu_q$ is asymptotic to $P_2$ at $z=\infty$. By Lemma \ref{le:indice-maior-igual-a-2}, we conclude that  
		$\mu(P_z)\geq2, ~\forall z\in \Gamma_q,$ 
	where $P_z$ is the asymptotic limit of $\tu_q$ at $z$. 
	Since $P_2$ is simple, we know that $\tu_q$ is somewhere injective. We can apply Theorem \ref{theo:fredholm-estimate} to $\tu_q$ and obtain
	$$0\leq \ind(\tu_q)=2-2\#\Gamma_q-2+\#\Gamma_q+1~,$$
	which implies $\#\Gamma_q\leq 1$. By Proposition \ref{rmk:either-or}, we have $\int_{\C\setminus \Gamma_q}u^*_q d\lambda>0$. Again by Theorem \ref{theo:fredholm-estimate}, we have $\ind(\tu_q)\geq 1$ and consequently $\#\Gamma_q=0$.
	We conclude that $\tu_q:\C\to \R\times S^3$ is a finite energy plane asymptotic to the orbit $P_2$.
	This finishes the proof of Proposition \ref{theo:bubbling-off-two-vertices}.

\subsection{The foliation}

\begin{proposition}\label{pr:familia-de-planos-se-aproxima-limite}
	Consider a sequence $\tilde{u}_n:=\tilde{u}_{\tau_n}$ in the family \eqref{eq:familia-de-planos} satisfying $\tau_n\to 0^+$. 
	Let $\tu_r=(a_r,u_r):\C\setminus \{0\}\to \R\times S^3$ and $\tu_q=(a_q,u_q):\C\to \R\times S^3$ be the finite energy spheres obtained in Proposition \ref{theo:bubbling-off-two-vertices}.
	\begin{enumerate}[label=(\roman*)]
		\item Given an $S^1$-invariant neighborhood $\mathcal{W}_3\subset C^\infty(\R/\Z,S^3)$ of the loop $t\mapsto x_3(T_3t)$, there exists $R_0>>1$ such that, for $R\geq R_0$ and large $n$, the loop $t\mapsto u_n(Re^{2\pi it})$ belongs to $\mathcal{W}_3$.
		\item Given an $S^1$-invariant neighborhood $\mathcal{W}_2 \subset C^\infty(\R/\Z,S^3)$ of the loop $t\mapsto x_2(T_2t)$, there exist $\epsilon_1>0$ and $R_1>>0$ such that, for  $R_1\delta_n\leq R\leq \epsilon_1$ and large $n$,  the loop $t\mapsto u_n(z_n+Re^{2\pi it})$ belongs to $\mathcal{W}_2$. 
		\item Given any neighborhood $\mathcal{V}$ of $u_r(\C\setminus\{0\})\cup u_q(\C)\cup P_2\cup P_3$, we have $u_{n}(\C)\subset \mathcal{V}$ for large $n$.
	\end{enumerate}	
		A similar statement works for any sequence $\tau_n\to 1^-$ with $u_q$ and $u_r$ replaced by $u'_q$ and $u'_r$ respectively. 
\end{proposition}

\begin{proof}
	We can assume that $\mathcal{W}_i,i=2,3$, contains only the periodic orbit $t\mapsto x_i(T_i\cdot)$ modulo $S^1$-reparametrizations.
	Let $\mathcal{W}$ be an $S^1$-invariant neighborhood of the set  of periodic orbits $P=(x,T)\in \mathcal{P}(\lambda)$ with $T\leq T_3$, viewed as maps $x_T:S^1\to S^3$, $x_T(t)=x(Tt)$, 
	such that $\mathcal{W}_2, \mathcal{W}_3\subset \mathcal{W}$.
	Using the normalization condition \eqref{eq:u_n-e-seq-germinante} and Lemma \ref{le:cylinders-with-small-area}, we find $R_0>>1$ such that, for $R\geq R_0$, the loops 
	$\{t\mapsto u_n(Re^{i2\pi t})\}, n\in \N$ belong to $\mathcal{W}$.
	By the asymptotic behavior of the planes $\tu_n$, we conclude that $\{t\mapsto u_n(Re^{i2\pi t})\}, n\in \N$ belong to $\mathcal{W}_3$ for $R\geq R_0$. This proves \textit{(i)}.
	
	Applying Lemma \ref{le:cylinders-with-small-area} as in the proof of Proposition \ref{pr:same-orbits-tree}, we find $\epsilon_1>0$ small and $R_1>>1$ such that for every $R$ satisfying  $\delta_nR_1\leq R\leq \epsilon_1$ and large $n$, the loop 
	$t\mapsto u_n(z_n+Re^{2\pi it})$ belongs to $\mathcal{W}_2$. This proves \textit{(ii)}.
	
	The proof of \textit{(iii)} follows from  \textit{(i)}, \textit{(ii)} and Proposition \ref{theo:bubbling-off-two-vertices}.
	\end{proof}


%



\begin{proposition}\label{le:the-cylinders-do-not-intersect-orbits}
	Let $\tu_r=(a_r,u_r),\tu_r=(a_r',u_r'):\C\setminus \{0\}\to \R\times S^3$ and $\tu_q=(a_q,u_q),\tu_q'=(a_q',u_q'):\C\to \R\times S^3$ be the finite energy spheres obtained in Proposition \ref{theo:bubbling-off-two-vertices}.
		Then the projected curves $u_r:\C\setminus \{0\}\to S^3$, $u_r':\C\setminus \{0\}\to S^3$, $u_q:\C\to S^3$ and $u_q':\C\to S^3$ 
		are embeddings which are transverse to the Reeb vector field and do not intersect $P_2\cup P_3$.
\end{proposition}
\begin{proof}
	
	We first prove that $u_r(\C\setminus \{0\})\cap P_3=\emptyset$. 
	Consider $F:\C\setminus \{0\}\to \R\times S^3$ defined by $F(e^{2\pi(s+it)})=(T_3s,x_3(T_3t))$. Note that $F$ is a finite energy $\tilde{J}$-holomorphic immersion.
	By Carleman's similarity principle, 
	the intersections of $\tu_r$ with $F$ are isolated.
	By positivity and stability of intersections of pseudo-holomorphic curves, 
	any such intersection implies intersection of $\tu_n$ with $F$ for large $n$, contradicting Theorem \ref{th:family-of-cylinders-wendl}.
%
	To prove that $u_r(\C\setminus \{0\})\cap P_2=\emptyset$, we proceed in the same way, noting that $u_{n}(\C)\cap P_2=\emptyset,~\forall n$. Indeed, $P_2$ and $P_3$ are not linked and $u_n$ is transverse to the Reeb vector field. 
	
	In the same way, we prove that $u_r'(\C\setminus \{0\})$, $u_q(\C)$ and $u_q'(\C)$ do not intersect $P_2\cup P_3$. 
	Theorem \ref{th:siefring-2.6} shows that $u_r$, $u_r'$, $u_q$ and $u_q'$ are embeddings which are transverse to the Reeb vector field.
\end{proof}	


\begin{proposition}\label{pr:unicos-cilindros-p3p2-e-plano}
	Let $\tu_r=(a_r,u_r),\tu_r'=(a_r',u_r'):\C\setminus \{0\}\to \R\times S^3$ and $\tu_q=(a_q,u_q),\tu_q'=(a_q',u_q'):\C\to \R\times S^3$ be the finite energy spheres obtained in Proposition \ref{theo:bubbling-off-two-vertices}.
	Then, up to reparametrization and $\R$-translation, $\tilde{u}_q=\tu'_q$, and $\tilde{u}_q$ is the unique finite energy $\tilde{J}$-holomorphic plane asymptotic to $P_2$.
	If $u_r(\C\setminus \{0\})\neq u_r'(\C\setminus \{0\})$, then (up to reparametrization and $\R$-translation) $\tu_r$
	and $\tu_r'$
	are the unique finite energy $\tilde{J}$-holomorphic cylinders asymptotic to $P_3$ and $P_2$ at $z=\infty$ and $z=0$ respectively that do not intersect $\R\times \left(P_2\cup P_3\right)$. Moreover, $\tu_r$ and $\tu_r'$ approach $P_2$ in opposite directions {according to Definition \ref{de:opposite-directions}.}
\end{proposition}	

\begin{proof}
	Our proof follows {\cite[Proposition C.1]{dePS2013}}.
	First we prove that $\tu_q$ and $\tu'_q$ coincide up to reparametrization and $\R$-translation.
	Following Theorem \ref{theo:asymptotic behavior}, let $\eta_q$ be the asymptotic eigensection of $\tu_q$ at $\infty$ and  $\eta_q'$ the asymptotic eigensection of $\tu'_q$ at $\infty$. 
	It follows from Lemma \ref{le:wind=1,0} that 
	\begin{equation}\label{eq:wind-sigma=1}
	\wind(\eta_q)=\wind_\infty(\tu_q,\infty)=1 \text{ and } \wind(\eta_q')=\wind_\infty(\tu_q',\infty)=1.
	\end{equation}
	Since, by Proposition \ref{pr:properties-asymptotic-operator} and formula \eqref{eq:conley-zehnder-wind}, $\nu_{P_2}^{neg}$ is the unique negative eigenvalue of $A_{P_2,J}$ with winding number equal to $1$, it follows that $\eta_q$ and $\eta_q'$ are $\nu^{neg}_{P_2}$-eigensections. 
	Since the eigenspace of $\nu_{P_2}^{neg}$ is one dimensional, we find a constant $c\neq0$ such that 
	$\eta_q=c\eta_q'~.$
	Suppose, contrary to our claim, that $u_q(\C)\neq u'_q(\C)$. Then, by \eqref{eq:wind-sigma=1}, Proposition \ref{le:the-cylinders-do-not-intersect-orbits} and Theorem \ref{th:2.4-siefring}, we have 
	\begin{equation}\label{eq:uq_cap_uq'=0}
		u_q(\C)\cap u'_q(\C)=\emptyset~.
	\end{equation}
	By \eqref{eq:wind-sigma=1}, \eqref{eq:uq_cap_uq'=0} and Theorem \ref{th:2.5-siefring}, we conclude that $c<0$, that is, $\tu_q$ and $\tu_q'$ approach $P_2$ in opposite directions.
	This implies that $u_q(\C)\cup u'_q(\C)\cup P_2$ is a $C^1$-embedded sphere, where each hemisphere is a strong transverse section (see Remark \ref{rm:strong-transverse-section}), which is a contradiction with the hypotheses of Theorem \ref{theo:main-theorem}.
	We have proved that $\tu_q$ and $\tu'_q$ coincide up to reparametrization and $\R$-translation. 
	
	Using the same arguments above and noting that, by Theorem \ref{th:theorem-1.3-hwz-prop2}, any $\tj$-holomorphic plane asymptotic to $P_2$ is embedded and does not intersect $P_2$, we prove that any finite energy $\tilde{J}$-holomorphic plane asymptotic to $P_2$ coincides with $\tu_q$ up to reparametrization and $\R$-translation. 
	
	Now we prove the assertions about the cylinders $\tu_r$ and $\tu_r'$. 
	Let $\eta_r$ and $\eta_r'$ be the asymptotic eigensections of $\tu_r$ and $\tu_r'$ at $0$, respectively. 
	By Lemma \ref{le:wind=1,0}, we have 
	\begin{equation}\label{eq:wind-eta=1}
	\wind_\infty(\tu_r,0)=\wind_\infty(\tu'_r,0)=1.
	\end{equation}
	Using Proposition \ref{pr:properties-asymptotic-operator} and formula \eqref{eq:conley-zehnder-wind}, we conclude that  $\eta_r$ and $\eta_r'$ are $\nu^{pos}_{P_2}$-eigensections. 
	Since the eigenspace of $\nu_{P_2}^{pos}$ is one dimensional, we conclude that there exists a nonzero constant $c$ such that $\eta_r'=c\eta_r$.
	Assume that $u_r(\C\setminus \{0\})\neq  u_r'(\C\setminus \{0\})$.
	Then, by \eqref{eq:wind-eta=1}, Proposition \ref{le:the-cylinders-do-not-intersect-orbits} and Theorem \ref{th:2.4-siefring}, we have 
	\begin{equation}\label{eq:cilindros-nao-intersectam}
		u_r(\C\setminus \{0\})\cap  u_r'(\C\setminus \{0\})=\emptyset.
	\end{equation}
	By \eqref{eq:wind-eta=1}, \eqref{eq:cilindros-nao-intersectam} and Theorem \ref{th:2.5-siefring}, we conclude that $c<0$, that is, $\tu_q$ and $\tu_q'$ approach $P_2$ in opposite directions.

	By the same arguments above, we conclude that $\tu_r$ and $\tu_r'$ are the unique cylinders with the properties given in the statement. 
\end{proof}


\begin{proposition}\label{pr:foliation}
	Let $\tu_r=(a_r,u_r),\tu_r'=(a_r',u_r'):\C\setminus \{0\}\to \R\times S^3$ and $\tu_q=(a_q,u_q):\C\to \R\times S^3$ be the finite energy spheres obtained in Proposition \ref{theo:bubbling-off-two-vertices}.
	Then the cylinders $u_r$ and $u_r'$ satisfy $u_r(\C\setminus \{0\})\cap u_r'(\C\setminus \{0\})=\emptyset$, the surface $T:=P_2\cup P_3\cup u_r(\C\setminus \{0\})\cup u'_r(\C\setminus \{0\})$ is homeomorphic to a torus and $T\setminus P_3$  is $C^1$-embedded. 
	The union of the image of the family $\{u_\tau:\C\to S^3\},~\tau\in(0,1)$, given by \eqref{eq:familia-de-planos} with the images of $u_q$, $u_r$ and $u'_r$ determine a smooth foliation 
	of $\mathcal{R}_1\setminus(P_2\cup P_3)$, where $\mathcal{R}_1\subset S^3$ is a closed region with boundary  $T$. 
	Moreover, $\mathcal{R}_1$ is homeomorphic to a solid torus. 
\end{proposition}
\begin{proof}
	By Proposition \ref{le:the-cylinders-do-not-intersect-orbits} and Theorem \ref{th:2.4-siefring} we conclude that the images of the projected curves $u_r$, $u_q$ and $\{u_\tau\}, \tau\in (0,1)$, are mutually disjoint.
	Moreover, if $u_r(\C\setminus \{0\})\neq u_r'(\C\setminus \{0\})$, then the images of $u_r$, $u_r'$, $u_q$ and $\{u_\tau\}, \tau\in (0,1)$ are mutually disjoint.
	
	Now we prove that $u_r(\C\setminus \{0\})\cap u_r'(\C\setminus \{0\})=\emptyset$.
	Suppose, by contradiction, that the images of $u_r$ and $u_r'$ have a nonempty intersection.
	Then we have $u_r(\C\setminus \{0\})= u_r'(\C\setminus \{0\})$. 
	Let $p \in S^3\setminus \left( u_r(\C\setminus \{0\})\cup u_q(\C)\cup P_2\cup P_3\right)$ and consider a neighborhood $\mathcal{V}$ of $u_r(\C\setminus \{0\})\cup u_q(\C)\cup P_2\cup P_3$ such that $p\notin \mathcal{V}$. By Proposition \ref{pr:familia-de-planos-se-aproxima-limite}, we have $u_{\tau_0}(\C), u_{\sigma_0}(\C)\subset \mathcal{V}$ for some $\tau_0$ sufficiently close to $0$ and $\sigma_0$ sufficiently close to $1$. 
	The surface $S= u_{\tau_0}(\C)\cup u_{\sigma_0}(\C) \cup P_3$ is a piecewise smooth embedded sphere. By Jordan-Brouwer separation theorem, $S$ divides $S^3\setminus S$ into two disjoint regions $A_1$ and $A_2$ with boundary $S$. One of these regions, say $A_1$, contains $p$ and the other contains  $ u_r(\C\setminus \{0\})\cup u_q(\C)\cup P_2$. 
	The intersection of the image of the family $\{u_\tau\}_{\tau\in (0,1)}$ with ${A_1}$ is nonempty, open and closed in $A_1$.
	Thus $p$ is in the image of the family $\{u_\tau\}_{\tau\in (0,1)}$.
	We conclude that 
	$$S^3= u_r(\C\setminus \{0\})\cup u_q(\C)\cup P_2\cup P_3\cup\{u_\tau(\C)\}_{\tau\in (0,1)}.$$
	This contradicts the fact that the orbit $P_1$ is not linked to $P_3$ and the curves $u_r, u_q$ and $u_\tau$ are transverse to the Reeb vector field. 
	We have proved that  $u_r(\C\setminus \{0\})\cap u_r'(\C\setminus \{0\})=\emptyset.$
	Consequently, the surface $T=P_2\cup P_3\cup u_r(\C\setminus \{0\})\cup u'_r(\C\setminus \{0\})$ is a piecewise smooth embedded torus, and it follows from Proposition \ref{pr:unicos-cilindros-p3p2-e-plano} that 
	$T\setminus P_3$ is $C^1$-embedded.

	 By Jordan-Brouwer separation theorem, $T$ divides $S^3$ into two closed regions $\mathcal{R}_1$ and $\mathcal{R}_2$ with disjoint interiors and boundary $T$. 
	One of the regions, say $\mathcal{R}_1$, contains the image of the family $\{u_\tau\}$ and the plane $u_q(\C)$.
	Now we show that 
	$$\mathcal{R}_1= P_2\cup P_3\cup u_r(\C\setminus \{0\})\cup u'_r(\C\setminus \{0\})\cup u_q(\C)\cup\{u_\tau(\C)\}_{\tau\in (0,1)}.$$
	Let $p\in \mathcal{R}_1\setminus \left(P_2\cup P_3\cup u_r(\C\setminus \{0\})\cup u'_r(\C\setminus \{0\})\cup u_q(\C)\right)$ 
	and let $\mathcal{V}$ be a neighborhood of $P_2\cup P_3\cup u_r(\C\setminus \{0\})\cup u'_r(\C\setminus \{0\})\cup u_q(\C)$ such that
	$p\notin \mathcal{V}$. 
	By Proposition \ref{pr:familia-de-planos-se-aproxima-limite}, we have $u_{\tau_0}(\C), u_{\sigma_0}(\C)\subset \mathcal{V}$ for some $\tau_0$ sufficiently close to $0$ and $\sigma_0$ sufficiently close to $1$. 
	The surface $S= u_{\tau_0}(\C)\cup u_{\sigma_0}(\C) \cup P_3$ is a piecewise smooth embedded sphere. By Jordan-Brouwer separation theorem, $S$ divides $S^3\setminus S$ into two disjoint regions $A_1$ and $A_2$ with boundary $S$. One of these regions, say $A_1$ contains $p$ and is contained in $\mathcal{R}_1$. 
	Thus the intersection of the image of the family $\{u_\tau\}$ with $A_1$ is open, closed and nonempty in $A_1$. 
	This implies that $p$ is in the image of $u_\tau$, for some $\tau\in(0,1)$.


		To prove that $\mathcal{R}_1$ is homeomorphic to a solid torus, we 
		follow the proof of the Solid torus theorem in \cite{rolfsen2003knots}. 
		The curve $P_2$ is the boundary of an embedded 2-disk $\mathcal{D}$ contained in $\mathcal{R}_1$ such that $\mathring{\mathcal{D}}\subset \mathring{\mathcal{R}_1}$.
		Let $N$ be a bicollar neighborhood of $\mathcal{D}$ in $\mathcal{R}_1$, so that $N\cap T$ is an annular neighborhood of $P_2$ in $T$.
		The boundary of $\mathcal{R}_1\setminus N$ is the union of two disks on $\partial N$ and the set $T \setminus N$, which is an annulus.
		Thus $\mathcal{R}_1\setminus N$ is bounded by a piecewise smooth $2$-sphere in $S^3$. By the generalized Schönflies theorem, $\mathcal{R}_1\setminus N$ is homeomorphic to a closed $3$-ball.
		It follows that $\mathcal{R}_1$ is homeomorphic to a $3$-ball with a $D\times [0,1]$ attached, where $D$ is a closed 2-disk (by mapping $D\times\{0\}$ and $D\times \{1\}$ onto disjoint disks on the boundary of the $3$-ball). Since $\mathcal{R}_1$ is orientable, $\mathcal{R}_1$ is homeomorphic to $D\times S^1$.
	\end{proof}
The proof of Proposition \ref{pr:foliation-solid-torus} is complete.

\section{A cylinder asymptotic to $P_2$ and $P_1$}\label{se:cilindro-p2-p1}
In this section, we continue the proof of Theorem \ref{theo:main-theorem}. Recall that we have obtained a foliation of a closed region $\mathcal{R}_1\subset S^3$. Now we find a finite energy cylinder asymptotic to $P_2$ at its positive puncture and $P_1$ at its negative puncture, whose projection to $S^3$ is contained in the complement of $\mathcal{R}_1$. 

\begin{proposition}\label{pr:a-cylinder-asymptotic-p1-p2}
	Let $\lambda$ be a tight contact form on $S^3$ satisfying the hypotheses of Theorem \ref{theo:main-theorem}.
	Then there exists a finite energy cylinder $\tv_r=(b_r,v_r):\C\setminus \{0\}\to \R\times S^3$ asymptotic to $P_2$ at its positive puncture $z=\infty$ and $P_1$ at its negative puncture $z=0$.
	The projection $v_r$ is an embedding transverse to the Reeb vector field. Define $V=v_r(\C\setminus\{0\})$. Then $V\cap \mathcal{R}_1=\emptyset$ and $D\cup V\cup P_2$ is a $C^1$-embedded disk, where $D=u_q(\C)$ is given by Proposition \ref{pr:foliation-solid-torus}.
\end{proposition}

The contact structure $\xi=\ker \lambda$ coincides, up to a diffeomorphism, with $\xi_0:=\ker f\lambda_0$, where $\lambda_0$ is defined by \eqref{eq:liouville-form} and $f:S^3\to (0,+\infty)$ is a smooth function. Thus, we can assume that $\lambda=f\lambda_0$ for some smooth function $f:S^3\to (0,+\infty)$ without loss of generality.
{Following \cite{hwz1998} we define a symplectic cobordism between $(S^3,\lambda=f\lambda_0)$ and $(S^3,\lambda_E)$, where $\lambda_E$ is a dynamically convex contact form on $S^3$.}
Given $0<r_1<r_2$, with $\frac{r_1^2}{r_2^2}$ irrational, let $\lambda_E=f_E\lambda_0$ be the contact form associated to the ellipsoid 
$$E=\left\{(x_1,y_1,x_2,y_2)\in \R^4|\frac{x_1^2+y_1^2}{r_1^2}+\frac{x_2^2+y_2^2}{r_2^2}=1\right\}~,$$
that is, $f_E(x,y)=\left(\frac{x_1^2+y_1^2}{r_1^2}+\frac{x_2^2+y_2^2}{r_2^2}\right)^{-1}$.
The Reeb vector field $R_E$ defined by $\lambda_E$ has precisely two simple periodic orbits $\bar{P}_0$ and $\bar{P}_1$. Both periodic orbits and its iterates are nondegenerate. Their Conley-Zehnder indices are $\mu(\bar{P}_0)=3$ and $\mu(\bar{P}_1)=2k+1\geq 5$ where $k\geq 2$ is determined by $k<1+\left(\frac{r_1^2}{r_2^2}\right)<k+1$. See {\cite[Lemma 1.6]{hofer1995characterization}} for a proof of these facts.

We choose $0<r_1<r_2$ small enough so that 
$f_E<f \text{ pointwise on } S^3$
and a smooth function $h:\R\times S^3\to \R^+$ satisfying 
\begin{equation}
\begin{aligned}
h(a,\cdot)&=f_E, \text{ if } a\leq -2,\\
h(a,\cdot)&=f, \text{ if } a\geq 2,
\end{aligned}
\end{equation}

\begin{equation}\label{eq:derivada-cobordismo}
\dfrac{\partial h}{\partial a}\geq 0 \text{ on }\R\times S^3\text{ and }\dfrac{\partial h}{\partial a} >\sigma >0\text{ on }[-1,1]\times S^3.
\end{equation}
In view of \eqref{eq:derivada-cobordismo}, the $2$-form $d(h\lambda_0)$ restricted to $[-1,1]\times S^3$ is a symplectic form.

We consider the family of contact forms $\{\lambda_a=h(a,\cdot)\lambda_0, a\in \R\}$. The contact structure $\xi=\ker \lambda_a$ does not depend on $a$.
Choose $J_E\in \mathcal{J}(\xi,d\lambda_E)$ and let $\{J_a\in \mathcal{J}(\xi,d\lambda_a),a\in \R\}$ be a smooth family of $d\lambda_a$-compatible complex structures on $\xi$ so that $J_a=J$ if $a\geq 2$ and $J_a=J_E$ if $a\leq -2$.
We consider smooth almost complex structures $\bar{J}$ on the symplectization $\R\times S^3$ with the following properties.
On $(\R\setminus [-1,1])\times S^3$, we consider 
$$\bar{J}|_\xi=J_a \text{ and } \bar{J}\partial_a=R_{\lambda_a}~.$$
On $[-1,1]\times S^3$ we only require $\bar{J}$ to be compatible with the symplectic form $d(h\lambda_0)$. The space of such almost complex structures on $\R\times S^3$ is nonempty and contractible in the $C^\infty$-topology and will be denoted by $\mathcal{J}(\lambda, J, \lambda_E,J_E)$. 


\subsection{Generalized finite energy surfaces}
\begin{definition}
	Let $(S,j)$ be a closed Riemann surface and let $\Gamma\subset S$ be a nonempty finite set. A nonconstant smooth map $\tilde{u}:S\setminus \Gamma\to \R\times S^3$ is called a \textit{generalized finite energy surface} if it is $\bar{J}$-holomorphic, that is, satisfies $d\tilde{u}\circ j =\bar{J}(\tilde{u})\circ d\tilde{u}$, for some $\bar{J}\in \mathcal{J}(\lambda, J, \lambda_E,J_E)$, as well as the energy condition 
	$$E(\tilde{u})<+\infty~,$$
	where the energy $E(\tu)$ is defined as follows. Let $\Sigma$ be the collection of smooth functions $\phi:\R\to [0,1]$ satisfying $\phi'\geq 0$ and $\phi=\frac{1}{2}$ on $[-1,1]$. 
	Given $\phi\in \Sigma$ we define the $1$-form $\tau_\phi$ on $\R\times S^3$ by 
	\begin{equation*}
	\tau_\phi(a,x)(h,k)=\phi(a)\lambda_a(x)(k),
	\end{equation*}
	for any $(a,x)\in \R\times S^3$ and $(h,k)\in T_{(a,x)}(\R\times S^3)$. 
	Then
	\begin{equation}
	E(\tilde{u})=\sup_{\phi\in \Sigma}\int_{S\setminus \Gamma}\tilde{u}^*d\tau_\phi.
	\end{equation}
\end{definition} 
Theorem \ref{theo:fredholm-estimate} is still valid for almost complex structures in $\mathcal{J}(\lambda,J,\lambda_E,J_E)$.

\begin{theorem}[\cite{dragnev2004fredholm}]\label{th:fredholm-estimates-generalized}
	There exists a dense subset $\mathcal{J}^E_{reg}\subset \mathcal{J}(\lambda,J,\lambda_E,J_E)$ such that if $\tilde{u}=(a,u):S\setminus \Gamma\to \R\times S^3$ is a somewhere injective generalized finite energy surface for $\bar{J}\in \mathcal{J}^E_{reg}$, and has nondegenerate asymptotic limits at all of its punctures, then 
	$$0\leq \ind(\tilde{u})=\mu(\tilde{u})-\chi({S})+\#\Gamma~.$$
\end{theorem}

\subsection{A family of $\bar{J}$-holomorphic planes asymptotic to $P_2$}
We are interested in the space of finite energy generalized $\bar{J}$-holomorphic planes asymptotic to the orbit $P_2$, for fixed $\bar{J}\in \mathcal{J}^E_{reg}$.
The following theorem is a consequence of results from \cite{hofer1999properties3}.
\begin{theorem}[\cite{hofer1999properties3}]\label{th:espaço-de-curvas-e-variedade}
	Let $\tu_0:\C\to \R\times S^3$ be an embedded finite energy $\bar{J}$-holomorphic plane, asymptotic to a nondegenerate, simple Reeb orbit $P=(x,T)$ satisfying $\mu(P)=2$. Then there exists a smooth embedding 
	$$\tilde{\Phi}:\C\times (-\epsilon,\epsilon)\to \R\times S^3$$
	with the following properties:
	\begin{itemize}
		\item $\tilde{\Phi}(\cdot,0)=\tu_0$;
		\item For every $\tau\in (-\epsilon,\epsilon)$, the map $z\mapsto \tilde{\Phi}(z,\tau)$ is a generalized finite energy $\bar{J}$-holomorphic plane asymptotic to $P$;
		\item If $\tilde{u}_n$ is a sequence of finite energy $\bar{J}$-holomorphic planes asymptotic to $P$ satisfying $\tilde{u}_n\to \tilde{u}_0$ in $C^\infty_{loc}(\C)$ as $n\to +\infty$,
		then there exist sequences $A_n, B_n$ in $\C$ with $A_n\to 1$, $B_n\to 0$ and $\tau_n$ in $(-\epsilon,\epsilon)$ with $\tau_n\to0$ such that 
		$$\tilde{u}_n(z)=\tilde{\Phi}(A_nz+B_n,\tau_n)$$
		for sufficiently large $n$.
	\end{itemize}
\end{theorem}

From now on we fix  $\bar{J}\in \mathcal{J}^{E}_{reg}$, where $\mathcal{J}^{E}_{reg}$ is given by Theorem \ref{th:fredholm-estimates-generalized}.
Let $\Theta$ be the space of generalized finite energy $\bar{J}$-holomorphic planes asymptotic to $P_2$, modulo holomorphic reparametrizations. By Theorem \ref{th:espaço-de-curvas-e-variedade}, $\Theta$ is a smooth $1$-dimensional manifold.

\begin{lemma}\label{le:theta-nao-vazio}
	The space $\Theta$ is nonempty. 
\end{lemma}
\begin{proof}
	Let $\tu_q=(a_q,u_q):\C\to \R\times S^3$ be the $\tilde{J}$-holomorphic plane asymptotic to $P_2$ given by Theorem \ref{theo:bubbling-off-two-vertices}.
	After an $\R$-translation we can assume that $\min_{z\in \C}a_q(z)>2$, so that $\tu_q=(a_q,u_q)$ can be viewed as a generalized finite energy $\bar{J}$-holomorphic plane asymptotic to $P_2$.
\end{proof}
By the proof of Lemma \ref{le:theta-nao-vazio}, we have $\tu_q\in \Theta$. 
Let $\Theta'$ be the connected component of $\Theta$ containing $\tu_q$.

\subsection{Existence of a $\boldsymbol{\tj}$-holomorphic cylinder}

Consider a sequence $\tv_n=(b_n,v_n)$ of generalized finite energy planes representing elements of $\Theta'$ . The energy $E(\tv_n)$ is uniformly bounded by $T_2$. The following statement adapted from \cite{Hryniewicz2016} is a corollary of the SFT compactness theorem of \cite{bourgeois2003compactness}.


\begin{theorem}[\cite{Hryniewicz2016},Theorem 3.11]\label{theo:sft-compactness-corollary}
	Up to a subsequence of $\tv_n$, still denoted by $\tv_n$, there exists a bubbling-off tree $\mathcal{B}=(\mathcal{T},\mathcal{U})$ with the following properties
	\begin{itemize}
		\item For every vertex $q$ of $\mathcal{T}$ there exist sequences $z_n^q, ~\delta_n^q\in \C$ and $c_n^q\in \R$ such that 
		\begin{equation}\label{eq:sft-compactness-convergence}
		\tv_n(z^q_n+\delta_n^q \cdot)+c_n^q\to \tv_q ~\text{in }C^\infty_{loc}(\C\setminus \Gamma_q)\text{ as }n\to \infty
		\end{equation}
		Here $\tv + c := (b + c, v)$, where $\tv = (b, v)$ and $c\in \R$
		\item The curve $\tv_r$ is asymptotic to $P_2$ at $\infty$ and the asymptotic limits of all curves $\tv_q$ are closed orbits with periods $\leq T_2$  of the Reeb flow of either $\lambda$ or $\lambda_E$.
	\end{itemize}
\end{theorem}

Here, a bubbling off tree $\mathcal{B}=(\mathcal{T},\mathcal{U})$ consists of a finite, rooted tree $\mathcal{T}=(E,\{r\},V)$, with edges oriented away from the root, and a finite set of pseudo-holomorphic spheres $\mathcal{U}$, satisfying the following properties.
\begin{itemize}
	\item There is a bijective correspondence between vertices $q\in \mathcal{T}$ and finite-energy punctured spheres $\tv_q:\C\setminus \Gamma_q\to \R\times S^3 \in \mathcal{U}$. Each $\tv_q:\C\setminus \Gamma_q\to \R\times S^3$ is pseudo-holomorphic with respect to either $\tilde{J}$, $\tilde{J}_E$ or $\bar{J}$ and has finite energy.
	\item Each sphere $\tv_q$ has exactly one positive puncture at $\infty$ and a finite set $\Gamma_q$ of negative punctures. 
	If $\tv_q$ is $\tilde{J}$ or $\tilde{J}_E$-holomorphic and its contact area vanishes, then $\#\Gamma_q\geq 2$.
	\item  Each ordered path $(q_1,\dotsc, q_N)$ from the root $q_1=r$  to a leaf $q_N$, where $q_{k+1}$ is a direct descendant of $q_k$, contains at most one vertex $q_i$ such that $\tv_{q_i}$ is $\bar{J}$-holomorphic, in which case $\tv_{q_j}$ is $\tilde{J}$-holomorphic $\forall 1\leq j<i$, and $\tv_{q_j}$ is $\tilde{J}_E$-holomorphic $\forall i<j\leq N$.
	
	\item If the vertex $q$ is not the root then $q$ has an incoming edge $e$ from a vertex $q'$, and $\#\Gamma_q$ outgoing edges $f_1,\dots , f_{\#\Gamma_q}$ to vertices $p_1, \dots , p_{\#\Gamma_q}$ of $\mathcal{T}$, respectively. The edge $e$ is associated to the positive puncture of $\tv_q$ and the edges $f_1,\dots , f_{\#\Gamma_q}$ are associated to the negative punctures of $\tv_q$. 
	The asymptotic limit of $\tv_q$ at its positive puncture coincides with the asymptotic limit of $\tv_{q'}$ at its negative puncture associated to $e$. 
	In the same way, the asymptotic limit of $\tv_q$ at a negative puncture corresponding to $f_i$ coincides with the asymptotic limit of $\tv_{p_i}$ at its unique positive puncture.
\end{itemize}

The following statement is a reformulation of Lemma 3.12 from \cite{Hryniewicz2016}, adapted to our set-up. 
\begin{lemma}[\cite{Hryniewicz2016}, Lemma 3.12]\label{le:sft-compactness-corollary}
	Let $z_n^q, \delta_n^q, c_n^q$ be sequences such that \eqref{eq:sft-compactness-convergence} holds for all vertices $q$ of $\mathcal{T}$. Then we can assume, up to a selection of a subsequence still denoted by $\tv_n$, that one of the three mutually excluding possibilities holds for each vertex $q$.
	\begin{itemize}
		\item[(I)] $c_n^q$ is bounded, $b_n(z_n +\delta_n^q\cdot)$ is $C^0_{loc}(\C\setminus \Gamma_q)$-bounded and $\tv_q$ is a $\bar{J}$-holomorphic curve;
		
		\item[(II)] $c_n^q\to -\infty$, $b_n(z_n^q+\delta_n^q\cdot)\to +\infty$ in $C^0_{loc}(\C\setminus \Gamma_q)$ as $n\to \infty$ and $\tv_q$ is a $\tilde{J}$-holomorphic curve;
		
		\item[(III)]$c_n^q\to +\infty$, $b_n(z_n^q+\delta_n^q\cdot)\to -\infty$ in $C^0_{loc}(\C\setminus \Gamma_q)$ as $n\to \infty$ and $\tv_q$ is a $\tilde{J}_E$-holomorphic curve.
		
	\end{itemize}
	Moreover,  if $q$ is a vertex for which (III) holds, then $\tv_q$ is asymptotic at its positive puncture to a Reeb orbit having period strictly less than $T_2$. In particular, (III) does not hold for the root $r$.
\end{lemma}

{The following lemma is proved using the maximum
	principle combined with estimates for cylinders with small area (Lemma \ref{le:cylinders-with-small-area}). }
\begin{lemma}\label{le:plane-near-orbit}
	Let $z_n^r, \delta_n^r, c_n^r$ be sequences such that \eqref{eq:sft-compactness-convergence} holds for the root $r$. Then, by Theorem \ref{theo:sft-compactness-corollary}, $P_2=(x_2,T_2)$ is the asymptotic limit of $\tv_r$ at the positive puncture $\infty$. For every $\R/\Z$-invariant neighborhood $\mathcal{W}$ of $t\mapsto x_2(T_2t)$ in $C^\infty(\R/\Z,S^3)$ and for every number $M>0$, there exist $R_0>0$ and $n_0$ such that if $R>R_0$ and $n>n_0$, then the loop $t\mapsto v_n(z_n^r+\delta_n^rRe^{i2\pi t})$ belongs to $\mathcal{W}$ and $b_n^r(z_n^r + \delta_n^r e^{i2\pi t})+c_n^r>M$.
\end{lemma}

Since $\Theta'$ is a connected $1$-dimensional manifold without boundary, it is diffeomorphic either to $S^1$ or to an open interval. It can not be diffeomorphic to $S^1$, since it contains the family $\{[(a_q+t,u_q)]\}_{t>0}$ of equivalence classes of translations of $\tu_q$, where $\tu_q=(a_q,u_q)$ is the plane obtained in Theorem \ref{theo:bubbling-off-two-vertices}. Thus, the family is diffeomorphic to an interval and we can assume that  $\Theta'=\{[\tv_\tau=(b_\tau,v_\tau)]\}_{\tau\in(\tau_-,+\infty)}$, where for $\tau\geq 0$, $\tv_\tau=(a_q+\tau,u_q)$.
Consider a sequence $\tv_n:=\tv_{\tau_n}$ satisfying $\tau_n\to \tau_-$.

\begin{proposition}\label{prop:arvore-p1-p2}
	The bubbling-off tree obtained as an SFT-limit of the sequence $[\tv_n]$, as in Theorem \ref{theo:sft-compactness-corollary}, is as follows. The tree has vertices $r,q$, where $r$ is the root and $q$ is a leaf and direct descendant of $r$. The root $r$ corresponds to a $\tilde{J}$-holomorphic cylinder $\tv_r=(b_r,v_r):\C\setminus \{0\}\to \R\times S^3$ asymptotic to $P_2$ at its positive puncture $z=\infty$ and to 
	$P_1$ at its negative puncture $z=0$.
	The leaf $q$ corresponds to a $\bar{J}$-holomorphic plane asymptotic to $P_1$. 
\end{proposition}

\begin{proof}
	
	Let $\mathcal{B}=(\mathcal{T},\mathcal{U})$ be the bubbling-off tree given by Theorem \ref{theo:sft-compactness-corollary} and let $\tv_r:\C\setminus \Gamma_r\to \R\times S^3$ be the finite energy sphere associated to the root of $\mathcal{T}$. By Lemma \ref{le:sft-compactness-corollary}, $\tv_r$ is not $\tilde{J}_E$-holomorphic. 
	Now we show that $\tv_r$ is $\tilde{J}$-holomorphic. 
	
	Suppose, by contradiction, that $\tv_r$ is $\bar{J}$-holomorphic. By Theorem \ref{theo:sft-compactness-corollary}, $\tv_r$ is asymptotic to $P_2$ at $\infty$. Since $P_2$ is simple, $\tv_r$ is somewhere injective. If $\Gamma_r=\emptyset$, then $\tv_r\in \Theta$, contradicting the fact that the interval  $\Theta'$ is maximal.
	Assume $\Gamma_r\neq \emptyset$.  
	By Theorem \ref{th:fredholm-estimates-generalized}, it follows that 
	\begin{equation}
	0\leq\ind(\tv_r)=2-\sum_{z\in \Gamma_r}\mu(P_z)-2+\#\Gamma_r+1~.
	\end{equation}
	Then 
	\begin{equation}
	\sum_{z\in \Gamma_r}\mu(P_z)-\#\Gamma_r\leq 1~.
	\end{equation}
	Contradicting the fact that for all $z\in \Gamma_r$, the asymptotic limit $P_z$ of $\tv_r$ at $z$ is a Reeb orbit of $\lambda_E$, which is a dynamically convex contact from, that is, $\mu(P_z)\geq 3$, for all $z\in \Gamma_r$.
	Thus, $\tv_r$ is $\tilde{J}$-holomorphic.
	
	Let $m:=\min a_q(\C)$. 
	We claim that  $$\limsup_n\left(\min b_n(\C)\right)\leq m~.$$
	To prove the claim, suppose by contradiction that 
	$\limsup_n\left(\min b_n(\C)\right) >m>2.$ 
	Then there exists a 
	subsequence $\tv_{n_k}=(b_{n_k},v_{n_k})$ satisfying $\min b_{n_k}(\C)>m$. By the definiton of $\bar{J}$, the planes $\tv_{n_k}$ are $\tj$-holomorphic.
	By Proposition \ref{pr:unicos-cilindros-p3p2-e-plano}, we know that  $\tu_q$ is the unique finite energy $\tj$-holomorphic plane asymptotic to $P_2$, up to reparametrization and $\R$-translation.
	This implies  $[\tv_{n_k}]=[a_q+\tau_k,u_q]$ for a sequence $\tau_k>0$. 
	This contradicts the fact that $\Theta'$ is an interval.
	
	
	Now we show that $\Gamma_r\neq \emptyset$.
	Suppose, by contradiction, that $\tv_r$ is a $\tilde{J}$-holomorphic plane. 
	Let $z_n$, $\delta_n$ and $c_n$ be the sequences given by Theorem \ref{theo:sft-compactness-corollary}, such that 
	$$\tv_n(z_n+\delta_n \cdot)+c_n\to \tv_r ~\text{in }C^\infty_{loc}(\C)\text{ as }n\to \infty.$$
	The limit  $\tv_r$ satisfies \ref{le:sft-compactness-corollary}(II), so that $c_n\to -\infty$ and $b_n(z_n+\delta_n\cdot)\to +\infty$ in $C^0_{loc}(\C)$ as $n\to \infty$.
	By Lemma \ref{le:plane-near-orbit}, there exists $R_0>0$ such that $b_n(z_n+\delta_nz)>m-c_n>m$ for $|z|>R_0$ and $n$ large enough.
	For $|z|\leq R_0$, by Lemma \ref{le:sft-compactness-corollary}(II), we have $b_n(z_n+\delta_n z)>m$, for $n$ large enough. 
	Thus, $\inf b_n(\C)>m$ for $n$ large enough.
	This contradicts $\limsup_n \inf b_n(\C)\leq m$, and concludes the  proof of $\Gamma_r\neq \emptyset$.
	
	So far, we know that $\tv_r:\C\setminus\Gamma_r \to \R\times M$ is a $\tilde{J}$-holomorphic sphere and $\Gamma_r\neq \emptyset$. The next step is to prove that every negative asymptotic limit of $\tv_r$ has Conley-Zehnder index equal to 1.

	\paragraph{\textit{Claim I}} If $q$ is not the root and $P_\infty$ is the asymptotic limit of $\tv_q$ at $\infty$, then $\mu(P_\infty)\geq 1$.
	To prove Claim I, we argue indirectly assuming that $\mu(P_\infty)\leq 0$.
	The curve $\tv_q:\C\setminus \Gamma_q\to \R\times S^3$ factors as  
	$\tv_q=\tu\circ p~,$
	where $\tu:\C\setminus \Gamma'\to \R\times S^3$ is a somewhere injective finite energy sphere and $p$ is a polynomial. 
	If $P$ is the asymptotic limit of $\tu$ at $\infty$, then $P^{\deg p}=P_\infty$.
	By Lemma \ref{le:properties-cz-index}, $\mu(P_\infty)\leq 0$ implies $\mu(P)\leq 0$.
	By Theorem \ref{th:fredholm-estimates-generalized}, we have 
	$$0\leq \ind \tu:=\mu(P)-\sum_{z'\in \Gamma'}\mu(P_{z'})+\#\Gamma_q-1,$$
	where $P_{z'}$ is the asymptotic limit of $\tu$ at $z'$.
	If $\Gamma'=\emptyset$, we already have a contradiction. 
	Otherwise, there exists $z_0'\in \Gamma'$ such that $\mu(P_{z_0'})\leq 0$.
	Let $z_0\in \Gamma_q$ be such that $p(z_0)=z_0'$. Then $P_{z_0}=P_{z_0'}^k$, for some $k\leq \deg p$, where $P_{z_0}$ is the asymptotic limit of $\tv_q$ at $z_0$.
	By Lemma \ref{le:properties-cz-index}, we have $\mu(P_{z_0})\leq 0$. 
	Since the tree has a finite number of vertices, by induction we find a leaf $l$ of the tree such that the finite energy plane $\tu_l:\C\to \R\times S^3$ is asymptotic to an orbit $P$ with $\mu(P)\leq 0$, a contradiction. This proves Claim I.
	
	\paragraph{\textit{Claim II}} For every $z\in \Gamma_r$, we have  $\mu(P_z)=1$, where $P_z$ is the asymptotic limit of $\tv_r$ at $z$.
	To prove Claim II, first note that, by Theorem \ref{theo:fredholm-estimate}, we have
	$$0\leq \ind \tv_r=2-\sum_{z\in \Gamma_r} \mu(P_z)+\#\Gamma_r-1~.$$
	It follows that 
	$$\sum_{z\in \Gamma_r} \mu(P_z)\leq \#\Gamma_r+1~.$$
	Since, by Claim I, we have  $1\leq \mu(P_z)$, for every $z\in \Gamma_r$, there exists at most one puncture $z_0\in \Gamma_r$ such that $\mu(P_{z_0})\geq 2$.
	If there exists such puncture, we have $\ind (\tv_r)=0$. By Theorem \ref{theo:fredholm-estimate}, this implies $\pi\circ dv_r\equiv 0$.
	By the definition of bubbling-off tree, this implies $\#\Gamma_r\geq 2$.
	By Theorem \ref{theo:vanishing-dlambda-energy}, there exists a periodic orbit $P$ and a polynomial  $p:\C\to \C$ such that $p^{-1}(0)=\Gamma_r$ and $\tv_r=F_P\circ p$, where $F_P$ is the cylinder over the orbit $P$, 
	which contradicts the fact that $P_2$ is simple.
	This proves Claim II.
	We have also proved that $\int_{\C\setminus \Gamma_r} v_r^*d\lambda >0$.
	
    We are now in a position to show that $\Gamma_r=\{0\}$ and $\tv_r:\C\setminus \{0\}\to \R\times S^3$ is a cylinder asymptotic to $P_1$ at $z=0$.
	By Lemma \ref{le:wind=1,0}, we conclude that $\wind_\infty(\tv_r,z)=1,~\forall z\in \Gamma_r \cup \{\infty\}$.
	Recall that
	$u_q(\C)\cap P_2=\emptyset~.$
	Thus, the curves $\tu_q$ and $\tv_r$ satisfy condition (2) of Theorem \ref{th:2.4-siefring}. 
	Here $\tu_q$ is the plane asymptotic to $P_2$ obtained by Theorem \ref{theo:bubbling-off-two-vertices}. It follows from Theorem \ref{th:2.4-siefring} that the projected curve $u_q$ does not intersect any of the negative asymptotic limits of $\tv_r$. 
	This implies that $P_2$ is contractible in $S^3\setminus P_z$, for every $P_z$ asymptotic limit of $\tv_r$ at $z\in \Gamma_r$. 
	Consequently,
	$${\rm lk}(P_2,P_z)={\rm lk}(P_z,P_2)=0,~\forall z\in \Gamma_r.$$
	Since, by hypothesis, the orbit $P_1$ is the only orbit with Conley-Zehnder index $1$ and  period less than $T_2$ that is not linked to $P_2$, it follows that
	$$P_z=P_1,~~\forall z\in \Gamma_r~.$$
	Applying Theorem \ref{th:2.4-siefring} to the finite energy curves  $\tv_r:\C\setminus \Gamma_r\to \R\times S^3$ and $\tu_\tau:\C\to \R\times S^3$, where $\tu_\tau$ is any of the planes in the family \eqref{eq:familia-de-planos}, we prove that $v_r$ does not intersect $u_\tau$.
	We conclude that $v_r(\C\setminus \Gamma_r)\subset\mathcal{R}_2$, where $\mathcal{R}_2:=S^3\setminus \mathring{\mathcal{R}_1}$ and $\mathcal{R}_1$ is the closed region given by Proposition \ref{pr:foliation}.
	The region $\mathcal{R}_1$ contains an embedded disk with boundary $P_2$, so that $P_2$ is contractible in $\mathcal{R}_1$. One can show, using  Mayer-Vietoris sequence, that the holomology class of $\R/\Z\ni t \mapsto  x_2(T_2 t)$ generates $H_1(\mathcal{R}_2,\Z)$.  
	The projected curve $v_r$ defines a singular 2-chain in $C_2(\mathcal{R}_2,\Z)$, so that $[x_2(T_2\cdot)]=n[x_1(T_1\cdot)]$ in $H_1(\mathcal{R}_2,\Z)$, where $n=\#\Gamma_r$. It follows that $\#\Gamma_r=1$ and we can assume that $\Gamma_r=\{0\}$.

	 Now we prove that the next (and last) level of the bubbling-off tree consists of a $\bar{J}$-holomorphic plane. 
	Let $\tv_q:\C\setminus \Gamma_q\to \R\times S^3$ be the finite energy sphere associated to the unique vertex that is a direct descendant of the root $r$.
	The asymptotic limit of $\tv_q$ at $\infty$ is $P_1$, that is a simple orbit. 
	It follows that $\tv_q$ is somewhere injective. 
	{By the definition of bubbling-off tree, $\tv_q$ is either $\tilde{J}$-holomorphic or $\bar{J}$-holomorphic.}
	By Theorem \ref{th:fredholm-estimates-generalized}, we have
	$$0\leq\ind \tv_q=1-\sum_{z\in \Gamma_q} \mu(P_z)+\#\Gamma_q-1=\#\Gamma_q-\sum_{z\in \Gamma_q} \mu(P_z) ~,$$
	where $P_z$ is the asymptotic limit of $\tv_q$ at $z\in \Gamma_q$.
	Since, by Claim I, $\mu(P_z)\geq 1$ for all $z\in \Gamma_q$, it follows that  $\ind \tv_q=0$ and $\mu(P_z)=1$, for all $z\in \Gamma_q$.
	
	Suppose, by contradiction, that $\tv_q$ is $\tilde{J}$-holomorphic.
	By Theorem \ref{theo:fredholm-estimate}, we have $\pi\circ dv_q\equiv 0$. 
	By the definition of bubbling-off tree, this implies $\#\Gamma_q\geq 2$.
	Theorem \ref{theo:vanishing-dlambda-energy} and the fact that $P_1$ is simple lead to a contradiction.
	We have proved that $\tv_q$ is $\bar{J}$-holomorphic.
	
	Suppose that $\Gamma_q\neq \emptyset$. By the definition of bubbling-off tree, if $l$ is a vertex of the tree that is a direct descendant of $q$, then $\tv_l$ is necessarily $\tilde{J}_E$-holomorphic. The asymptotic limit of $\tv_l$ at $\infty$ is equal to $P_z$ for some $z\in \Gamma_q$. But $\mu(P_z)=1$ for all $z\in \Gamma_q$, contradicting the fact that all closed orbits of $R_E$ have Conley-Zehnder index $\geq 3$.
	We have proved that $\Gamma_q=\emptyset$. This finishes the proof of Proposition \ref{prop:arvore-p1-p2}.
	
\end{proof}

\begin{proposition}\label{pr:v_r-mergulho-nao intersecta-limites}
	The curve $\tv_r=(b_r,v_r):\C\setminus \{0\}\to \R\times S^3$ obtained in Proposition \ref{prop:arvore-p1-p2} is an embedding.  The projected curve $v_r:\C\setminus \{0\}\to S^3$ is an embedding which is transverse to the Reeb vector field and does not intersect any of its asymptotic limits.
	{Moreover, $v_r(\C\setminus \{0\})\cap \mathcal{R}_1=\emptyset$, where $\mathcal{R}_1$ is the closed region obtained in Proposition \ref{pr:foliation}.}
\end{proposition} 

To prove Proposition \ref{pr:v_r-mergulho-nao intersecta-limites}, we need the following.

\begin{lemma}\label{pr:cilindro-nao-intersecta-orbita}
	Let $\lambda$ be a tight contact form on $S^3$. 
	Let $ \tu:\R\times S^1\to \R\times S^3$ be a finite energy cylinder asymptotic to nondegenerate simple Reeb orbits $P=(x,T)$ at $+\infty$ and $\bar{P}=(\bar{x},\bar{T})$ at $-\infty$. Assume that 
	$P\neq \bar{P}$, 
		$P\cup \bar{P}$ is an unlink and 
	$\mu(P),~\mu(\bar{P}) \in \{1,2,3\}$. 
	Then $u(\R\times S^1)\cap P=\emptyset$ and $u(\R\times S^1)\cap \bar{P}=\emptyset$.
\end{lemma}
\begin{proof}
	Our arguments follow the proof of Theorem 4.4 from \cite{hofer1995properties2}, so we sketch the proof here and refer to the results of \cite{hofer1995properties2} when necessary.
	
	A finite energy surface $\tu:\C\setminus \Gamma\to \R\times S^3$ such that $\pi\cdot du$ is not identically zero can intersect its asymptotic limits in at most finitely many points (see {\cite[Theorem 5.2]{hofer1996properties1}}).
	This allows the definition of an algebraic intersection index as follows.
	Take a small embedded $2$-disk $\mathcal{D}$ transversal to the periodic orbit at a point $x(t)$ of $P$ and tangent to $\xi=\ker \lambda$ at $x(t)$, that is, $T_{x(t)\mathcal{D}}=\xi_{x(t)}$.
	We orient the disk in such a way that $T_{x(t)\mathcal{D}}$ and $\xi_{x(t)}$ have the same orientation.
	Let $M\in \R$ be such that all the intersection points of $u$ with $P$ are contained in $u((-\infty,M)\times S^1)$.
	Let $\varphi:\D\to S^3$ be a disk map to $\bar{P}$ such that $\varphi(\D)\cap P=\emptyset$. Such disk exists since $P\cup \bar{P}$ is a trivial link.
	Glue the disk $\D$ to $[-\infty,M]\times S^1$ along $-\infty\times S^1$ to form a new disk $D$.
	Let $\bar{u}:\bar{\R}\times S^1\to S^3$ be the map obtained by defining $\bar{u}(-\infty,t)=\bar{x}(\bar{T}t)$ and $\bar{u}(+\infty,t)=x(Tt)$ and 
	define $U:D\to S^3$ by 
	$$U|_{[-\infty,M]\times S^1}=\bar{u},~~~U|_{\D}=\varphi~.$$
	Consider $U_*:H_1(\partial D,\Z)\to H_1(S^3\setminus P,\Z)$. 
	If $\alpha$ is the generator of $H_1(\partial D,\Z)$, 
	there exists an integer, called 
	$int(u)$,
	such that 
	$$U_*(\alpha)=int(u)[\partial \mathcal{D}]\in H_1(S^3\setminus P,\Z)~.$$
	{Here we have used the fact that $H_1(S^3\setminus P,\Z)$ is generated by $[\partial \mathcal{D}]$.}
	
	It follows from the proof of Theorem 4.6 in \cite{hofer1995properties2} that $int(u)$ is the oriented intersection number of $U$ and $P$. 
	Moreover, all the intersections are in the image of the map $u$, since there is no intersections of $P$ with $\varphi(\D)$.
	Following the proof  of Theorem 4.6 in \cite{hofer1995properties2}, one can show that $int(u)\geq 0$ and  $int(u)=0$ if and only if $u(\R\times S^1)\cap P=\emptyset$.
	Now we show that $int(u)=0$.
	
	{Let $\Phi:\mathcal{U}\to S^3$ be an embedding of an open neighborhood $\mathcal{U}$ of the zero section of $\xi|_P$. 
		We require that $\Phi(0_p)=p$ and the fiberwise derivative of $\Phi$ at $0_p$ is the inclusion of $\xi$ into $T_pS^3$.}
	Consider  a nonvanishing section $v(t)$ of $\xi$ along $P$ which is contained in $\mathcal{U}$.
	Define the loop $\beta(v)$ by $\beta(v)(t)=\Phi\circ v(t)$ for $0\leq t\leq T$. 
	It is contained in $S^3\setminus P$ and we denote by $[\beta(v)]\in H_1(S^3\setminus P)$ the homology class generated by this loop.
	{Fix the  global trivialization $\Psi:\xi\to S^3\times \R^2$.}
	Let $\wind(v,\Psi)$ be the winding number of the small section $v(t)$ with respect to the trivialization $\Psi$.
	It is proved in \cite{hofer1995properties2} that 
	$$\wind(v,\Psi)[\partial \mathcal{D}]-[\beta(v)]\in H_1(S^3\setminus P)$$
	is independent of the section $v$ as described above.
	We  define a constant $c(u)$ by
	$$(\wind(v,\Psi)-c(u))[\partial\mathcal{D}]=[\beta(v)]~.$$
	If we choose, for example, the special section $v(t)$ such that $\Phi\circ v(t)= u(s^*,t)$ for some large $s^*\in \R$, then by definition 
	$[\beta(v)]=int(u)[\partial\mathcal{D}]$ 
	and taking the limit as $s^*\to \infty$, we have 
	$$\wind_\infty(\tu,\infty)-c(u)=int(u)~.$$
	It is proved in \cite{hofer1995properties2} that there exists an embedded disk $F=\varphi(\D)$ with $\varphi(\partial \D)=P$ whose characteristic distribution has $e^+\geq 1$ positive elliptic points, and that 
	$$c(u)=2e^+-1~~.$$
	By Lemma \ref{le:wind=1,0}, we have $\wind_\infty(\tu,\infty)=1$, so that
	$$int(u)=2-2e^+~.$$
	Since $int(u)\geq 0$ and $e^+\geq 1$, we have $int(u)=0$ and $e^+=1$.
	This shows that $u(\R\times S^1)\cap P=\emptyset$.
	We can repeat the arguments replacing $P$ by $\bar{P}$ to show that $u(\R\times S^1)\cap \bar{P}=\emptyset$ and conclude the proof.
\end{proof}

\begin{proof}[Proof of Proposition \ref{pr:v_r-mergulho-nao intersecta-limites}]
	From Lemma \ref{pr:cilindro-nao-intersecta-orbita}, we conclude that $v_r$ does not intersect its asymptotic limits. 
	Thus, $\tv_r$ satisfies condition \textit{(2)} of Theorem \ref{th:siefring-2.6}. We conclude that $\tv_r$ is an embedding and $v_r$ is an embedding transverse to the Reeb vector field.
	Applying Theorem \ref{th:2.4-siefring} to the finite energy cylinders $\tv_r:\C\setminus \{0\}\to \R\times S^3$ and $\tu_r:\C\setminus \{0\}\to \R\times S^3$, where $\tu_r$ is the cylinder obtained in Proposition \ref{theo:bubbling-off-two-vertices}, we prove that $v_r$ does not intersect $u_r$. Similarly, we prove that $v_r$ does not intersect $u_r'$. Since $P_1\subset (S^3\setminus \mathcal{R}_1)$, we conclude that $v_r(\C\setminus \{0\})\subset (S^3\setminus \mathcal{R}_1)$.
\end{proof}



\begin{proposition}\label{pr:unico-cilinro-p1-p2}
	Let $\tv_r=(b_r,v_r):\C\setminus \{0\}\to \R\times S^3$ be the $\tj$-holomorphic cylinder obtained in Proposition \ref{prop:arvore-p1-p2}. Then, up to reparametrization and $\R$-translation, $\tv_r$ is the unique $\tilde{J}$-holomorphic cylinder  asymptotic to $P_2$ at $\infty$ and to $P_1$ at $0$ that does not intersect $\R\times \left(P_1\cup P_2\right)$. Moreover, the cylinder $\tv_r$ and the plane $\tu_q$ obtained in Proposition \ref{theo:bubbling-off-two-vertices} approach $P_2$ in opposite directions, according to Definition \ref{de:opposite-directions}. Consequently, $u_q(\C)\cup v_r(\C\setminus \{0\})\cup P_2$ is a $C^1$-embedded disk.
\end{proposition}

\begin{proof}
	The proof follows {\cite[Proposition C.1]{dePS2013}}.
	Following Theorem \ref{theo:asymptotic behavior}, let $\eta_q$ be the asymptotic eigensection of $\tu_q$ at $\infty$ and let $\eta_r$ be the asymptotic eigensection of $\tv_r$ at $\infty$.
	By Lemma \ref{le:wind=1,0}, we have $\wind_\infty(\tu_q,\infty)=\wind(\tv_r,\infty)=1$.
	Using formula \eqref{eq:conley-zehnder-wind} and Proposition \ref{pr:properties-asymptotic-operator}, we conclude that $\nu_{P_2}^{neg}$ is the unique negative eigenvalue of $A_{P_2,J}$ with winding number $1$.
	By Proposition \ref{pr:properties-asymptotic-operator}, we know that the eigenspace of $\nu_{P_2}^{neg}$ is one dimensional.
	Thus, there exists $c\neq 0$ such that $\eta_r=c\eta_q$.
	By Proposition \ref{pr:v_r-mergulho-nao intersecta-limites} and Theorem \ref{th:2.5-siefring}, we conclude that $c<0$, that is, $\tv_r$ and $\tu_q$ approach $P_2$ in opposite directions.
	Using the same arguments above, we conclude that any cylinder with the properties given in the statement must have the same image as $\tv_r$. 
\end{proof}
The proof of Proposition \ref{pr:a-cylinder-asymptotic-p1-p2} is complete. 

\section{A family of cylinders asymptotic to $P_3$ and $P_1$}\label{se:familia-cilindros-folheacao}
Recall that we have obtained a foliation of a closed region $\mathcal{R}_1\subset S^3$ and  a finite energy cylinder asymptotic to $P_2$ at its positive puncture and $P_1$ at its negative puncture, whose projection to $S^3$ is contained in $\mathcal{R}_2:=S^3\setminus \mathring{\mathcal{R}_1}$. 
Now we construct a foliation of $\mathcal{R}_2$. More precisely, we prove the following statement. 
\begin{proposition}\label{pr:a-family-of-cylinders}
	Let $\lambda$ be a tight contact form on $S^3$ satisfying the hypotheses of Theorem \ref{theo:main-theorem}.
	Then there exists a family of finite energy cylinders $\{\tw_\tau=(c_\tau,w_\tau):\C\setminus \{0\}\to \R\times S^3\}_{\tau\in (0,1)}$, all of them  asymptotic to $P_3$ at their positive punctures $z=\infty$ and $P_1$ at their negative punctures $z=0$.
	The projections $w_\tau$  are embeddings transverse to the Reeb vector field and $\{C_\tau:=w_\tau(\C\setminus\{0\})\}_{\tau\in (0,1)}\subset \mathcal{R}_2\setminus (P_1\cup P_2\cup P_3)$. The union of $\{C_\tau\}_{\tau\in (0,1)}$, $U_1$, $U_2$ and $V$ determine a smooth foliation of $\mathcal{R}_2\setminus (P_1\cup P_2\cup P_3)$. Here the surfaces $U_1$, $U_2$ and $V$ are given by Propositions \ref{pr:foliation-solid-torus} and \ref{pr:a-cylinder-asymptotic-p1-p2}.
\end{proposition}

In Subsection \ref{se:conclusao-prova}, we also prove that ${\rm{sl}}(P_i)=-1$, $i=1,2,3$, 
completing the proof of Theorem \ref{theo:main-theorem}.


\subsection{Gluing}
The following statement is a consequence of the usual gluing theorem for pseudo-holomorphic curves in symplectizations, see \cite[\S7]{nelson2013} or \cite[\S10]{wendl2016lectures} for a nice exposition.
\begin{theorem}\label{th:gluing}
	Let $P_+,P,P_-\in \mathcal{P}(\lambda)$ be simple Reeb orbits such that 
	$\mu(P_+)=\mu(P)+1=\mu(P_-)+2$. 
	Let $\tu_+,\tu_-:\C\setminus\{0\}\to \R\times S^3$ be finite energy cylinders, which are pseudo-holomorphic with respect to $\tj\in \mathcal{J}_{reg}$ and such that $\tu_+$ is asymptotic to $P_+$ at the positive puncture $z=\infty$ and to $P$ at the negative puncture $z=0$ and $\tu_-$ is asymptotic to $P$ at the positive puncture $z=\infty$ and to $P_-$ at the negative puncture $z=0$. 
	Then there exists $R_0\in \R$ and a family of finite energy cylinders 
	$$\{\tu_R:\C\setminus\{0\}\to \R\times S^3\},~R\in [R_0,+\infty)$$
	asymptotic to $P_+$ at  the positive puncture $z=\infty$ and to $P_-$ at the negative puncture $z=0$.
	For every sequence $R_n\in [R_0,+\infty)$ satisfying $R_n\to \infty$,
	there exist sequences $\delta_n^\pm\in \C$ and $c_n^\pm\in \R$ such that 
	\begin{equation*}
		\begin{aligned}
			\tu_R(\delta_n^+\cdot)+c_n^+&\to \tu_+ ~~\text{in }C^\infty_{loc}(\C\setminus \{0\})\\
			\tu_R(\delta_n^-\cdot)+c_n^-&\to \tu_- ~~\text{in }C^\infty_{loc}(\C\setminus \{0\})
		\end{aligned}
	\end{equation*}	 
\end{theorem}


\subsection{A family of $\tj$-holomorphic cylinders}
Let $\tu_r:\C\setminus \{0\}\to \R\times S^3$ and $\tv_r:\C\setminus\{0\}\to \R\times S^3$ be the finite energy cylinders obtained by Proposition \ref{theo:bubbling-off-two-vertices} and Proposition \ref{prop:arvore-p1-p2} respectively. 
Applying Theorem \ref{th:gluing} to $\tu_r$ and $\tv_r$ we obtain a family of finite energy $\tj$-holomorphic cylinders 
\begin{equation}\label{eq:familia-de-cilindros}
	\{\tw_\tau=(c_\tau,w_\tau):\C\setminus \{0\}\to \R\times S^3\},\tau\in [R_0,+\infty),
\end{equation}
all of them asymptotic to the orbit $P_3$ at the positive puncture $z=\infty$ and $P_1$ at the negative puncture $z=0$.


\begin{proposition}\label{pr:cilindro-e-mergulho}
	For every $\tau\in[R_0,+\infty)$, $\tw_\tau$  is an embedding. The projection $w_\tau:\C\setminus\{0\}\to S^3$  is an embedding which does not intersect its asymptotic limits and is transverse to the Reeb vector field. Moreover, $w_\tau(\C\setminus \{0\})\cap \mathcal{R}_1=\emptyset$, where $\mathcal{R}_1$ is the closed region defined in Proposition \ref{pr:foliation}.
\end{proposition}
\begin{proof}
	By Lemma \ref{pr:cilindro-nao-intersecta-orbita} and Theorem \ref{th:siefring-2.6}, we conclude that $\tw_\tau$ is an embedding and $w_\tau$ is an embedding which does not intersect its asymptotic limits and is transverse to the Reeb vector field.
	Applying Lemma \ref{le:wind=1,0} and Theorem \ref{th:2.4-siefring} to $\tw_\tau$ and $\tu_r$, we conclude that $w_\tau$ and $u_r$ do not intersect and that $w_\tau$ does not intersect the orbit $P_2$.  Similarly, we conclude that $w_\tau$ and $u_r'$ do not intersect, where the $\tj$-holomorphic cylinder $\tu_r'=(a'_r,w_r'):\C\setminus \{0\}\to\R\times S^3$ is obtained in Proposition \ref{theo:bubbling-off-two-vertices}. Consequently, $w_\tau(\C\setminus \{0\})\cap \mathcal{R}_1=\emptyset$. 
\end{proof}

Applying Theorem \ref{th:family-of-cylinders-wendl} to the maps $\tw_\tau$ in \eqref{eq:familia-de-cilindros}, we obtain a maximal smooth one-parameter family of finite energy cylinders, containing the family \eqref{eq:familia-de-cilindros}. 
Assuming that $\tau$ strictly increases in the direction of $R_\lambda$ and 
the normalization $\tau\in (0,1)$, we denote this maximal family by   
\begin{equation}\label{eq:familia-maximal-cilindros}
\{\tw_\tau=(c_\tau,w_\tau):\C\setminus\{0\}\to \R\times S^3\}, ~\tau\in (0,1).
\end{equation}


\begin{proposition}\label{th:bubbling-tree-cylinders}
	Consider a sequence $\tilde{w}_n=(c_n,w_n):\C\setminus\{0\}\to \R\times S^3$ in the family \eqref{eq:familia-maximal-cilindros}, where $\tw_n=\tilde{w}_{\tau_n}$ and $\tau_n\to 1^-$. 
	Let $\tu_r,\tu_r':\C\setminus \{0\}\to \R\times S^3$ and $\tv_r:\C\setminus\{0\}\to \R\times S^3$ be the finite energy cylinders obtained by Proposition \ref{theo:bubbling-off-two-vertices} and Proposition \ref{prop:arvore-p1-p2} respectively. 
	Then after suitable reparametrizations and $\R$-translations of $\tw_n$, $\tu_r,~\tu_r'$ and $\tv_r$, we have
	\begin{enumerate}[label=(\roman*)]
		\item up to a subsequence, 
		$\tw_n\to \tu_r$ in $C^\infty_{loc}(\C\setminus \{0\})$ as $n\to \infty$.
		\item There exist sequences $\delta_n^+\to 0^+$ and $d_n\in \R$ such that, up to a subsequence,
		$\tw_n(\delta_n \cdot)+d_n\to \tv_r$ 	in $C^\infty_{loc}(\C\setminus \{0\})$ as $n\to \infty$. 		
	\end{enumerate} 
	A similar statement holds for any sequence $\tau_n\to 0^+$, with $\tu_r$ replaced by $\tu_r'$. 
	
\end{proposition}
The Proof of Proposition \ref{th:bubbling-tree-cylinders} is given in Subsections \ref{se:bubbling-familia-cilindros} and \ref{se:prova-teorema-cilindros}.

\subsection{Bubbling-off analysis for the family of cylinders}\label{se:bubbling-familia-cilindros}
Most of the material in Subsection \ref{se:bubbling-familia-cilindros} is adapted from {\cite[\S6.2]{hwz2003}}. However, we can not directly apply the results of \cite{hwz2003} since our hypotheses are slightly different. 

Consider a sequence $\tilde{w}_n=(c_n,w_n):\C\setminus\{0\}\to \R\times S^3$ in the family \eqref{eq:familia-maximal-cilindros}, where $\tw_n=\tilde{w}_{\tau_n}$ and $\tau_n\to 1^-$. 
Note that since all cylinders $\tw_n$ are asymptotic to $P_3$ at $z=0$, we have  $0<E(\tw_n)=T_3$. 
We reparametrize the sequence so that
\begin{equation}\label{eq:reparametrizacao-bubbling-cilindros}
\int_{\C\setminus \D}w_n^*d\lambda =\frac{\sigma(T_3)}{2}~.
\end{equation}
Define $\Theta=\{z\in \C\setminus \{0\}| \exists \text{ subsequence }\tw_{n_j} \text{ and } z_j\to z \text{ s.t. }|d\tw_{n_j}(z_j)|\to \infty\}.$ 
By the same arguments used in the proof of Proposition \ref{pr:limit-germinating-sequence}, we can assume that $\Theta$ is finite and  $\Theta\subset \D\setminus \{0\}$. Moreover, there exists a $\tj$-holomorphic map 
\begin{equation}\label{eq:definicao-w-limite-wn}
	\tw=(c,w):\C\setminus\left(\{0\}\cup \Theta\right)\to \R\times S^3
\end{equation}
such that, up to a subsequence, still denoted by $\tw_n$, 
$$\tw_n\to \tw\text{ in }C^\infty_{loc}\left(\C\setminus(\{0\}\cup \Theta),\R\times S^3\right)$$
and $E(\tw)\leq T_3$.
The punctures in $\{0\}\cup \Theta$ are non-removable and negative, and the puncture $z=\infty$ is positive.
Indeed, for any $\epsilon$ sufficiently large or small, we have
$$\int_{\partial B_\epsilon (0)}w^*\lambda =\lim_{n\to \infty}\int_{\partial B_\epsilon(0)}w_n^*\lambda \in [T_1,T_3] ,$$
where $\partial B_\epsilon(z)$ is oriented counterclockwise. 
It follows that $\infty$ is a positive puncture and $0$ is a negative puncture.
If $z\in \Theta$, then for any sufficiently small $\epsilon$, we have
$$\int_{\partial B_\epsilon(z)}w^*\lambda =\lim_{n\to \infty}\int_{\partial B_\epsilon(z)}w_n^*\lambda=\lim_{n\to \infty}\int_{B_\epsilon(z)}w_n^*d\lambda\geq T>0,$$
where $T\leq T_3$ is a period. 
This follows from the same arguments used in \S\ref{se:germinating-seuqences}. 
We conclude that $z$ is a negative puncture.
By the same arguments used in the proof of Proposition \ref{theo:bubbling-off-two-vertices} we conclude that the asymptotic limit of $\tw$ at $\infty$ is $P_3$.

\begin{lemma}\label{le:int-omega->0}
	$\int_{\C\setminus (\{0\}\cup \Theta)}w^*d\lambda >0~.$
\end{lemma}
\begin{proof}
	If $\Theta=\emptyset$, then it follows from \eqref{eq:reparametrizacao-bubbling-cilindros} that $\int_{\C\setminus\{0\}}w^*d\lambda \geq \frac{\sigma(T_3)}{2}$.
	Now assume $\Theta \neq \emptyset$ and suppose, contrary to our claim, that  $\int_{\C\setminus (\{0\}\cup \Theta)}w^*d\lambda =0$.
	By Theorem \ref{theo:vanishing-dlambda-energy}, there exists a polynomial $p:\C\to \C$ and a periodic orbit $P\in \mathcal{P}(\lambda)$ such that 
	$p^{-1}(0)=\{0\}\cup\Theta$  and $\tw=F_P\circ p$,
	where $F_P$ is the cylinder over the orbit $P$.
	But this implies $\deg p \geq 2$, contradicting the fact that the asymptotic limit of $\tw$ at $\infty$ is $P_3$, that is a simple orbit.
\end{proof}

\subsubsection{Soft rescaling near $\boldsymbol{z\in \Theta}$}
Assume $\Theta\neq \emptyset$ and take a puncture $z\in \Theta$.
Now we proceed as in the \textit{soft rescaling} done in \S\ref{se:soft-rescaling}. 
Define the mass $m(z)$ of $z$ as in \eqref{eq:int-maior-sigma-c}.
Fix $\epsilon>0$ satisfying \eqref{eq:m-mepsilon} and
choose sequences $z_n\in \overline{B_\epsilon(z)}$ and $0<\delta_n<\epsilon$ satisfying \eqref{eq:min-bn}-\eqref{eq:int-sigma-c}.
It follows that $z_n\to z$ and, passing to a subsequence, we have $\delta_n\to 0$.
Take $R_n\to \infty$ such that $\delta_nR_n<\frac{\epsilon}{2}$ and define
\begin{equation}
\begin{aligned}
\tv_n=(b_n,v_n):B_{R_n}(0)&\to \R\times S^3\\
\zeta&\mapsto (c_n(z_n+\delta_n\zeta)-c_n(z_n+2\delta_n),w_n(z_n+\delta_n\zeta)).
\end{aligned}
\end{equation}
The sequence $\tv_n$ is a germinating sequence according to Definition \ref{de:germinating-sequence}.
Let
\begin{equation}
\Theta_1=\{\zeta\in \C|\exists \zeta_j\to \zeta \text{ and subsequence } \tv_{n_j} \text{ s.t. } |d\tv_{n_j}(\zeta_j)|\to \infty \}~.
\end{equation}
Passing to a subsequence, we can assume that $\Theta_1$ is finite.
Let 
$$\tv_z:\C\setminus \Theta_1\to \R\times S^3$$
be a limit of $\tv_n$ as defined in \ref{de:limit-germinating-sequence}.
Let $P_z$ be the asymptotic limit of $\tw$ at $z$. 
Then $\tv_z$ is asymptotic to $P_z$ at its unique positive puncture $\infty$.
Using Lemma \ref{le:indice-maior-igual-a-2}, we conclude the following.

\begin{lemma}
	If $z\in \Theta$ and $\tw$ is asymptotic to $P_z$ at $z$, then $\mu(P_z)\geq2$.
\end{lemma}

{Since $P_3$ is simple, it follows that $\tw$ is somewhere injective.}
By Theorem \ref{theo:fredholm-estimate}, we have
$$1\leq \ind(\tw)=3-\sum_{z\in \{0\}\cup \Theta}\mu(P_z)+\#\Theta~$$
and consequently 
\begin{equation}\label{eq:estimate-fredholm-sylinders}
\sum_{z\in \{0\}\cup \Theta}\mu(P_z)\leq 2+\#\Theta~.
\end{equation}
This proves the following lemma.
\begin{lemma}\label{le:bubbling-cilindros-fredholm-estimate}
	Assume $\Theta\neq \emptyset$. Then $\mu(P_0)\leq 1$.	If $\mu(P_0)=1$, then $\Theta=\{z\}$ and $\mu(P_z)=2$. Here $P_z$ is the asymptotic limit of $\tw$ at the puncture  $z\in \{0\}\cup \Theta$.
\end{lemma}

\subsubsection{Soft-rescaling near $\boldsymbol{z=0}$}\label{se:soft-rescaling-cilindros}
For any $\epsilon>0$, define
\begin{equation}\label{eq:def-mepsilon}
	m_\epsilon(0)=\lim_{n\to \infty}\int_{B_\epsilon(0)\setminus \{0\}}w_n^*d\lambda,
\end{equation}
and define the \textit{mass of the puncture $z=0$} by
\begin{equation}\label{eq:mass-cylinders}
m(0)=\lim_{\epsilon\searrow 0}m_\epsilon(0)~.
\end{equation}
Note that, for large $n$ and small $\epsilon$, we have
\begin{equation}\label{eq:int-t1}
\int_{B_\epsilon(0)\setminus \{0\}}w^*_nd\lambda=\int_{\partial B_\epsilon(0)}w^*_n\lambda-\lim_{\delta\to 0}\int_{\partial B_\delta(0)}w^*_n\lambda
=\int_{\partial B_\epsilon(0)}w_n^*\lambda-T_1.
\end{equation}
It follows that 
\begin{equation}\label{eq:m0=t0-t1}
m(0)=\lim_{\epsilon\searrow 0}\int_{\partial B_\epsilon(0)}w^*\lambda-T_1=T_0-T_1,
\end{equation}
where $T_0$ is the period of the asymptotic limit $P_0=(x_0,T_0)$ of $\tilde{w}$ at the puncture $z=0$.
We have two cases: 
\begin{itemize}
	\item either $m(0)=0$ or 
	\item $m(0)>0\Rightarrow m(0)>\sigma(T_3)$.
\end{itemize}

\paragraph{I}First assume that $m(0)>\sigma(T_3)>0$. 
We claim that there is a sequence $\delta_n\to 0$ satisfying 
\begin{equation}\label{eq:int=m(0)-sigma/2}
\int_{B_{\delta_n(0)}\setminus \{0\}}w_n^*d\lambda=m(0)-\frac{\sigma(T_3)}{2}.
\end{equation}
Indeed, there exists a sequence $\delta_n$ satisfying the equation above, since using \eqref{eq:reparametrizacao-bubbling-cilindros} and $\int_{\C\setminus \{0\}}w_n^*d\lambda=T_3-T_1\geq m(0)>\sigma(T_3)$, we conclude
\begin{equation*}
\int_{\D\setminus\{0\}}w_n^*d\lambda\geq m(0)-\frac{\sigma(T_3)}{2}> 0~.
\end{equation*}
Now we show that $\liminf \delta_n=0$, so that, passing to a subsequence, still denoted by $\delta_n$, the claim is true.
Suppose that there exists $0<\epsilon'<\liminf \delta_n$. Then we have the contradiction
$$m(0)-\frac{\sigma(T_3)}{2}=\lim_{j\to \infty}\int_{B_{\delta_n(0)}\setminus \{0\}}w_n^*d\lambda\geq \lim_{j\to \infty}\int_{B_\epsilon'(0)\setminus\{0\}}w^*_nd\lambda \geq m(0)~.$$ 
This proves our claim.

Let $\epsilon_0>0$ be small enough so that the disks $B_{\epsilon_0}(z),~z\in \Theta \cup\{0\}$ are disjoint. 
Define
\begin{equation}\label{eq:defi-vn}
\tv_n(z)=(b_n(z),v_n(z))=(c_n(\delta_nz)-c_n(2\delta_n),w_n(\delta_nz))
\end{equation}
for $z\in B_{\frac{\epsilon_0}{\delta_n}}(0)\setminus \{0\}$.
It follows from \eqref{eq:def-mepsilon}, \eqref{eq:mass-cylinders}, \eqref{eq:int=m(0)-sigma/2} and \eqref{eq:defi-vn}  
that, for large $n$ and small $\epsilon_0$, we have the estimate
\begin{equation}\label{eq:estimativa-cilindro-menos-igual-sigma}
\begin{aligned}
\int_{B_{\frac{\epsilon_0}{\delta_n}(0)\setminus \D}}v_n^*d\lambda&=\int_{B_{\frac{\epsilon_0}{\delta_n}(0)\setminus \{0\}}}v_n^*d\lambda-\int_{\D\setminus\{0\}}v_n^*d\lambda\\
&=\int_{B_{\epsilon_0}(0)\setminus \{0\}}w_n^*d\lambda-\left(m(0)-\frac{\sigma(T_3)}{2}\right)\\
&\leq m(0)+\frac{\sigma(T_3)}{2}-\left(m(0)-\frac{\sigma(T_3)}{2}\right)=\sigma(T_3).
\end{aligned}
\end{equation}
Define
\begin{equation}\label{eq:theta_0}
	\Theta_0=\{z\in \C|\exists z_j\to z \text{ and subsequence } \tv_{n_j} \text{ s.t. } |d\tv_{n_j}(z_j)|\to \infty \}~.
\end{equation}
Using the proof of Proposition \ref{pr:limit-germinating-sequence} we conclude, passing to a subsequence,  that $\Theta_0$ is finite and $\Theta_0\subset \D\setminus\{0\}$. Moreover, there exists a $\tilde{J}$-holomorphic map 
$\tv_0=(b_0,v_0):\C\setminus \{0\}\cup \Theta_0\to \R\times S^3$
such that, passing to a subsequence
$$\tv_n\to \tv_0\text{ in } C^\infty_{loc}(\C\setminus \{0\}\cup \Theta_0)~.$$
The map $\tv_0$ is nonconstant, the punctures in $\{0\}\cup \Theta_0$ are non-removable and negative, and the puncture $z=\infty$ is positive.

\begin{lemma}\label{le:limites-iguais-seq-cilindros}
	The asymptotic limit of $\tv_0$ at its unique positive puncture $z=\infty$ is equal to $P_0$, the asymptotic limit of $\tw$ at $\{0\}$. 
\end{lemma}	
\begin{proof}
	Let $\mathcal{W}\subset C^\infty(S^1,S^3)$ be as in the statement of Lemma \ref{le:cylinders-with-small-area}. 
	Let $P_\infty$ be the asymptotic limit of $\tv$ at $\infty$ and 
	let $\mathcal{W}_\infty$ and $\mathcal{W}_0$ be connected components of $\mathcal{W}$ containing $P_\infty$ and $P_0$ respectively.
	Since $\tw_n \to \tw$ in $C^\infty_{loc}$, we can choose $0<\epsilon_0'<\epsilon_0$ small enough so that, if $0<\rho\leq \epsilon_0'$ is fixed, then the loop 
	$t\in S^1\mapsto w_n(\rho e^{i2\pi t})$ 
	belongs to $\mathcal{W}_0$ for large $n$.
	Since $\tv_n\to \tv$ in $C^\infty_{loc}$, we can choose $R_0>1$ large enough so that, if $R\geq R_0$ is fixed, then  the loop 
	$t\in S^1\mapsto v_n(Re^{i2\pi t})=w_n(\delta_nRe^{i2\pi t})$ 
	belongs to $\mathcal{W}_\infty$ for large $n$.
	By \eqref{eq:int=m(0)-sigma/2}, we can show that
	\begin{equation}\label{eq:def-e-cylinders-familia-cilindros}
	e:=\liminf \int_{\partial B_{\delta_n R_0(0)}}w_n^*\lambda>0.
	\end{equation}
	Consider, for each $n$, the $\tj$-holomorphic cylinder $\tilde{C}_n:\left[\frac{\ln R_o\delta_n}{2\pi},\frac{\ln \epsilon_0'}{2\pi}\right]\times S^1\to \R\times S^3$, defined by $\tilde{C}_n(s,t)=\tw_n(e^{2\pi(s+it)})$.
	It follows from \eqref{eq:estimativa-cilindro-menos-igual-sigma} that 
	\begin{equation}\label{eq:leq-sigma-cylinders-familia-cilindros}
	\int_{\left[\frac{\ln R_o\delta_n}{2\pi},\frac{\ln \epsilon_0'}{2\pi}\right]\times S^1}C_n^*d\lambda\leq \sigma(T_3)
	\end{equation}
	for large $n$.
	Using \eqref{eq:def-e-cylinders-familia-cilindros} and \eqref{eq:leq-sigma-cylinders-familia-cilindros} and applying Lemma \ref{le:cylinders-with-small-area}
	as in the proof of Proposition \ref{pr:same-orbits-tree}, we conclude that $\mathcal{W}_\infty=\mathcal{W}_0$ and consequently that $P_\infty=P_0$.
\end{proof}

\begin{lemma}\label{le:either-or-cilindros}
	Either 
	\begin{itemize}
		\item $\int_{\C\setminus \{0\}\cup \Theta_0}v_0^*d\lambda>0$ or
		\item $\int_{\C\setminus \{0\}\cup \Theta_0}v_0^*d\lambda=0 \text{ and } \#\Theta_0\geq 1$.
	\end{itemize}
\end{lemma}
\begin{proof}
	
	On the contrary, suppose that $\int_{\C\setminus \{0\}\cup \Theta_0}v_0^*d\lambda=0 \text{ and } \Theta_0= \emptyset$. Then
	\begin{equation*}
	\begin{aligned}
	T_0=\int_{\partial \D}v_0^*\lambda&=\lim_{n\to \infty}\int_{\partial B_{\delta_n}(0)}w_n^*\lambda\\
	&=\int_{B_{\delta_n}(0)\setminus \{0\}}w_n^*\lambda+\lim_{\epsilon\to 0}\int_{\partial B_\epsilon(0)}w_n^*\lambda\\
	&=m(0)-\frac{\sigma(T_3)}{2}+T_1\\
	&=T_0-\frac{\sigma(T_3)}{2},
	\end{aligned}
	\end{equation*}
	a contradiction.
\end{proof}	

\paragraph{II}Now assume that $m(0)=0$.
Let $\epsilon>0$ be small enough so that $m_\epsilon(0)\leq \frac{\sigma(T_3)}{2}$.
Define, for any sequence $\delta_n\to 0$, 
\begin{equation*}
\tv_n(z)=(b_n(z),v_n(z))=(c_n(\delta_nz)-c_n(2\delta_n),w_n(\delta_nz)),~~~z\in B_{\frac{\epsilon}{\delta_n}}(0)\setminus\{0\}.
\end{equation*}
Using \eqref{eq:def-mepsilon}, we conclude that 
\begin{equation}\label{eq:bubbling-cilindro-nao-tem-bubbling}
	\lim_{n\to \infty}\int_{B_{\frac{\epsilon}{\delta_n}}(0)\setminus\{0\}}v_n^*d\lambda=\int_{B_\epsilon(0)\setminus \{0\}}w_n^*d\lambda=m_\epsilon(0)\leq \frac{\sigma(T_3)}{2}~.
\end{equation}
It follows that $\Theta_0=\emptyset$, where $\Theta_0$ is defined as in \eqref{eq:theta_0}.
Indeed, if $z\in \Theta_0$, arguing as in the proof of Proposition \ref{pr:limit-germinating-sequence}, we find sequences $r_j\to 0^+$, $z_j\to z$ and $n_j\to +\infty$ such that $\lim_{n\to \infty}\int_{B_{r_j}(z_j)}v_n^*d\lambda\geq T,$  
for some period $T$, contradicting \eqref{eq:bubbling-cilindro-nao-tem-bubbling}.
Thus, passing to a subsequence, still denoted $\tv_n$, there exists a $\tj$-holomorphic map $\tv_0=(b_0,v_0):\C\setminus\{0\}\to \R \times S^3$ such that 
$\tv_n\to \tv_0 \text{ in } C^\infty_{loc}(\C\setminus \{0\})~.$ 

\begin{lemma}\label{le:orbitas-coincidem-massa-zero}
	If $m(0)=0$, the curve $\tv_0$ is a trivial cylinder over the orbit $P_1$. Moreover, $\tw$ is asymptotic to $P_1$ at its negative puncture $z=0$. 
\end{lemma}
\begin{proof}
It follows from \eqref{eq:bubbling-cilindro-nao-tem-bubbling} that  
$\int_{\C\setminus\{0\}}v_0^*d\lambda=0~.$ 
Indeed, if $\int_{\C\setminus\{0\}}v_0^*d\lambda>0$, then $\int_{\C\setminus\{0\}}v_0^*d\lambda>\sigma(T_3)$, which contradicts \eqref{eq:bubbling-cilindro-nao-tem-bubbling}.
Note that 
\begin{equation}\label{eq:int-vn}
\begin{aligned}
	\int_{\partial \D}v_0^*\lambda&= \lim_{n\to \infty}\int_{\partial \D}v_n^*\lambda\\
	&=\lim_{n\to \infty}\int_{\partial B_{\delta_n}(0)}w^*_n\lambda\\
	&=\lim_{n\to \infty}\left(\int_{\partial B_\epsilon(0)}w^*_n\lambda-\int_{B_\epsilon(0)\setminus B_{\delta_n}(0)}w_n^*\lambda\right).
\end{aligned}
\end{equation}
Using \eqref{eq:int-t1} we conclude that 
\begin{equation}\label{eq:teorema-sanduiche}
	0\leq \lim_{n\to \infty}\int_{B_\epsilon(0)\setminus B_{\delta_n}(0)}w_n^*\lambda\leq \lim_{n\to \infty}\int_{B_\epsilon(0)\setminus \{0\}}w_n^*\lambda=\int_{\partial B_\epsilon(0)}w^*\lambda-T_1.
\end{equation}
It follows from \eqref{eq:int-vn}, \eqref{eq:teorema-sanduiche} and \eqref{eq:m0=t0-t1} that  
\begin{equation}
\int_{\partial \D}v_0^*\lambda= 
\lim_{n\to \infty}\left(\int_{\partial B_\epsilon}w^*_n\lambda-\int_{B_\epsilon\setminus B_{\delta_n}}w_n^*\lambda\right)=T_1.
\end{equation}
We conclude that $\tv_0$ is a cylinder over a $T_1$-periodic orbit $P=(x,T_1)$.

Now we prove that $P=P_1=P_0$, where $P_0=(x_0,T_0)$ is the asymptotic limit of $\tilde{w}$ at the puncture $z=0$.
Let $\mathcal{W}\subset C^\infty(S^1,S^3)$ be as in the statement of Lemma \ref{le:cylinders-with-small-area}.
	For fixed $n$, we know that	$(t\mapsto w_n(\epsilon e^{2\pi it}))\to x_1(T_1t) \text{ as } \epsilon\to 0~.$ 
	Thus, we can choose a sequence $\delta_n\to 0$ such that  
	$$(t\mapsto w_n(\delta_ne^{2\pi i t})=v_n(e^{2\pi it}))\in \mathcal{W}_1,~~\forall n~,$$
	where $\mathcal{W}_1$ is the connected component of $\mathcal{W}$ containing $P_1$. We conclude that $P=P_1$.
	Since $(t\mapsto w(\epsilon e^{2\pi it})) \to x_0(T_0t)~\text{ as }\epsilon\to 0,$ 
	we conclude, from estimate \eqref{eq:bubbling-cilindro-nao-tem-bubbling} and Lemma \ref{le:cylinders-with-small-area}, arguing as in the proof of Proposition \ref{pr:same-orbits-tree}, that  $P=P_0$.
\end{proof}

This completes the analysis of the case $m(0)=0$.

Going back to the case $m(0)>\sigma(T_3)>0$,
if the mass of the puncture $z=0$ of $\tv_0$ is positive or $\Theta_0\neq 0$, we repeat the process.
It necessarily stops after finitely many iterations, when we reach punctures with zero mass or run out of bubbling-off points. This follows from Lemmas \ref{rmk:either-or} and \ref{le:either-or-cilindros}.
{We obtain a bubbling-off tree of finite energy spheres, defined as in Theorem \ref{theo:sft-compactness-corollary}. 
The leaves of the tree correspond to finite energy planes originating from the bubbling-off points and a cylinder over the orbit $P_1$, originating from the puncture $z=0$.



\subsection{Proof of Proposition \ref{th:bubbling-tree-cylinders}}\label{se:prova-teorema-cilindros}
The finite energy curve $\tw:\C\setminus (\Theta\cup \{0\})\to \R\times S^3$ defined by \eqref{eq:definicao-w-limite-wn} is asymptotic to $P_3$ at its positive puncture $z=\infty$ and to an orbit $P_0=(x_0,T_0)$ at the negative puncture $z=0$. If $\Theta\neq \emptyset$, the punctures in $\Theta$ are negative.

One of the leaves of the bubbling-off tree obtained from the sequence $\tw_n$ is a cylinder over the orbit $P_1$ originated from the puncture $z=0$. 
Using the fact that $\mu(P_1)=1$ and an argument similar to Claim I in the proof of Theorem \ref{prop:arvore-p1-p2}, we conclude that 
\begin{equation}\label{eq:indice-maiorigual-1}
	\mu(P_0)\geq 1.
\end{equation}

We first show that $\Theta=\emptyset$. 
On the contrary, suppose that $\Theta\neq\emptyset$. Then by Lemma \ref{le:bubbling-cilindros-fredholm-estimate}, we have $\mu(P_0)=1$, $\#\Theta=1$ and $\mu(P_z)=2$, where $P_z=(x_z,T_z)$ is the asymptotic limit at the unique puncture $z\in\Theta$.
The orbit $P_z$ is not linked to $P_3$. 
Indeed, for each $n$, $w_n$ is an embedding whose image does not intersect $P_3$, which implies that any contractible loop in $w_n(\C\setminus \{0\})$ is not linked to $P_3$ as well. It follows that any loop in the image of $w$ near the end $z\in \Theta$ is not linked to $P_3$ and consequently that $P_z$ is not linked to $P_3$. 
Since $T_z<T_3$, by hypothesis we have $P_z=P_2$. 
This implies that $P_2$ is contractible in $\mathcal{R}_2$. 
Since the region $\mathcal{R}_1$ contains an embedded disk with boundary $P_2$ and consequently $P_2$ is contractible in $\mathcal{R}_1$, it is easy to prove, using  Mayer-Vietoris sequence, that the holomology class of $\R/\Z \ni t\to x_2(T_2 t)$ generates $H_1(\mathcal{R}_2,\Z)$, a contradiction.

Now we show that the mass $m(0)$ of the puncture $z=0$, defined by \eqref{eq:mass-cylinders}, is positive.
On the contrary, suppose that $m(0)=0$.  By Lemma \ref{le:orbitas-coincidem-massa-zero}, we know that  
$\tw$ is asymptotic to $P_1$ at its negative puncture $z=0$. This contradicts the fact that the family of cylinders \eqref{eq:familia-maximal-cilindros} is maximal.

So far, we know that $\tw:\C\setminus\{0\}\to \R\times S^3$ is a $\tj$-holomorphic cylinder asymptotic to $P_3$ at $\infty$ and to the orbit $P_0$ at $0$. By \eqref{eq:estimate-fredholm-sylinders} and \eqref{eq:indice-maiorigual-1}, we have $1\leq \mu(P_0)\leq 2.$ 

The second level of the bubbling-off tree obtained from the sequence $\tw_n$ consists of  a unique vertex associated to a finite energy sphere 
$\tv_0:\C\setminus \{0\}\cup \Theta_0 \to \R\times S^3$, asymptotic to $P_0$ at its positive puncture $\infty$.
Observe that since the orbit $P_1$ is not linked to $P_3$ and, for each $n$, $P_3$ does not intersect the image of $w_n$, we know that any loop in the image of $w_n$ is also not linked to $P_3$. It follows that any loop in the image of $\tw$ is not linked to $P_3$ and consequently, that $P_0$ is also not linked to $P_3$.

\paragraph{\textit{Claim I:} $\mu(P_0)=2$.} 
To prove the claim, suppose by contradiction that $\mu(P_0)=1$.
If $\tv_0$ is somewhere injective, using Theorem \ref{theo:fredholm-estimate}, we have
$$0\leq 1-\mu(P_0^v)-\sum_{z\in \Theta_0}\mu(P_z^v)+\#\Theta_0,$$
where $P_z^v$ is the asymptotic limit of $\tv_0$ at the puncture $z\in \Theta_0\cup\{0\}$.
Using the fact that $\mu(P_1)=1$ and an argument similar to Claim I in the proof of Theorem \ref{prop:arvore-p1-p2}, we conclude that $\mu(P_0^v)\geq 1$.
If $z\in \Theta_0$, by Lemma \ref{le:indice-maior-igual-a-2}, we have $\mu(P_z^v)\geq 2$.
We conclude that $\#\Theta_0=0$ and $\pi\cdot dv_0\equiv 0$, which contradicts Lemma \ref{le:either-or-cilindros}.
If $\tv_0$ is not somewhere injective, there exists a somewhere injective $\tj$-holomorphic curve $\tu_0:\C\setminus \Gamma\to \R\times S^3$ and a polynomial $p:\C\to \C$ such that 
$\tv_0=\tu_0\circ p$, $p(\Theta_0\cup\{0\})=\Gamma$ and $p^{-1}(\Gamma)=\Theta_0\cup \{0\}$. 
This implies that $P_0=P_\infty^{\deg p}$, where $P_\infty$ is the asymptotic limit of $\tu_0$ at $\infty$. 
Using Lemma \ref{le:properties-cz-index} and the assumption $\mu(P_0)=1$, we conclude that $\mu(P_\infty)=1$. 
For every $z\in \Theta_0\cup \{0\}$, we have  $P_z^v=(P^u_w)^k$, where $p(z)=w$, $k\mid\deg p$ and $P_w^u$ is the asymptotic limit of $\tu_0$ at the puncture $w$. Since $\mu(P_z^v)\geq 1,~\forall z\in \Theta_0\cup \{0\}$, using Lemma \ref{le:properties-cz-index} we conclude that $\mu(P_z^u) \geq 1,~\forall z\in \Gamma$. 
Applying Theorem \ref{theo:fredholm-estimate} to the curve $\tu_0$ we have
$$0\leq 1-\sum_{z\in \Gamma}\mu(P_z^u)+\#\Gamma-1,$$
where $P_z^u$ is the asymptotic limit of $\tu_0$ at the puncture $z\in\Gamma$.
It follows that $\mu(P_z^u)=1,~\forall z\in \Gamma$.
By Theorem \ref{theo:fredholm-estimate}, we have $\pi\cdot du_0\equiv 0$, which implies $\pi\cdot dv_0\equiv 0$. 
It follows from Theorem \ref{theo:vanishing-dlambda-energy} that $\tv_0=F_P\circ p$, where $F_P$ is a cylinder over an orbit $P\in \mathcal{P}(\lambda)$, $p:\C\to \C$ is a  polynomial and 
$P_0=P^{\deg p}$. Since $\mu(P_0)=1$, we conclude, using Lemma \ref{le:properties-cz-index}, that $\mu(P)=1$. 
By Lemma \ref{le:either-or-cilindros}, we have $\Theta_0\neq \emptyset$. 
Let $z\in \Theta_0$. Then $\mu(P_z^v)\geq 2$ and $P_z^v=P^k$, where $k\mid \deg p$.
This implies that  $2\leq \mu(P^k)\leq \mu(P^{\deg p})=1$, a contradiction. 
This proves our claim.

Since $T_0<T_3$ and $P_0$ is not linked to $P_3$, by hypothesis we have $P_0=P_2.$

Now we prove that the mass $m(0)$ of the puncture $z=0$ of $\tv_0$ is zero, according to definition \eqref{eq:mass-cylinders}. 
This implies that $\tv_0$ is asymptotic to $P_1$ at its unique negative puncture $z=0$.
Suppose, by contradiction, that $m(0)>0$. 
Since $P_2$ is simple, we know that $\tv_0$ is somewhere injective. Applying Theorem \ref{theo:fredholm-estimate} to $\tv_0$, we have 
$$0\leq \ind(\tv_0)=2-\mu(P_0^v)-\sum_{z\in \Theta_0}\mu(P^v_z)+\#\Theta_0.$$
If $\ind(\tv_0)=0$, then $\pi\cdot dv_0\equiv0$. By Theorem \ref{theo:vanishing-dlambda-energy}, we conclude that $\tv_0$ is a cylinder over $P_2$, contradicting Lemma \ref{le:either-or-cilindros}.
Therefore, $\ind(\tv_0)\geq 1$ and we have 
$$1\leq \mu(P_0^v)\leq 1-\sum_{z\in \Theta_0}\mu(P^v_z)+\#\Theta_0.$$
Thus $\#\Theta_0=\emptyset$ and $\mu(P_0^v)=1$. 
Following the same arguments used to prove \textit{Claim I}, we get a contradiction with $m(0)>0$.
We have proved that $\tv_0$ is asymptotic to $P_1$ at $z=0$. 

We claim that $\Theta_0=\emptyset$. 
On the contrary, suppose that $\Theta_0\neq \emptyset$. Since $\tv_0$ is somewhere injective and $\pi\circ dv_0$ is not identically zero, by Theorem \ref{theo:fredholm-estimate} we have
$$1\leq \ind(\tv_0)=2-\mu(P_1)-\sum_{z\in \Theta_0}\mu(P^v_z)+\#\Theta_0=1-\sum_{z\in \Theta_0}\mu(P^v_z)+\#\Theta_0.$$
This implies that $\mu(P^v_z)\leq 1,~\forall z\in \Theta_0$, contradicting \ref{le:indice-maior-igual-a-2}.
So far, we have proved the following statement.
\begin{proposition}\label{pr:bubbling-off-cilindros-limite}
	Consider a sequence $\tilde{w}_n=(c_n,w_n):\C\setminus\{0\}\to \R\times S^3$ in the family \eqref{eq:familia-maximal-cilindros}, where $\tw_n=\tilde{w}_{\tau_n}$ and $\tau_n\to 1^-$. 
	There exists a cylinder $\tw:\C\setminus \{0\}\to \R\times S^3$ asymptotic to $P_3$ at its positive puncture $\infty$ and to $P_2$ at its negative puncture $0$ and a cylinder $\tv_0:\C\setminus \{0\}\to \R\times S^3$ asymptotic to $P_2$ at its positive puncture $\infty$ and to $P_1$ at its negative puncture $0$, such that,  after suitable reparametrizations and $\R$-translations of $\tw_n$, we have
	\begin{enumerate}[label=(\roman*)]
		\item up to a subsequence, 
		$\tw_n\to \tw$ in $C^\infty_{loc}(\C\setminus \{0\})$ as $n\to \infty$.
		\item There exist sequences $\delta_n^+\to 0^+$ and $d_n\in \R$ such that, up to a subsequence,
		$\tw_n(\delta_n \cdot)+d_n\to \tv_0$ 	in $C^\infty_{loc}(\C\setminus \{0\})$ as $n\to \infty$. 		
	\end{enumerate} 
	A similar statement holds for any sequence $\tau_n\to 0^+$ with $\tw$ replaced with a cylinder $\tw'$ with the same asymptotics as $\tw$ and $\tv_0$ replaced with a cylinder $\tv_0'$ with the same asymptotics as $\tv_0$. 
\end{proposition}

\begin{proposition}\label{pr:familia-de-cilindros-se-aproxima-limite}
	Let $\tw_n$, $\tw$ and $\tv_0$ be as in Proposition \ref{pr:bubbling-off-cilindros-limite}. Then
	\begin{enumerate}[label=(\roman*)]
		\item Given an $S^1$-invariant neighborhood $\mathcal{W}_3\subset C^\infty(\R/\Z,S^3)$ of the loop $t\mapsto x_3(T_3t)$, there exists $R_3>>1$ such that, for $R\geq R_3$ and large $n$, the loop $t\mapsto w_{n}(Re^{2\pi it})$ belongs to $\mathcal{W}_3$.
		\item Given an $S^1$-invariant neighborhood $\mathcal{W}_2\subset C^\infty(\R/\Z,S^3)$ of the loop $t\mapsto x_2(T_2t)$, there exist $\epsilon_2>0$ and $R_2>1$ such that  the loop $t\mapsto w_n(Re^{2\pi it})$ belongs to $\mathcal{W}_2$, for  $R_2\delta_n\leq R\leq \epsilon_2$ and large $n$.
		\item  Given an $S^1$-invariant neighborhood $\mathcal{W}_1\subset C^\infty(\R/\Z,S^3)$ of the loop $t\mapsto x_1(T_1t)$, there exist $\epsilon_1>0$ such that  the loop $t\mapsto w_n(\rho e^{2\pi it})$ belongs to $\mathcal{W}_1$, for  $\rho\leq \epsilon_1\delta_n$ and large $n$.
		\item Given any neighborhood $\mathcal{V}\subset \mathcal{R}_2$ of $w(\C\setminus\{0\})\cup v_0(\C\setminus {0})\cup P_1 \cup P_2\cup P_3$, we have $w_{n}(\C\setminus \{0\})\subset \mathcal{V}$, for large $n$.
	\end{enumerate}	
	A similar statement holds for any sequence $\tw_{\tau_n}$ such that $\tau_n\to 0^+$, with $w$ replaced by $w'$ and $v_0$ replaced by $v_0'$. 
\end{proposition}
\begin{proof}
	We can assume that $\mathcal{W}_i,i=1,2,3$, contains only the periodic orbit $t\mapsto x_i(T_i\cdot)$ modulo $S^1$-reparametrizations.   
	Let $\mathcal{W}\subset C^\infty(\R/\Z,S^3)$ be as in the statement of Lemma \ref{le:cylinders-with-small-area} and such that $\mathcal{W}_1\cup \mathcal{W}_2\cup \mathcal{W}_3\subset \mathcal{W}$.

	Using the normalization condition \eqref{eq:reparametrizacao-bubbling-cilindros} we can apply Lemma \ref{le:cylinders-with-small-area} and find $R_3>>1$ such that for $R\geq R_3$, the loops $\{t\mapsto w_n(Re^{i2\pi t})\}, n\in \N$ belong to $\mathcal{W}$.
	By the asymptotic behavior of the cylinders $\tw_n$, 
	we conclude that for any $R\geq R_3$ and large $n$, the loop $t\mapsto w_n(Re^{2\pi it})$ belongs to $\mathcal{W}_3$. 
	Assertion \textit{(i)} is proved. 
	
	Now we prove \textit{(ii)}. Recall that the asymptotic limit of $\tw$ at $z=0$ is the orbit $P_2$.  
	We can apply  \ref{le:cylinders-with-small-area} as in the proof of Proposition \ref{le:limites-iguais-seq-cilindros} to find $\epsilon_2>0$ small and $R_2>>1$ such that for every $R$ satisfying  $\delta_nR_2\leq R\leq \epsilon_2$ and large $n$, the loop 
	$t\mapsto w_n(Re^{2\pi it})$ belongs to $\mathcal{W}_2$. 
	
	To prove \textit{ (iii)}, first recall that the mass of the puncture $z=0$ of the sequence $\tv_n$ is zero. 
	Applying \ref{le:cylinders-with-small-area} as in the proof of Lemma \ref{le:orbitas-coincidem-massa-zero}, we find $\epsilon_1 >0$ small  and a sequence $\delta'_n\to 0$ such that 
	$t\mapsto w_n(\delta_n \rho' e^{2\pi i t})=v_n(\rho' e^{2\pi it})\in \mathcal{W}_1$ for $\delta'_n\leq \rho'\leq \epsilon _1$. 
	The sequence $v_n(\delta_n'\cdot)$ converges  in $C^\infty_{loc}$ to the cylinder over $P_1$, so that applying Lemma \ref{le:cylinders-with-small-area} again, we conclude that for large $n$, the loop $t\mapsto v_n(\delta_n' r e^{2\pi it})$ belongs to $\mathcal{W}_1$ for every $r\leq 1$. 
	We conclude that for any $\rho\leq \delta_n\epsilon_1$ and $n$ large enough, the loop $t\mapsto w_n(\rho e^{2\pi it})$ belongs to $\mathcal{W}_1$.
	
	The proof of \textit{(iv)} follows from \textit{(i)}, \textit{(ii)}, \textit{(iii)} and Proposition \ref{pr:bubbling-off-cilindros-limite}.
\end{proof}


\begin{proposition}\label{pr:tw=tu'}
	Let $\tw$, $\tv_0$, $\tw'$ and $\tv_0'$ be as in Proposition \ref{pr:bubbling-off-cilindros-limite}. Let $\tu_r$ and $\tu_r'$ be the finite energy cylinders obtained in Proposition \ref{theo:bubbling-off-two-vertices} and let $\tv_r$ be the finite energy cylinder obtained in Proposition \ref{prop:arvore-p1-p2}.  Then, up to reparametrization and $\R$-translation, $\tv_0=\tv_0'=\tv_r$, $\tw=\tu_r$ and $\tw'=\tu_r'$.
\end{proposition}
\begin{proof}
	By positivity and stability of intersection of pseudo-holomorphic curves and the fact that the cylinders $\tw_\tau$ in the family \eqref{eq:familia-maximal-cilindros} do not intersect $\R\times \left(P_1\cup P_2\cup P_3\right)$, we conclude that the cylinders $\tw$, $\tv_0$, $\tw'$ and $\tv_0'$ do not intersect $\R\times \left(P_1\cup P_2\cup P_3\right)$. 
	It follows from Proposition \ref{pr:unico-cilinro-p1-p2} that $\tv_0=\tv_0'=\tv_r$, up to reparametrization and $\R$-translation. 
	{It follows from Proposition \ref{pr:unicos-cilindros-p3p2-e-plano} that $\tw$ is either equal to $\tu_r$ or $\tu_r'$, up to reparametrization and $\R$-translation.}
	Similarly, $\tw'$ is either equal to $\tu_r$ or $\tu_r'$.
	Since the family of cylinders \eqref{eq:familia-de-cilindros} was obtained by Theorem \ref{th:gluing} applied to $\tu_r$ and $\tv_r$, we conclude, using Proposition \ref{pr:familia-de-cilindros-se-aproxima-limite}, that $\tw=\tu_r$. 
	
	It remains  to prove that $\tw'=\tu_r'$.
	Fix $\tau_0\in (0,1)$. 
	By Jordan-Brouwer separation theorem, the surface 
	$$S:=w_{\tau_0}(\C\setminus \{0\})\cup u_r(\C\setminus\{0\})\cup v_r(\C\setminus \{0\})\cup P_1 \cup P_2\cup P_3$$
	divides $S^3/S$ into two disjoint regions with boundary $S$.
	One of these regions, that we call $A$, is contained in $\mathcal{R}_2$, since $S$ does not intersect any of the curves foliating the interior of the region $\mathcal{R}_1$. 
	
	We will show that $A$ is foliated by cylinders in the family $\{w_\tau\}$. 
	Let $p\in A$ and let $\mathcal{V}\subset \mathcal{R}_2$ be a small neighborhood of $S$ in $\mathcal{R}_2$ separating $p$ and $S$. 
	By Proposition \ref{pr:familia-de-cilindros-se-aproxima-limite}, for $\tau_1>\tau_0$ sufficiently close to $1$, we have $w_{\tau_1}(\C\setminus\{0\})\subset \mathcal{V}$. 
	Moreover, $w_{\tau_1}(\C\setminus \{0\})\subset  \mathcal{V}\cap A$, since there are points in the image of the family $w_{\tau}$ converging to points in a compact subset of $u_r(\C\setminus \{0\})$, as $\tau\to 1^-$. 
	Let $B\subset A$ be the region limited by $w_{\tau_0}(\C\setminus \{0\})\cup P_3\cup P_1\cup w_{\tau_1}(\C\setminus \{0\})$. 
	Thus, the image of the family $w_\tau$, $\tau\in (\tau_0,\tau_1)$ is open, closed and nonempty in $B$. 
	It follows that $p$ is in the image of $w_\tau$ for some $\tau\in (\tau_0,\tau_1)$. 
	We conclude that $A$ is foliated by the images of the cylinders  $\{w_\tau\}$, $\tau\in (\tau_0,1)$. 
	
	Since every neighborhood of a compact set of $u_r(\C\setminus \{0\})$ in $\mathcal{R}_2$ is contained in $\bar{A}$, the family $\{w_\tau\}$, $\tau\in (\tau_0,1)$, foliates $A$ and the cylinders in the family $\{w_\tau\}$ do not intersect each other, we conclude that 
	$\tw'=\tu_r'$.
\end{proof}	
The proof of Proposition \ref{th:bubbling-tree-cylinders} is complete. 

\subsection{Conclusion of the proof of Theorem \ref{theo:main-theorem}}	\label{se:conclusao-prova}

\begin{proposition}\label{pr:folheacao-final}
	Let $\tu_r$ and $\tu_r'$ be the finite energy cylinders obtained in Proposition \ref{theo:bubbling-off-two-vertices} and let $\tv_r$ be the finite energy cylinder obtained in Proposition \ref{prop:arvore-p1-p2}. The images of the family $\{w_\tau\},\tau \in (0,1)$ given by \eqref{eq:familia-maximal-cilindros}, $u_r$, $u_r'$, and  $v_r$ determine a smooth foliation of $\mathcal{R}_2\setminus (P_1\cup P_2\cup P_3)$.
	Here $\mathcal{R}_2=\overline{S^3\setminus \mathcal{R}_1}$, where $\mathcal{R}_1$ is the region obtained in Proposition \ref{pr:foliation}.
\end{proposition}
\begin{proof}
	Arguing as in the proof of  Proposition \ref{pr:tw=tu'}, 
	we conclude that for $\tau_0\in (0,1)$ fixed, $\{w_\tau(\C\setminus\{0\})\},~\tau\in (\tau_0,1)$ determine a foliation of a region $A$ in $\mathcal{R}_2$ bounded by
	$$S:=w_{\tau_0}(\C\setminus \{0\})\cup u_r(\C\setminus\{0\})\cup v_r(\C\setminus \{0\})\cup P_1 \cup P_2\cup P_3.$$
	Repeating the arguments in the proof of Proposition \ref{pr:tw=tu'}, with $u_r$ replaced by $u_r'$, we find a foliation of the complement of $A$ in $\mathcal{R}_2$.  
\end{proof}

The proof of Proposition \ref{pr:a-family-of-cylinders} is complete. Now we complete the proof of Theorem \ref{theo:main-theorem}.

\begin{proposition}\label{pr:self-linking-1}
	${\rm sl}(P_i)=-1$, for $i=1,2,3$.
\end{proposition}
\begin{proof}
	First we prove that ${\rm{sl}}(P_3)=-1$. Our proof is adapted from \cite[\S6]{hs2011}.
	Let 
	 $F=F_\tau$ be one of the disks with boundary $P_3$ obtained in Proposition \ref{pr:foliation-solid-torus}.  
	Consider the orientation of $S^3$ induced by $\lambda\wedge d\lambda$ and 
	let $o$ be the orientation on $\bar{F}$ induced by $R_\lambda$. 
	The characteristic distribution of $\bar{F}$, defined by 
	$(T\bar{F}\cap \xi)^\bot,$
	where $\bot$ means the $d\lambda$-symplectic orthogonal,
	can be parametrized by a smooth vector field $X$ on $\bar{F}$, which is transverse to $\partial F$ pointing outwards. 
	After a $C^\infty$ perturbation of $F$ away from a neighborhood of $\partial F$ and keeping the transversality of $R_\lambda$ and ${F}$, we can assume that all singular points of $X$ are nondegenerate. 
	More details about these facts can be found in \cite{hofer1995properties2} and \cite{hofer1995characterization}.
	Let $o'$ be the orientation on $\xi$ given by $d\lambda|_\xi$. 
	A zero $p$ of $X$ is called positive if $o$ coincides with $o'$ at $p$, and negative otherwise.  
	Since $R_\lambda$ is transverse to ${F}$, all singularities of $X$ have the same sign and they must be positive since $\int_Fd\lambda=\int_{P_2}\lambda >0$. 
	Considering $X$ as a section of the bundle $\xi\to \bar{F}$, we have 
	\begin{equation}\label{eq:eq1prova-sl}
		\wind(X|_{\partial F},\Psi|_{\bar{F}})=\sum_{X(z)=0} {\rm sign}\left(DX(z):(T_zF,o_z)\to (\xi_z,o'_z)\right),
	\end{equation}
	where $\Psi$ is a trivialization of $\xi\to S^3$.
	Let $\Psi_F:T\bar{F}\to \bar{F}\times \R^2$ be a trivialization of $T\bar{F}\to \bar{F}$. 
	Considering $X$ as a section of the bundle $T\bar{F}\to \bar{F}$ and using the fact that $X$ points outwards at $\partial F$, we have 
	\begin{equation}\label{eq:eq2prova-sl}
		1=\wind(X|_{\partial F},\Psi_{F})=\sum_{X(z)=0} {\rm sign}\left(DX(z):(T_zF,o_z)\to (T_zF,o_z)\right).
	\end{equation}
	Since all singularities of $X$ are positive and $DX(z)$ at a singularity $z$ of $X$ does not depend on whether we consider $X$ as a section of $\xi\to \bar{F}$ or as a section of $T\bar{F}\to \bar{F}$, the right side of equations \eqref{eq:eq1prova-sl} and \eqref{eq:eq2prova-sl} coincide.

	Consider the nonvanishing section of $\xi\to \bar{F}$ defined by $Z(z)=\Psi^{-1}(z,(1,0))$. After a small perturbation of $P_3$ in the direction of $Z$, which we can assume to be transverse to ${F}$ and $\xi$, we find a new closed curve $\tilde{P}_3$. Consider a trivialization $\tilde{\Psi}$ of $(\xi|_{\partial F},o)$ induced by the nonvanishing section $X|_{\partial F}$. 
	We have
	$${\rm sl}(P_3)=\tilde{P}_3\cdot \bar{F}= \wind(Z,\tilde{\Psi})=-\wind (X|_{\partial F},\Psi)=-1.$$
	The proof of ${\rm sl}(P_2)=-1$ is completely analogous, replacing $F$ with the disk $D$ given by Proposition \ref{pr:foliation-solid-torus}. 
	
	Now we prove that ${\rm sl}(P_1)=-1$. 
	Let $V$ be the cylinder with boundary $P_1\cup P_2$ given by Proposition \ref{pr:a-cylinder-asymptotic-p1-p2}. 
	Let $o$ be the orientation on $\bar{V}$ induced by $R_\lambda$. The characteristic distribution of $\bar{V}$ can be parametrized by a smooth vector field $\tilde{X}$ on $\bar{V}$, which is transverse to $\partial V$ pointing outwards. Note that, along $P_2$, we have $\tilde{X}(z)=-X(z)$. 
	After a $C^\infty$ perturbation of $X$ away from a neighborhood of $\partial V$ and keeping the transversality of $R_\lambda$ and ${V}$, we can assume that all singular points of $\tilde{X}$ are nondegenerate. 
	Since $R_\lambda$ is transverse to ${V}$, all singularities of $\tilde{X}$ have the same sign and they must be positive since $\int_Vd\lambda=\int_{P_2}\lambda -\int_{P_1}\lambda>0$.
	Let $\Psi_V$ be a trivialization of $T\bar{V}\to \bar{V}$. Then we have 
	\begin{equation*}
		\begin{aligned}
			\wind(\tilde{X}|_{P_2}, \Psi)-\wind(\tilde{X}|_{P_1}, \Psi)&=\sum_{\tilde{X}(z)=0} {\rm sign}\left(D\tilde{X}(z):(T_zV,o_z)\to (\xi_z,o'_z)\right)\\
			&=\sum_{\tilde{X}(z)=0} {\rm sign}\left(D\tilde{X}(z):(T_zV,o_z)\to (T_zV,o_z)\right)\\
			&=\wind(\tilde{X}|_{P_2},\Psi_V)- \wind(\tilde{X}|_{P_1},\Psi_V)\\
			&=0.
		\end{aligned}
	\end{equation*}
	Here we have used the fact that $\tilde{X}$ points outwards. 
	Consequently, 
	$$\wind(\tilde{X}|_{P_1},\Psi)=\wind(\tilde{X}|_{P_2}, \Psi)=\wind({X}|_{P_2}, \Psi)=1.$$
	We know that $\mathcal{D}:=D\cup V\cup P_1\cup P_2$ is a $C^1$-embedded disk with boundary $P_1$.
	Note that $\mathcal{D}$ with the orientation induced by $P_1$ is an Seifert surface for $P_1$. 
	Consider the nonvanishing section of $\xi\to \mathcal{D}$ defined by $Z(z)=\Psi^{-1}(z,(1,0))$. After a small perturbation of $P_1$ in the direction of $Z$, which we can assume to be transverse to the interior of $\mathcal{D}$ and $\xi$, we find a new closed curve $\tilde{P}_1$. Consider a trivialization $\tilde{\Psi}$ of $(\xi|_{\partial \mathcal{D}},o)$ induced by the nonvanishing section $\tilde{X}|_{P_1}$. 
	Thus, 
	$${\rm sl}(P_1)=\tilde{P}_1\cdot (D\cup V)= \wind(Z,\tilde{\Psi})=-\wind (\tilde{X}|_{P_1},\Psi)=-1.$$
\end{proof}		

We have proved the existence of a $3-2-1$ foliation adapted to $\lambda$. {The proof of Theorem \ref{theo:main-theorem} is complete. } 

\section{Proof of Proposition \ref{TH:NECESSIDADE-3}}\label{se:prova-necessidade}
In this section we prove Proposition \ref{TH:NECESSIDADE-3}, which is restated below. 
\begin{proposition}
	Assume there exists a $3-2-1$ foliation adapted to the contact form $\lambda$ and let $P_2=(x_2,T_2)$ be the binding orbit with Conley-Zehnder index $2$, as in Definition \ref{de:3-2-1-foliation}.
	Then there is no $C^1$-embedding $\psi:S^2\to S^3$ such that $\psi({S^1\times \{0\}})=x_2(\R)$ 
	and each hemisphere is a strong transverse section.
\end{proposition}
Throughout the section, we follow the notation of Definition \ref{de:3-2-1-foliation}.
Before proving Proposition \ref{TH:NECESSIDADE-3} we need two lemmas, which are stated below.
\begin{lemma}\label{le:transversal-ao-toro}
	If $\psi:S^2\to S^3$ is a $C^1$-embedding  such that $\psi({S^1\times \{0\}})=P_2$ and  $\psi(S^2)\setminus P_2$ is transverse to the Reeb vector field, then we can assume that $\psi|_{S^2\setminus (S^1\times \{0\})}$ is also transverse to $T\setminus (P_3\cup P_2)=U_1\cup U_2$.
\end{lemma}

\begin{proof}
	Define 
$	F: \R\times S^2\to S^3$ by 
$	(t,x)\mapsto\varphi^t\circ \psi(x)$, where $\varphi^t$ is the Reeb flow. 
	Then $F$ satisfies $F(t,\cdot)(S^1\times \{0\})=\psi(S^1\times \{0\})$, $\forall t\in \R$, and 
	$$dF_{(t,x)}(a,v)=aR_\lambda(\varphi^t\circ \psi(x))+d\varphi^t_{\psi(x)}\cdot d\psi_xv~,$$
	for every $(t,x)\in \R\times S^2$ and $(a,v)\in \R\times T_{x}S^2$.
	Since $\psi$ is transverse to the Reeb flow away from $S^1\times \{0\}$, we know that $\pi\circ d\psi_x$ is surjective.
	Since $d\varphi^t$ preserves the splitting $TS^3=\R R_\lambda\oplus \xi$, it follows that
	$\pi\circ  d\varphi^t_{\psi(x)}\circ d\psi_x$ is surjective.
	We conclude that  $dF(t,x)$ is surjective for all $x\in S^2\setminus (S^1\times \{0\})$ and $t\in \R$. In particular, $F$ is transverse to $T\setminus (P_2\cup P_3)$ on $\R\times (S^2\setminus (S^1\times\{0\}))$. 
	This implies
	\footnote{
			If $X$, $S$ and $Y$ are smooth manifolds, $Z$ is a submanifold of $Y$ and $F:S\times X\to Y$ is a smooth map transverse to $Z$, then for almost every $s\in S$, $f_s:=F(s,\cdot)$ is transverse to $Z$.
			A proof of this fact can be found in {\cite[\S3]{guillemin2010differential}}.
	}
	that for almost all $t\in \R$, $F(t,\cdot)|_{S^2\setminus S^1\times \{0\}}$ is transverse to $T\setminus (P_2\cup P_3)$.
	
	Moreover, $F(t,\cdot)^*d\lambda=(\varphi^t\circ \psi)^*d\lambda=\psi^*d\lambda$, and it follows that $F(t,\cdot)$ is still transverse to the flow for all $t\in \R$.
	Thus, we can replace $\psi$ with $F(t,\cdot)$ for some $t$ satisfying $F(t,\cdot)|_{S^2\setminus S^1\times \{0\}} \pitchfork T\setminus (P_2\cup P_3)$ to get the desired embedding.  
\end{proof}

\begin{lemma}\label{LE:TRANSVERSAL-AO-LONGO-DE-P_2}
	Assume that $\psi:S^2\to S^3$ is a $C^1$ embedding such that $\psi({S^1\times \{0\}})=P_2$  and such that 
	the image of each closed hemisphere is a strong transverse section.
	Then $\psi$ is transverse to the torus $T$ along $P_2$. 
\end{lemma}
{The proof of Lemma \ref{LE:TRANSVERSAL-AO-LONGO-DE-P_2} is postponed to Subsection \ref{ch:apendice-prova-lema}. Subsection \ref{se:relative-position} below consists of preliminary results. 
	Now we use Lemmas \ref{le:transversal-ao-toro} and \ref{LE:TRANSVERSAL-AO-LONGO-DE-P_2} to prove Proposition \ref{TH:NECESSIDADE-3}.
}

\begin{proof}[Proof of Proposition \ref{TH:NECESSIDADE-3}]
	Suppose, by contradiction, that $\psi:S^2\to S^3$ is a $C^1$ embedding such that $\psi({S^1\times \{0\}})=P_2$  and such that 
	each hemisphere is a strong transverse section.
	
	Let $T=U_1\cup U_2\cup P_2\cup P_3$ be the torus given by Definition \ref{de:3-2-1-foliation}.
	$T$ divides $S^3$ into two closed regions $\mathcal{R}_1$ e $\mathcal{R}_2$. The region $\mathcal{R}_1$ contains an embedded disk with boundary $P_2$, so that $P_2$ is contractible in $\mathcal{R}_1$. One can also show, using  Mayer-Vietoris sequence, that the holomology class of $x_{T_2}$ generates $H_1(\mathcal{R}_2,\Z)$.  
	Lemma \ref{LE:TRANSVERSAL-AO-LONGO-DE-P_2} shows that the image of any neighborhood of $S^1\times \{0\}\subset S^2$ by $\psi$ intersects both $\mathcal{R}_1$ and $\mathcal{R}_2$ away from $\psi(S^1\times \{0\})$.
	This implies that  the sphere $\psi(S^2)$ intersects the torus $T$ away from $P_2$. Indeed, one of the hemispheres of $\psi(S^2)$ intersects $\mathcal{R}_2$ and can not be contained in $\mathcal{R}_2$, since this would imply $P_2$ contractible in $\mathcal{R}_2$, a contradiction. 
	
	Since 
	$\psi|_{S^2\setminus S^1\times \{0\}}$ is transverse to $P_3$ and ${\rm lk}(P_2,P_3)=0$, we conclude that $\psi(S^2)\cap P_3=\emptyset$. 
	Using Lemmas \ref{le:transversal-ao-toro} and \ref{LE:TRANSVERSAL-AO-LONGO-DE-P_2}, we conclude that 
	$\psi$ intersects $T$ transversely and the intersection $\psi(S^2)\cap T$ is contained in a closed subset of $T\setminus P_3$. 
	We conclude that the preimage of the intersection $\psi(S^2)\cap T$ by $\psi$ is a $1$-dimensional submanifold of $S^2$ which is a closed subset of $S^2$. 	
	It follows that each connected component of $\psi^{-1}(\psi(S^2)\cap T)$ is diffeomorphic to $S^1$.  
	
	The equator $S^1\times\{0\}$ is one of the connected components of the boundary of a region $R\subset S^2$ such that  $\psi(R)\subset \mathcal{R}_2$.
	Thus, one of the other connected components of the boundary of $R$, that we denote by $S$, is such that $\psi|_{S}$ is homologous to ${x_2}_{T_2}$ in $\mathcal{R}_2$.
	Denoting $H^1(T,\Z)=\Z[{x_2}_{T_2}]\oplus \Z [m]$, the homology class of $\psi|_{S}$ in $H^1(T,\Z)$ is $(1,l)$ for some $l\in \Z$. 
	Since $S$ does not intersect $P_2$, $l$ must be zero. 
	This implies that $\psi({S})$ and $P_2$ divide $T$ into two connected regions. 
	
	Now fix an orientation on $S^1\times \{0\}\subset S^2$ in such a way that $\psi|_{S^1\times \{0\}}$ preserves orientation.
	Consider the closed hemispheres of $S^2$, that we call $H_+$ and $H_-$, with the orientation induced by the orientation of $S^1\times \{0\}$. 
	It follows that 
	$$0<T_2=\int_{\R/\Z}{x_2}_{T_2}^*\lambda=\int_{H_\pm}\psi^*d\lambda~.$$
	Since $\psi$ is transverse to the Reeb vector field $R_\lambda$ in $H_\pm\setminus S^1\times\{0\}$, we have $\psi^*d\lambda>0$ in $H_\pm\setminus S^1\times\{0\}$.
	Let  $B$ be the connected region of $S^2$ bounded by $S^1\times \{0\}$ and $S$. Since $B$ is contained in one of the hemispheres of $S^2$, we have 
	\begin{equation}\label{eq:t2-maximo}
	0<\int_{B}\psi^*d\lambda=\int_{\R/\Z}{x_2}_{T_2}^*\lambda-\int_{S}\psi^*\lambda=T_2-\int_{\psi(S)}\lambda~.
	\end{equation}
	Since $\psi(S^2)$ does not intersect $P_3$, we know that one of the regions of $T$ bounded by $P_2$ and $\psi(S)$, that we denote by $A$, satisfies either $A\subset U_1$ or $A\subset U_2$. 
	Recall that $U_1$ and $U_2$ are oriented in such a way that  $d\lambda|U_{1,2}$ is an area form, $P_3$ is a positive asymptotic limit and $P_2$ is a negative asymptotic limit. Then we have
	\begin{equation}\label{eq:t2-minimo}
	0<\int_{A}d\lambda=\int_{\psi(S)}\lambda-\int_{P_2}\lambda=\int_{\psi(S)}\lambda-T_2~,
	\end{equation}
	contrary to \eqref{eq:t2-maximo}. This proves the proposition. 
\end{proof}

\subsection{{Relative position of sections along $\boldsymbol{P_2}$}}\label{se:relative-position}

The orbit $P_2$ is hyperbolic and lies in the intersection of its stable manifold $W^+(P_2)$ and its unstable manifold $W^-(P_2)$. 
The tangent spaces of $W^{\pm}(P_2)$ along the periodic solution $t\mapsto x_2(t)$ are spanned by the Reeb vector field $R_\lambda(x_2(t))$ and vector fields $v^\pm(t)\in \xi_{x_2(t)}$ defined below.

Consider the path of symplectic matrices 
$\Phi(t)=\Psi_{x_2(t)}\circ d\varphi^t|_{\xi(x_2(0))}\circ \Psi^{-1}_{x_2(0)},~t\in\R,$  
where $\Psi$ is any global symplectic trivialization of $\xi$. 
Since $\mu(P_2)=2$, we know that $\Phi(T_2)$ has two eigenvalues $\beta, \beta^{-1}$, where $\beta>1$.
The vector $v^-(0)$ is an eigenvector of $d\varphi^{T_2}|_{\xi(x_2(0))}$ associated to the eigenvalue $\beta$ and the vector $v^+(0)$ is an eigenvector of $d\varphi^{T_2}|_{\xi(x_2(0))}$ associated to the eigenvalue $\beta^{-1}$.
For each $t$, we can define 
$$v^\pm(t)=d\varphi^t|_{\xi(x_2(0))}v^\pm(0),$$ 
so that $v^-(t)$ is an eigenvector of  $d\varphi^{T_2}|_{\xi(x(t))}$ associated to the eigenvalue $\beta>1$ and $v^+(t)$ is an eigenvector of  $d\varphi^{T_2}|_{\xi(x_2(t))}$ associated to the eigenvalue $\beta^{-1}$.

Replacing $v^\pm(t)$ with $-v^\pm(t)$ if necessary, we can assume that $\{v^-(t),v^+(t)\}$ is a positive basis for $\xi_{x(t)}$, for each $t$.
The basis $\{v^-(t),v^+(t)\}$ determines four open quadrants in $\xi_{x(t)}$. 
Let $\rm{(I)}$ and $\rm{(III)}$ be the open quadrants between  $\R v^-(t)$ and $\R v^+(t)$ following the positive orientation and $\rm{(II)}$ and $\rm{(IV)}$ the open quadrants between  $\R v^+(t)$ and $\R v^-(t)$.


\begin{proposition}\label{pr:secoes-fortes-wind-1}
	Let $t\mapsto v(t)\in \xi_{x_2(t)}$ be a $T_2$-periodic nonvanishing section such that 
	$\wind(v,\Psi)=1~.$ 
	If $d\lambda(v(t),\mathcal{L}_{R_\lambda}v(t))>0,~\forall t\in \R,$
	then $v(t)$ belongs to regions $(I)$ or $(III)$ for every $t\in \R$. 
	If 
	$d\lambda(v(t),\mathcal{L}_{R_\lambda}v(t))<0,~\forall t\in \R,$  
	then $v(t)$ belongs to regions $(II)$ or $(IV)$ for every $t\in \R$. 
\end{proposition}
\begin{proof}
	Let $S(t)=-J_0\dot{\Phi}(t)\Phi^{-1}(t)$. Recall that $t\mapsto S(t)$ is $T_2$-periodic, and for each $t$, $S(t)$ is real and symmetric matrix. 
	Define $A(t)=J_0S(t)$. 
	In what follows we will also write $v^\pm(t)$ for the representations of the sections $v^\pm(t)$ in the trivialization $\Psi$.
	The sections $v^\pm(t)$ are solutions of the equation
	\begin{equation}\label{eq:edo-original}
	\dot{x}(t)=A(t)x(t)
	\end{equation} satisfying
	$v^+(T_2)=\frac{1}{\beta}v^+(0),~~~v^-(T_2)=\beta v^-(0)~.$
	In the basis $\{v^-(0),v^+(0)\}$, $\Phi(T_2)$ has the form 
	$$\Phi(T_2)=
	\begin{bmatrix}
	\beta & 0 \\
	0& \frac{1}{\beta} 
	\end{bmatrix}.$$
	We want to find a matrix $B$ satisfying $e^{T_2B}=\Phi(T_2)$ and a $T_2$-periodic map $t\mapsto P(t)$ such that $\Phi(t)=P(t)e^{tB},~\forall t\in \R$, so that with the change of variables $y=P^{-1}(t)x$, the equation \eqref{eq:edo-original} becomes 
	\begin{equation}\label{eq:y'=By}
	\dot{y}(t)=By(t).
	\end{equation}
	Define 
	$$B=\frac{1}{T_2}\begin{bmatrix}
	\ln \beta & 0\\
	0& -\ln \beta
	\end{bmatrix},$$ 
	and $t\mapsto P(t)$ by
	$\Phi(t)=P(t)e^{tB}.$  
	The map $t\mapsto P(t)$ is $T_2$-periodic. 
	In fact, 
	$$\Phi(t)e^{T_2B}=\Phi(t)\Phi(T_2)=\Phi(t+T_2)=P(t+T_2)e^{tB}e^{T_2B} \Rightarrow \Phi(t)=P(t+T_2)e^{tB}.$$
	If $x(t)$ is a solution of \eqref{eq:edo-original}, then $y(t)=P(t)^{-1}x(t)$ satisfies
	\begin{equation*}
	\begin{aligned}
	\dot{y}(t)&=(\dot{P(t)^{-1}})x(t)+P(t)^{-1}\dot{x}(t)\\
	&=-P(t)^{-1}\dot{P}(t)P(t)^{-1}(P(t)y(t))+P(t)^{-1}A(t)(P(t)y(t))\\
	&=(-P(t)^{-1}\dot{P}(t)+P(t)^{-1}A(t)P(t))y(t).	
	\end{aligned}
	\end{equation*}
	Using the identities $\dot{\Phi}(t)=A(t)\Phi(t)$ and $\Phi(t)=P(t)e^{tB}$, we get
	\begin{equation*}
	\begin{aligned}
	&A(t)P(t)e^{tB}=\dot{\Phi}(t)=P(t)e^{tB}B+\dot{P}(t)e^{tB}\\
	&\Rightarrow B= -P(t)^{-1}\dot{P}(t)+P(t)^{-1}A(t)P(t).
	\end{aligned}
	\end{equation*}
	Thus, $y(t)$ is a solution of \eqref{eq:y'=By}. 
	In the same way, if $y(t)$ solves \eqref{eq:y'=By}, then $x(t)=P(t)y(t)$ solves \eqref{eq:edo-original}.
	
	Writing equation \eqref{eq:y'=By} in coordinates, we get
	\begin{equation*}
	\left\{ \begin{array}{l}
	\dot{y_1}(t)=\lambda y_1(t)\\
	\dot{y_2}(t)=-\lambda y_2(t),
	\end{array} \right.  
	\end{equation*}
	where $\lambda=\frac{1}{T_2}\ln \beta$.
	Writing $y=(y_1,y_2)$ in polar coordinates, 
	\begin{equation*}
	\left\{ \begin{array}{l}
	{y_1}(t)=\rho(t)\cos \eta(t)\\
	{y_2}(t)=\rho(t)\sin\eta(t),
	\end{array} \right.  
	\end{equation*}
	we get 
	\begin{equation}\label{eq:derivada-zeta}
	\rho^2(t)\dot{\eta}(t)=\rho^2(t)\left(-2\lambda \sin\eta(t)\cos \eta(t)\right).
	\end{equation}

	Let $t\mapsto v(t)\in \xi_{x_2(t)}$ be a $T_2$-periodic nonvanishing section on $x_2^*\xi$. We want to compute 
	$d\lambda(v(t),\mathcal{L}_{R_\lambda}v(t)).$ 
	In the symplectic trivialization $\Psi$, $\mathcal{L}_{R_\lambda} v(t)$ takes the form
	$$\mathcal{L}_{R_\lambda}v(t)=\frac{d}{dt}v(t)-J_0S(t)v(t),$$
	where we also denote by $v(t)$ its representation in the trivialization $\Psi$.
	Thus,
	\begin{equation*}
	\begin{aligned}
	d\lambda(v(t),\mathcal{L}_{R_\lambda}v(t))&=d\lambda_0(v(t),\dot{v}(t)-J_0S(t)v(t))\\
	&=d\lambda_0(v(t),\dot{v}(t))-d\lambda_0(v(t),A(t)v(t)).
	\end{aligned}
	\end{equation*}
	Writing $v(t)$ in the basis $\{v^-(0),v^+(0)\}$ as $v(t)=(v_1(t),v_2(t))$ and in the basis 
	$\{P(t)v^-(0),P(t)v^+(0)\}$ as $v(t)=(u_1(t),u_2(t))$, we have
	$$(r(t)\cos\theta(t),r(t)\sin\theta(t)):=(u_1(t),u_2(t))=P(t)^{-1}(v_1(t),v_2(t)).$$
	Thus, we have 
	\begin{equation*}
	\begin{aligned}
	d\lambda_0(v(t),\dot{v}(t))&=C \cdot \det P(t)\cdot  d\lambda_0\left((u_1(t),u_2(t)),P(t)^{-1}(\dot{v_1}(t),\dot{v_2}(t))\right)\\
	&= C \left(r(t)^2 \dot{\theta}(t)+ d\lambda_0\left((u_1(t),u_2(t)), \frac{d}{dt}{(P(t)^{-1})}(v_1(t),v_2(t))\right)\right),
	\end{aligned}
	\end{equation*}
	where $C$ is a positive constant. 
	For  fixed $t$, let $s\mapsto x_t(s)$ be the solution of \eqref{eq:edo-original} such that $x_t(t)=v(t)$.
	Then
	\begin{equation*}
	\begin{aligned}
	A(t)v(t)&=A(t)x_t(t)=\dot{x_t}(t).
	\end{aligned}
	\end{equation*}
	Writing $x_t(s)$ is the basis $\{v^-(0),v^+(0)\}$ as $x_t(s)=(x_1(s),x_2(s))$ and in the basis 
	$\{P(t)v^-(0),P(t)v^+(0)\}$ as $(y_1(s),y_2(s))$, we have 
	$$(\rho(s)\cos\eta(s),\rho(s)\sin\eta(s))=(y_1(s),y_2(s))=P^{-1}(x_1(s),x_2(s)).$$
	Thus, we have
	\begin{equation*}
	\begin{aligned}
	d\lambda_0(v(t),A(t)v(t))&=d\lambda_0(x_t(t),\dot{x}_t(t))\\
	&=C\left(\rho^2(t)\dot{\eta}(t)+d\lambda_0\left((u_1(t),u_2(t)), \frac{d}{dt}(P(t)^{-1})(v_1(t),v_2(t))\right)\right).
	\end{aligned}
	\end{equation*}
	We conclude that 
	\begin{equation}\label{eq:theta-eta}
	d\lambda(v(t),\mathcal{L}_{R_\lambda}v(t))=Cr^2(t)(\dot{\theta}(t)-\dot{\eta}(t)).
	\end{equation}
	Assume that $t\mapsto v(t)\in \xi_{x_2(t)}$ satisfies $\wind(v,\Psi)=1$ and that  
	$d\lambda(v(t),\mathcal{L}_{R_\lambda}v(t))>0,~\forall t\in \R~.$
	By \eqref{eq:theta-eta}, we have 
	$\dot{\theta}(t)>\dot{\eta}(t),~\forall t\in \R.$ 
	Note that $\wind(v^\pm(t),\Psi)=1$. 
	This follows from $\mu(P_2)=2$ and the geometric description of the Conley-Zehnder index given in \S\ref{se:Conley-Zehnder-index}.
	This implies that 
	\begin{equation}\label{eq:wind=0}
	\wind(v(t),P(t)v^\pm(0))=\wind(v(t),v^\pm(t))=0~.
	\end{equation}
	Now we show that for all $t\in \R$, $v(t)$ lies in regions $\rm{(I)}$ or $\rm{(III)}$. 
	On the contrary, suppose that there exists $t_0\in \R$ such that $v(t_0)$ belongs to regions $\rm{(II)}$, $\rm{(IV)}$, or is the direction of $P(t_0)v^-(0)$. 
	By \eqref{eq:derivada-zeta}, we have  $\dot{\theta}(t_0)>\dot{\eta}(t_0)\geq0$.
	This implies that $\theta(t)$ is increasing near $t_0$. 
	Since $\wind(v(t),P(t)v^\pm(0))=0$, this would force the existence of $t_1>t_0$ such that $\theta(t_1)=\theta(t_0)$ and $\dot{\theta}(t_1)\leq0$, a contradiction with $\dot{\theta}(t_1)>\dot{\eta}(t_1)\geq0$. 
	Now, suppose that, for some $t_0$, $v(t_0)$ is in the direction of $P(t_0)v^+(0)$, then, using \eqref{eq:derivada-zeta}, we have $\dot{\theta}(t_0)>\dot{\eta}(t_0)=0$. 
	But this would force $v(t)$ to go to $\rm{(II)}$ or $\rm{(IV)}$, which is impossible.
	We conclude that $v(t)$	lies in regions $\rm{(I)}$ or $\rm{(III)}$, for every $t\in \R$. 
	
	By a similar argument, we conclude that if $t\mapsto v(t)\in \xi_{x_2(t)}$ satisfies $\wind(v,\Phi)=1$ and 
	$d\lambda(v(t),\mathcal{L}_{R_\lambda}v(t))<0,~\forall t\in \R~,$
	then we have 
	$v(t)\in \rm{(II)}$ or $v(t)\in\rm{(IV)}$, $\forall t\in \R$.
\end{proof}

\subsection{Proof of Lemma \ref{LE:TRANSVERSAL-AO-LONGO-DE-P_2}}\label{ch:apendice-prova-lema}
\begin{lemma}\label{le:wind=1}
	Let  $\R/\Z\ni t\mapsto \eta(t)\in \xi_{x_2(T_2t)}$ be a section on ${x_2}_{T_2}^*\xi$ such that $\eta(\cdot)$ and ${R_\lambda}(x_2(T_2\cdot))$ generate $d\psi(TS^2)$ along $x_2(T_2\cdot)$. Let  $\R/\Z\ni t\mapsto \eta'(t)\in \xi_{x_2(T_2t)}$ be a section on ${x_2}_{T_2}^*\xi$ such that $\eta'(\cdot)$ and  ${R_\lambda}(x_2(T_2\cdot))$ generate the tangent space of $T$ along $x_2(T_2\cdot)$.
	Then $\wind(\eta', \Psi)=\wind(\eta,\Psi)=1$. 
\end{lemma}
\begin{proof}
	Let $H\subset S^2$ be one of the closed hemispheres of $S^2$ and let $F:=\psi(H)\subset S^3$. 
	Then $F$ is an embedded disk satisfying $\partial F=P_2$ and $\mathring{F}$ is transverse to the Reeb vector field $R_\lambda$. 
	The characteristic distribution 
	$(TF\cap \xi)^\bot$ 
	can be parametrized by a smooth vector field $X$ on $F$, which is transverse to $\partial F$ pointing outwards. 
	After a $C^\infty$ perturbation of $F$ 
	keeping the transversality of $R_\lambda$ and $\mathring{F}$, we can assume that all singular points of $X$ are nondegenerate. 
	Let $o$ be the orientation on $F$ induced by the orientation of $\partial F$ in the direction of $R_\lambda$. 
	Let $o'$ be the orientation on $\xi$ given by $d\lambda|_\xi$. 
	A zero $p$ of $X$ is called positive if $o$ coincides with $o'$ at $p$, and negative otherwise.  
	Since $R_\lambda$ is transverse to $\mathring{F}$, all singularities of $X$ have the same sign and they must be positive since $\int_Fd\lambda=\int_{P_2}\lambda >0$. 
	Considering $X$ as a section of the bundle $\xi\to F$, we have 
	\begin{equation}\label{eq:eq1prova}
	\wind(\eta,\Psi)=\wind(X|_{\partial F},\Psi)=\sum_{X(z)=0} {\rm sign}\left(DX(z):(T_zF,o_z)\to (\xi_z,o'_z)\right).
	\end{equation}
	Let $\Psi_F:TF\to F\times \R^2$ be a trivialization of $TF\to F$. 
	Considering $X$ as a section of the bundle $TF\to F$ and using the fact that $X$ points outwards at $\partial F$, we have 
	\begin{equation}\label{eq:eq2prova}
	1=\wind(X|_{\partial F},\Psi_F)=\sum_{X(z)=0} {\rm sign}\left(DX(z):(T_zF,o_z)\to (T_zF,o_z)\right).
	\end{equation}
	Since all the singularities of $X$ are positive and $DX(z)$ at a singularity $z$ of $X$ does not depend on whether we consider $X$  as a section of $\xi\to F$ or as a section of $TF\to F$, equations \eqref{eq:eq1prova} and \eqref{eq:eq2prova} coincide.
	We conclude that 
	$\wind(\eta,\Psi)=1.$ 
	
	By the same arguments above, if $\nu$ is a section such that $\{\nu(\cdot), R(x_2(T_2\cdot))\}$ generates the tangent space of the disk $\bar{D}$ along $x_2(T_2 \cdot)$, where ${D}$ is the disk given by the definition of $3-2-1$ foliation \ref{de:3-2-1-foliation}, we have
	$\wind(\nu,\Psi)=1.$ 
	Since $T$ is transverse to $\bar{D}$ along $x_2$, we obtain
	$\wind(\eta',\Psi)=1~.$
\end{proof}
Now we are ready to prove Lemma \ref{LE:TRANSVERSAL-AO-LONGO-DE-P_2}.

\begin{proof}[Proof of Lemma \ref{LE:TRANSVERSAL-AO-LONGO-DE-P_2}]
	Let  $S^1=\R/\Z\ni t\mapsto \eta(t)\in \xi$ be a section along $x_2(T_2\cdot)$ such that $\{\eta(\cdot), R(x_2(T_2\cdot))\}$ generates $d\psi(TS^2)$ along $x_2(T_2\cdot)$ and
	\begin{equation}\label{eq:prova-lemma-strong-transverse-section-esfera}
		d\lambda(\eta(t),\mathcal{L}_{R_\lambda}\eta(t))\neq 0,~\forall t\in S^1.
	\end{equation}
	Let $u:(-\epsilon,\epsilon)\times S^1\to \Psi(S^2)$; $(s,t)\mapsto u(s,t)$ be a smooth function such that 
	\begin{equation*}
	\left\{
	\begin{array}{lr}
	u(0,\cdot)={x_2}_{T_2}\\
	\frac{\partial}{\partial s}u(s,\cdot)\big|_{s=0}=\eta
	\end{array}
	\right.
	\end{equation*}
	Fix an orientation on $S^1\times \{0\}\subset S^2$ in such a way that $\psi|_{S^1\times \{0\}}$ preserves orientation.
	Consider the closed hemispheres of $S^2$, that we call $H_+$ e $H_-$, with the orientation induced by the orientation of $S^1\times \{0\}$. 
	It follows that 
	$$0<T_2=\int_{S^1}{x_2}_{T_2}^*\lambda=\int_{H_\pm}\psi^*d\lambda~.$$
	Since $\psi$ is transverse to the Reeb flow $R_\lambda$ along $H_\pm\setminus (S^1\times\{0\})$, we have $\psi^*d\lambda>0$ on $H_\pm\setminus (S^1\times\{0\})$.
	This implies that $0$ is a local maximum of the function 
	$$(-\epsilon,\epsilon)\ni s\mapsto \mathcal{A}(u(s,\cdot))\in \R,$$
	where $\mathcal{A}:C^\infty(\R/\Z,S^3)\to \R$ is the action functional defined by $\mathcal{A}(\gamma)=\int_{\R/\Z}\gamma^*\lambda$.
	For any $J\in\scs$, we have 
	\begin{equation}\label{eq:hessiana-funcional-de-acao}
		\begin{aligned}
			\frac{d^2}{ds^2}(\mathcal{A}(u(s,\cdot)))\bigg|_{s=0}&=\int_{S^1}d\lambda\left(A_{P,J}(\eta),J\eta\right)dt,\\
		\end{aligned}
	\end{equation}
	where $A_{P,J}$ is the asymptotic operator defined by \eqref{eq:operador-assintotico-extensao}.
	For a proof, see \cite[\S 1.4]{deoliveira2020}.
	It follows from \eqref{eq:asymptotic-operator-independent} and  \eqref{eq:hessiana-funcional-de-acao} that
	\begin{equation}
	\frac{d^2}{ds^2}(\mathcal{A}(u(s,\cdot)))\bigg|_{s=0}=T_2\int_{S^1}d\lambda(\eta(t),\mathcal{L}_{R_\lambda}\eta(t))dt\leq 0~.
	\end{equation}
By \eqref{eq:prova-lemma-strong-transverse-section-esfera}, we have
	\begin{equation}\label{eq:dlambda-negativo}
		d\lambda(\eta(t),\mathcal{L}_{R_\lambda}\eta(t))< 0,~\forall t\in S^1.
	\end{equation}
	Let  $\R/\Z\ni t\mapsto \eta'(t)\in \xi$ be a section along $x_2(T_2\cdot)$ such that $\{\eta'(\cdot), R(x_2(T_2\cdot))\}$ generates the tangent space of $T$ along $x_2(T_2\cdot)$ and
	\begin{equation}\label{eq:prova-lema-stron-transverse-section-toro}
		d\lambda(\eta'(t),\mathcal{L}_{R_\lambda}\eta'(t))\neq 0,~\forall t\in S^1.
	\end{equation}
	Let $v:(-\epsilon,\epsilon)\times S^1\to T$, $(s,t)\mapsto v(s,t)$ be a smooth function such that 
	\begin{equation*}
	\left\{
	\begin{array}{lr}
	v(0,\cdot)={x_2}_{T_2}\\
	\frac{\partial}{\partial s}v(s,\cdot)\big|_{s=0}=\eta'.
	\end{array}
	\right.
	\end{equation*} 
	Recall that the cylinders $U_1$ and $U_2$ given by Definition \ref{de:3-2-1-foliation} are oriented in such a way that  $d\lambda|U_{1,2}$ is an area form, $P_3$ is a positive asymptotic limit and $P_2$ is a negative asymptotic limit. 
	This implies that $0$ is a local minimum of the function $s\mapsto \mathcal{A}(v(s,\cdot))$.
	It follows from \eqref{eq:asymptotic-operator-independent} and  \eqref{eq:hessiana-funcional-de-acao} that
	\begin{equation}
	\frac{d^2}{ds^2}(\mathcal{A}(v(s,\cdot)))\bigg|_{s=0}=T_2\int_{S^1}d\lambda(\eta'(t),\mathcal{L}_{R_\lambda}\eta'(t))dt\geq 0~.
	\end{equation}
	By \eqref{eq:prova-lema-stron-transverse-section-toro}, we have
	\begin{equation}\label{eq:dlambda-positivo}
		d\lambda(\eta'(t),\mathcal{L}_{R_\lambda}\eta'(t))> 0,~\forall t\in S^1.
	\end{equation}
	Now we can apply Lemma \ref{le:wind=1} and Proposition \ref{pr:secoes-fortes-wind-1} to the sections $\eta$ and $\eta'$ and conclude that $\psi$ is transverse to the torus $T$ along $P_2$. 
\end{proof}

\section{Proof of Theorem \ref{pr:example}}\label{se:proposition-pre-exemple}
In this section, we prove Theorem \ref{pr:example}, which is restated below.
\begin{theorem}
	Let $\lambda$ be a tight contact form on $S^3$.
	Assume that there exist Reeb orbits $P_1=(x_1,T_1),~P_2=(x_2,T_2),~P_3=(x_3,T_3)\in \mathcal{P}(\lambda)$ that are nondegenerate, simple, and have Conley-Zehnder indices respectively $1$, $2$ and $3$. 
	Assume further that the orbits $P_1$, $P_2$, and $P_3$ are unknotted, $P_i$ and $P_j$ are not linked for $i\neq j$, $i,j\in\{1,2,3\}$, and the following conditions hold: 
	\begin{enumerate}[label=(\roman*)]
		\item $T_1<T_2<T_3<2T_1$;
		\item If $P=(x,T)\in\mathcal{P}(\lambda)$ satisfies $P\neq P_3,~T\leq T_3$ and  ${\rm lk}(P,P_3)=0$, then $P\in\{P_1,P_2\}$.
		\item There exists $J\in\mathcal{J}(\lambda)$ such that the almost complex structure $\tj=(\lambda,J)$ admits a finite energy plane $\tu:\C\to \R\times S^3$ asymptotic to $P_3$ at it positive puncture $z=\infty$ and a finite energy clylinder 
		$\tw:\C\setminus \{0\}\to \R\times S^3$ asymptotic to $P_3$ at its positive puncture $z=\infty$ and $P_1$ at its negative puncture $z=0$;
		\item There exists no $C^1$-embedding $\Psi:S^2\subset \R^3\to S^3$ such that $\Psi({S^1\times \{0\}})=x_2(\R)$ and 
		each hemisphere is a strong transverse section.
	\end{enumerate}	
	Then there exists a $3-2-1$ foliation adapted to $\lambda$ with binding orbits $P_1$, $P_2$, and $P_3$. Consequently, there exists at least one homoclinic orbit to $P_2$. 
\end{theorem}			
\begin{proof}[Proof of Theorem \ref{pr:example}] 
	By the same arguments used in \S\ref{se:a-family-of-planes}, we find a maximal one-parameter family of finite energy $\tj$-holomorphic  planes 
	\begin{equation}\label{eq:familia-de-planos-exemplo}
		\tilde{u}_\tau=(a_\tau, u_\tau):\C\to \R\times S^3,~~\tau\in (0,1)
	\end{equation}
	asymptotic to the orbit $P_3$. For each $\tau\in (0,1)$, $u_\tau$ is an embedding transverse to the Reeb flow and $u_{\tau_1}(\C)\cap u_{\tau_2}(\C)=\emptyset, ~\forall \tau_1\neq\tau_2.$
	We assume that 
	$\tau$ strictly increases in the direction of $R_\lambda$.
	

	Now we describe how the family $\{\tu_\tau\}$ breaks as $\tau \to 0^+$ and $\tau\to 1^-$. 
	Consider a sequence $\tau_n\in(0,1)$ satisfying $\tau_n\to 0^+$, and define $\tilde{u}_n:=\tilde{u}_{\tau_n}$. 
	\paragraph{\textit{Claim I}}
	There exists a $\tilde{J}$-holomorphic finite energy cylinder 
	\begin{equation}\label{eq:cilindro-ur-exemplo}
		\tu_r=(a_r,u_r):\C\setminus \{0\}\to \R\times S^3
	\end{equation}
	asymptotic to $P_3$ at its positive puncture $z=\infty$ and to $P_2$ at its negative puncture $z=0$, and a finite energy $\tilde{J}$-holomorphic plane 
	\begin{equation}
		\tu_q=(a_q,u_q):\C\to\R\times S^3
	\end{equation}
	asymptotic to $P_2$ at its positive puncture $z=\infty$, such that, after a suitable reparametrization and an $\R$-translations of $\tu_n$, the following hold
	\begin{enumerate}[label=(\roman*)]
		\item up to a subsequence, 
		$\tu_n\to \tu_r$ in $C^\infty_{loc}(\C\setminus \{0\})$ as $n\to \infty$.
		\item There exist sequences $\delta_n\to 0^+$, $z_n\in \C$ and $c_n\in \R$ such that, up to a subsequence,
		$\tu_n(z_n+\delta_n \cdot)+c_n\to \tu_q$ 	in $C^\infty_{loc}(\C)$ as $n\to \infty$.
		\item Given an $S^1$-invariant neighborhood $\mathcal{W}_3\subset C^\infty(\R/\Z,S^3)$ of the loop $t\mapsto x_3(T_3t)$, there exists $R_0>>1$ such that  the loop $t\mapsto u_n(Re^{2\pi it})$ belongs to $\mathcal{W}_3$, for $R\geq R_0$ and large $n$.
		\item Given an $S^1$-invariant neighborhood $\mathcal{W}_2\subset C^\infty(\R/\Z,S^3)$ of the loop $t\mapsto x_2(T_2t)$, there exists $\epsilon_1>0$ and $R_1>>0$ such that the loop $t\mapsto u_n(z_n+ Re^{2\pi it})$ belongs to $\mathcal{W}_2$, for $R_1\delta_n\leq R\leq \epsilon_1$ and large $n$.
		\item Given any neighborhood $\mathcal{V}$ of $u_r(\C\setminus\{0\})\cup u_q(\C)\cup P_2\cup P_3$, we have $u_n(\C)\subset \mathcal{V}$, for large $n$.  	
	\end{enumerate} 
	Here $(a,x)+c:=(a+c,x),~\forall (a,x)\in \R\times S^3,~c\in \R$.
	{A similar claim holds for any sequence ${\tau_n}\in (0,1)$ satisfying $\tau_n\to 1^-$. In this case we change the notation from $\tu_r$ and $\tu_q$ to $\tu_r'$ and $\tu_q'$ respectively. }
	
	To prove \textit{Claim I}, define $\gamma$ as any real number satisfying 
	\begin{equation}\label{eq:defi-gamma-exemplo}
		0<\gamma<\min\{T_1,T_2-T_1,T_3-T_2\}.
	\end{equation}
	After reparametrizing and translating $\tu_n$ in the $\R$ direction, we can assume that
	\begin{align}
		\int_{\C\setminus \D}u_n^*d\lambda=\gamma,~\forall n\in \N \label{eq:germinante1-exemplo}\\
		a_n(2)=0,~\forall n \in \N \label{eq:germinante2-exemplo}\\
		a_n(0)=\inf a_n(\C),~\forall n \in \N\label{eq:germinante3-exemplo}
	\end{align}
	Let $\Gamma\subset \C$ be the set of points $z\in \C$ such that there exist subsequence $\tu_{n_j}$ and sequence $z_j\to z$ satisfying $|d\tu_{n_j}(z_j)|\to \infty$. 
	
	If $\Gamma=\emptyset$, by elliptic estimates we find a $\tj$-holomorphic map $\tu_r:\C\to \R\times S^3$ such that, up to a subsequence, $\tu_n\to \tu_r$ in $C^\infty_{loc}(\C,\R\times S^3)$. 
	Using Fatou's Lemma we conclude that $E(\tu_r)\leq T_3$. By \eqref{eq:germinante1-exemplo} we have
	\begin{equation}
		\int_{\partial \D} u_r^*d\lambda=\lim_{n\to \infty}\int_{\partial \D}u_n^*\lambda=\lim_{n\to \infty}\int_{\D} u_n^*d\lambda= T_3-\gamma>0.
	\end{equation}
	Therefore, $\tu_r$ is nonconstant.
	
	If $\Gamma\neq \emptyset$, let $z\in \Gamma$ and let $z_n\to z$ be such that, passing to a subsequence still denoted by $\tu_n$, we have $|d\tu_n(z_n)|\to \infty$. 
	Consider a sequence $r_n\to 0^+$ such that $r_n|d\tu_n(z_n)|\to \infty$. 
	{We need the following lemma.
		\begin{lemma}[{\cite[Lemma 3.3]{hofer1992}}]\label{le:hofer-lemma}
			Let $(X,d)$ be a complete metric space and $f:X\to [0,+\infty)$ a continuous function. 
			For any $\epsilon_0>0$ and $x_0\in X$, there exist $\epsilon'_0\in (0,\epsilon_0]$ and $x'_0\in \overline{B_{2\epsilon_0}(x_0)}$ such that 
			\begin{equation*}
				\left\{
				\begin{array}{lr}
					f(x_0')\epsilon_0'\geq f(x_0)\epsilon_0\\
					d(x,x_0')\leq \epsilon_0'\Rightarrow f(x)\leq 2f(x_0')
				\end{array}	
				\right.	.
			\end{equation*} 
		\end{lemma}}
	By Lemma \ref{le:hofer-lemma}, we have, after perhaps modifying $r_n$ and $z_n$,
	\begin{equation}\label{eq:derivada-un-exemplo1}
		|d\tu_n(z)|\leq 2|d\tu_n(z_n)|,~\text{ for } z\in B_{r_n}(z_n).
	\end{equation}
	Denoting $\delta_n=|d\tu_n(z_n)|^{-1}$, define the sequence of $\tj$-holomorphic maps $\tv_n=(b_n,v_n):B_{\frac{r_n}{\delta_n}}(0)\to \R\times S^3$ by
	\begin{equation}\label{eq:defi-vn-bubbling-exemplo}
		\tv_n(z)=(a_n(z_n+\delta_nz)-a_n(z_n),u_n(z_n+\delta_nz)).
	\end{equation}
	From \eqref{eq:derivada-un-exemplo1} and \eqref{eq:defi-vn-bubbling-exemplo} we get
	\begin{align*}
		\tv_n(0)\in \{0\}\times S^3,~\forall n\in \N\\
		|d\tu_n(z)|\leq 2, ~\forall z\in B_{\frac{r_n}{\delta_n}}(0).
	\end{align*}
	By elliptic estimates, there exists a subsequence, still denoted by $\tv_n$ and a $\tj$-holomorphic map $\tv:\C\to \R\times S^3$ such that $\tv_n\to \tv$ in $C^\infty_{loc}(\C,\R\times S^3)$. 
	Since $|d\tv(0)|=\lim_{n\to \infty}|d\tv_n(0)|=1$, we know that $\tv$ is nonconstant. Since $E(\tv_n)\leq E(\tu_n)=T_3$, we have $E(\tv)\leq T_3$. 
	Any asymptotic limit of the plane $\tv$ is not linked to $P_3$ and has action $\leq T_3$. We conclude that $\tv$ is asymptotic to either $P_1$, $P_2$ or $P_3$.  Now we show that $\tv$ is asymptotic either to $P_3$ or to $P_2$. 
	Suppose, by contradiction, that $\tv$ is asymptotic to $P_1$. From equation \eqref{eq:conley-zehnder-wind}, Lemma \ref{le:wind-infty-estimate} and $\mu(P_1)=1$, we have $\wind_\infty(\tv,\infty)\leq 0$. Using equation \eqref{eq:wind_pi-wind_infty}, we get the contradiction 
	$$0\leq \wind_\pi(\tv)=\wind_\infty(\tv)-1\leq -1.$$
	
	Using Fatou's Lemma and passing to a subsequence, still denoted by $\tu_n$, we get
	\begin{equation}\label{eq:nao-tem-energia-suficiente-exemplo}
		T_2\leq \int_\C v^*d\lambda \leq 
		\lim_{n\to \infty}\int_{B_{r_n}(z_n)}u_n^*d\lambda\leq \int_\C u_n^*d\lambda =T_3.
	\end{equation}
	Since $T_3<2T_2$, we conclude from \eqref{eq:nao-tem-energia-suficiente-exemplo} that $\Gamma=\{z\}$. 
	From \eqref{eq:germinante2-exemplo} and elliptic estimates, we find a $\tj$-holomorphic map
	$$\tu_r:\C\setminus \{z\}\to \R\times S^3$$
	such that, up to a subsequence, $\tu_n\to\tu_r$ and $E(\tu_r)\leq T_3$.
	The puncture $z$ is negative. Indeed, for any $\epsilon>0$, we have 
	\begin{equation}\label{eq:massa-furo-negativo-bubbling-exemplo}
		m_\epsilon(z):=\int_{\partial B_\epsilon(z)}u_r^*\lambda =\lim_{n\to \infty}\int_{\partial B_\epsilon(z)}u_n^*\lambda=\lim_{n\to \infty}\int_{B_\epsilon(z)}u_n^*d\lambda\geq T_2>0, 
	\end{equation}
	where $B_\epsilon(z)$ is oriented counterclockwise. Here we have used \eqref{eq:nao-tem-energia-suficiente-exemplo}.
	The puncture $\infty$ is necessarily positive since $0<E(\tu_r)\leq T_3$. Using \eqref{eq:germinante3-exemplo} 
	we conclude that $z=0$. 
	
	
	Now we show that $\tu_r$ is asymptotic to $P_3$ at its positive puncture $z=\infty$ and to $P_2$ at its negative puncture $z=0$. 
	We need the following lemma, which is an adaptation of Lemma 4.9 from \cite{hwz2003} to our set-up. 
	\begin{lemma}\label{le:cylinders-small-area-exemplo}
		Consider a constant $e>0$ and let $\gamma$ be defined by \eqref{eq:defi-gamma-exemplo}.
		Identifying $S^1=\R/\Z$, let $\mathcal{W}\subset C^\infty(S^1,S^3)$ be an open neighborhood of the periodic orbits $P_1$, $P_2$ and $P_3$, viewed as maps $x_T:S^1\to S^3$, $x_T(t)=x(Tt)$. We assume that $\mathcal{W}$ is $S^1$-invariant, meaning that $y(\cdot+c)\in \mathcal{W}\Leftrightarrow y\in \mathcal{W}, \forall c\in S^1$, and that each of the connected components of $\mathcal{W}$ contains at most one periodic orbit modulo $S^1$-reparametrizations.  	
		Then there exists a constant $h>0$ such that the following holds. 
		If $\tu=(a,u):[r,R]\times S^1 \to \R\times S^3$ is a $\tj$-holomorphic cylinder such that the image of $\tu$ does not intersect $P_3$, $u(r,\cdot)$ is not linked to $P_3$, and 
		\begin{equation*}\label{eq:cylinders-small-area-exemplo}
			E(\tu)\leq T_3,~~~\int_{[r,R]\times S^1}u^*d\lambda\leq \gamma,~~~~\int_{\{r\}\times S^1}u^*\lambda\geq e~~~~\text{and}~~ r+h\leq R-h,
		\end{equation*} 
		
		then each loop $t\in S^1\to u(s,t)$ is contained in $\mathcal{W}$ for all $s\in [r+h,R-h]$.
	\end{lemma}
	\begin{proof}
		Arguing by contradiction, we find a sequence of $\tj$-holomorphic maps 
		$$\tu_n=(a_n,u_n):[r_n,R_n]\times S^1\to \R\times S^3$$
		 such that the image of $\tu_n$ does not intersect $P_3$, $u_n(r_n,\cdot)$ is not linked to $P_3$,   
		\begin{equation}\label{eq:propriedade-lemma-exemplo}
			E(\tu_n)\leq T_3,~~~~\int_{[r_n,R_n]\times S^1}u_n^*d\lambda\leq \gamma,~~~~\int_{\{r_n\}\times S^1}u^*\lambda\geq e,~~~~ r_n+n\leq R_n-n
		\end{equation}
		and 
		\begin{equation}\label{eq:nao-esta-na-viz-lemma}
			u_n(s_n,\cdot )\notin \mathcal{W} \text{ for some } s_n\in [r_n+n,R_n-n].
		\end{equation}
		Define a sequence of $\tj$-holomorphic maps $\tv_n=(b_n,v_n):[r_n-s_n,R_n-s_n]\times S^1\to \R\times S^3$ by  
		$$\tv_n(s,t)=(a_n(s+s_n,t)-a_n(s_n,0),u_n(s+s_n,t)).$$ 
		Note that $b_n(0,0)=0, \forall n\in \N$ and by \eqref{eq:nao-esta-na-viz-lemma} we have
		\begin{equation}\label{eq:nao-esta-na-viz-lemma-2}
			v_n(0,\cdot )\notin \mathcal{W}.
		\end{equation}
		We claim that $|d\tv_n(s,t)|$ is uniformly bounded in $n\in \N$ and $(s,t)\in\R\times S^1$. Conversely, suppose that there exist a subsequence $\tv_{n_j}$ and a sequence $z_j\to z\in \C$ satisfying $|d\tv_{n_j}(z_j)|\to \infty$. Arguing as in the proof of \eqref{eq:nao-tem-energia-suficiente-exemplo}, we get
		$T_2\leq \lim_{n\to \infty}\int_{B_{r_j}(z_j)}v_{n_j}^*d\lambda$, for some sequence $r_j\to 0^+$. This contradicts
		$$\int_{[r_n-s_n,R_n-s_n]\times S^1}v_n^*d\lambda=\int_{[r_n,R_n]\times S^1}u_n^*d\lambda\leq \gamma.$$
		Thus, passing to a subsequence, still denoted by $\tv_n$, we have $\tv_n\to \tv$ in $C^\infty_{loc}(\R\times S^1,\R\times S^3)$, where $\tv:\R\times S^1\to \R\times S^3$ is a finite energy cylinder. Using \eqref{eq:propriedade-lemma-exemplo}, we conclude that 
		\begin{equation*}
			E(\tv)\leq T_3,~~~\int_{\R\times S^1}v^*d\lambda\leq \gamma ~~\text{ and }~\int_{\{s\}\times S^1}u^*\lambda\geq e,\forall s\in \R.
		\end{equation*}
		Consequently, $\tv$ is nonconstant, has a positive puncture at $s=+\infty$, a negative puncture at $s=-\infty$ and any asymptotic limit of $\tv$ at $s=+\infty$ has period $\leq T_3$. 
		Since the loops $\tv_n(s, \cdot)$ are not linked to $P_3$, the asymptotic limits of $\tv$ are not linked to $P_3$ as well. Therefore, $\tv$ is asymptotic to either $P_1$, $P_2$ or $P_3$ at $s=+\infty$. 
		Since $\int_{\R\times S^1}v^*d\lambda\leq \gamma$, we conclude that $\tv$ is a cylinder over either $P_1$, $P_2$ or $P_3$. However, by \eqref{eq:nao-esta-na-viz-lemma-2}, we know that $v(0,\cdot)\notin \mathcal{W}$, a contradiction. 
	\end{proof}

	Any asymptotic limit $P_+=(x_+,T_+)$ of $\tu_r$ at $z=\infty$ is not linked to $P_3$ and satisfies $T_+\leq T_3$. Thus, $\tu_r$ is asymptotic to either $P_1$, $P_2$ or $P_3$ at $z=\infty$. 
	Let $\mathcal{W}$ be as in the statement of Lemma \ref{le:cylinders-small-area-exemplo}.
	Using \eqref{eq:germinante1-exemplo} and Lemma \ref{le:cylinders-small-area-exemplo}, we conclude that for each $n\in \N$ and large $s$,  $\{t\mapsto u_n(s,t)\}\in \mathcal{W}$ and $\{t\mapsto u_r(s,t)\}\in \mathcal{W}$. Since the planes $\tu_n$ are asymptotic to $P_3$ and for each fixed $s$, $\{t\mapsto u_n(s,t)\}\in \mathcal{W}$ converges to $\{t\mapsto u_r(s,t)\}\in \mathcal{W}$, as $n\to \infty$, we 
	conclude that $\tu_r$ is asymptotic to $P_3$ at $z=\infty$. 
	Consequently we have $\Gamma\neq \emptyset$, since $\Gamma=\emptyset$ would contradict the fact that the family \eqref{eq:familia-de-planos-exemplo} is maximal. 
	
	Using \eqref{eq:germinante1-exemplo}, we conclude that 
	\begin{equation}\label{eq:massa-furo-positivo-bubbling-exemplo}
		\int_{\partial D}u_r^*\lambda=\lim_{n\to \infty}\int_{\partial\D}u_n^*\lambda=\lim_{n\to \infty}\int_\D u_n^*d\lambda=T_3-\gamma.
	\end{equation}
	Therefore, 
	any asymptotic limit $P_-=(x_-,T_-)\in \mathcal{P}(\lambda)$ of $\tu_r$ at $z=0$ is not linked to $P_3$ and satisfies $T_2\leq T_-<T_3$. Here we have used \eqref{eq:massa-furo-negativo-bubbling-exemplo} and \eqref{eq:massa-furo-positivo-bubbling-exemplo}. 
	We conclude that $P_-=P_2$. 
	

	Now we proceed as in the \textit{soft rescaling} done in \S\ref{se:soft-rescaling}. 
	Let $m_\epsilon(0)$ be defined as in \eqref{eq:def-mepisilonz}. Since $m_\epsilon(0)$ is a nondecreasing function of $\epsilon$, we can fix $0<\epsilon<<1$ so that 
	\begin{equation}\label{eq:t2-mepsilon-exemplo}
		0\leq m_\epsilon(0)-T_2\leq\frac{\gamma}{2}.
	\end{equation}
	Arguing as in \S\ref{se:soft-rescaling}, we
	choose sequences $z_n\in \overline{B_\epsilon(0)}$ and $0<\delta_n<\epsilon$ satisfying
	\begin{align}
		a_n(z_n)\leq a_n(\zeta),~\forall \zeta \in B_\epsilon(0) \label{eq:min-an-exemplo},\\
		\int_{B_\epsilon(0)\setminus B_{\delta_n}(z_n)} u_n^*d\lambda =\gamma \label{eq:int-igual-gamma-exemplo},
	\end{align}
	such that $z_n\to 0$ and $\liminf \delta_n= 0$. 
	Thus, passing to a subsequence, we can assume $\delta_n\to 0$.
	
	Take any sequence $R_n\to +\infty$ satisfying
	$\delta_nR_n<\frac{\epsilon}{2}~$
	and define the sequence of $\tj$-holomorphic maps $\tilde{w}_n=(c_n,w_n):B_{R_n}(0)\to \R\times S^3$ by
	\begin{equation}\label{eq:definition-wn-exemplo}
		\tilde{w}_n(\zeta)=(a_n(z_n+\delta_n\zeta)-a_n(z_n+2\delta_n), u_n(z_n+\delta_n\zeta))~.
	\end{equation}
	Note that $E(\tw_n)\leq E(\tu_n)= T_3$ and $\tw_n(2)\in \{0\}\times S^3$. 
	Let $\Gamma_q$ be the set of points $z\in \C$ such that there exist subsequence $\tw_{n_j}$ and sequence $z_j\to z$ satisfying $|d\tw_{n_j}(z_j)|\to \infty$. 
	Then there exists a $\tj$-holomorphic map $\tu_q:\C\setminus \Gamma_q\to \R\times S^3$ such that, up to a subsequence, $\tw_n\to \tu_q$ in $C^\infty_{loc}(\C\setminus \Gamma_q,\R\times S^3)$.
	
	We claim that $\tu_q$ is nonconstant and asymptotic to $P_2$ at its positive puncture $z=\infty$. 
	Indeed, using \eqref{eq:int-igual-gamma-exemplo} and \eqref{eq:definition-wn-exemplo} we conclude that, for any $r\geq1$,
	\begin{equation}
			\int_{\partial B_r(0)}u_q^*\lambda
			= \lim_{n\to \infty}\int_{\partial B_{r\delta_n}(z_n)}u_n^*\lambda
			\geq \lim_{n\to \infty}\int_{B_{\delta_n}(z_n)}u_n^*d\lambda
			\geq T_2-\gamma,
	\end{equation}
	and by \eqref{eq:germinante1-exemplo}, we have 
	\begin{equation}
		\int_{\partial B_r(0)}u_q^*\lambda
		= \lim_{n\to \infty}\int_{\partial B_{r\delta_n}(z_n)}u_n^*\lambda
		\leq \lim_{n\to \infty}\int_{B_{\delta_nR_n}(z_n)}u_n^*d\lambda
		\leq T_3-\gamma.
	\end{equation}
	Here we have used $\delta_nR_n<\frac{\epsilon}{2}<\frac{1}{2}$ and $z_n\to 0$.
	Therefore, any asymptotic limit $P_+=(x_+,T_+)$ of $\tu_q$ at $\infty$ is not linked to $P_3$ and satisfies $P_2-\gamma\geq T_+\geq T_3- \gamma$. We conclude that $P_+=P_2$. 
	
	Now we prove that $\Gamma_q=\emptyset$. Suppose, by contradiction, that $z\in \Gamma_q$ and let $z_n\to z$ be such that, passing to a subsequence still denoted by $\tw_n$, $|d\tw_n(z_n)|\to \infty$. 
	Arguing as in the proof of \eqref{eq:nao-tem-energia-suficiente-exemplo}, we find a sequence $r_n\to 0^+$ such that, 
	for each sufficiently small $\epsilon>0$, we have
	$$\int_{\partial B_\epsilon (z)}u_q^*\lambda=\lim_{n\to \infty}\int_{\partial B_\epsilon (z)}w_n^*\lambda\geq \lim_{n\to\infty}\int_{B_{r_n}(z_n)}w_n^*d\lambda\geq T_2.$$
	We conclude that $\int_{\C\setminus \Gamma_q}u_q^*d\lambda=0$ and $\Gamma_q=\{z\}=\{0\}$.
	This implies that $\tu_q$ is a cylinder over the orbit $P_2$  and we get the contradiction
	$$T_2=\int_{\partial\D}u_q^*\lambda=\lim_{n\to \infty}\int_{\partial\D}w_n^*\lambda=\lim_{n\to \infty}\int_{B_{\delta_n}(z_n)}u_n^*d\lambda=m_\epsilon(z)-\gamma\leq T_2-\frac{\gamma}{2}.$$
	Here we have used \eqref{eq:t2-mepsilon-exemplo} and \eqref{eq:int-igual-gamma-exemplo}. 	
	{We have proved (i) and (ii).}
		
	{An application of Lemma \ref{le:cylinders-small-area-exemplo}, similar to the proof of Proposition \ref{pr:familia-de-planos-se-aproxima-limite}, proves (iii) and (iv). Assertion (v) is a consequence of (i)-(iv).} \textit{Claim I} is proved. 
		Following  the same arguments as in the proof of Propositions \ref{le:the-cylinders-do-not-intersect-orbits}-\ref{pr:foliation}, we conclude that
	\begin{itemize}
		\item Up to reparametrization, $u_q=u'_q$;
		\item $u_r(\C\setminus \{0\})\cap u_r'(\C\setminus \{0\})=\emptyset$;
		\item The union of the image of the family \eqref{eq:familia-de-planos-exemplo} with the images of $u_q$, $u_r$, $u'_r$, $x_2$ and  $x_3$ determine a singular foliation of a closed region $\mathcal{R}_1$, homeomorphic to a solid torus, such that $\partial \mathcal{R}_1=T$, where $T=P_2\cup P_3\cup u_r(\C\setminus \{0\})\cup u'_r(\C\setminus \{0\})$.
	\end{itemize}
	
	Now we find the foliation of $\mathcal{R}_2=S^3\setminus \overline{\mathcal{R}_1}$.
	Following the same arguments as in Proposition \ref{pr:cilindro-e-mergulho} we prove that the finite energy cylinder $\tw$  is an embedding. 
	Applying Theorem \ref{th:family-of-cylinders-wendl} to the finite energy cylinder $\tw$, we obtain a maximal one-parameter family of finite energy cylinders 
	\begin{equation}\label{eq:familia-de-cilindros-exemplo}
		\tilde{w}_\tau=(c_\tau, w_\tau):\C\setminus \{0\} \to\R\times S^3,~~\tau\in (0,1)
	\end{equation}
	asymptotic to $P_3$ at its positive puncture $z=\infty$ and to $P_1$ at its negative puncture $z=\{0\}$. 
	For each $\tau\in (0,1)$, $\tw_\tau$  is an embedding, the projection $w_\tau:\C\setminus\{0\}\to S^3$  is an embedding which does not intersect its asymptotic limits, and $w_\tau(\C\setminus \{0\})\cap \mathcal{R}_1=\emptyset$. 
	We assume that 
	$\tau$ strictly increases in the direction of $R_\lambda$.
	
	Consider a sequence 
	$\tau_n\in (0,1)$ satisfying $\tau_n\to 1^-$ and define $\tw_n:=\tw_{\tau_n}$.  
	
	\paragraph{\textit{Claim II}}
	There exists a finite energy cylinder 
	\begin{equation}
		\tw_r:\C\setminus \{0\}\to \R\times S^3
	\end{equation}
	 asymptotic to $P_3$ at its positive puncture $z=\infty$ and to $P_2$ at its negative puncture $z=0$, and a finite energy cylinder 
	 \begin{equation}
	 	\tw_q:\C\setminus \{0\}\to \R\times S^3
	 \end{equation}
	 asymptotic to $P_2$ at its positive puncture $z=\infty$ and to $P_1$ at its negative puncture $z=0$, such that,  after suitable reparametrizations and $\R$-translations of $\tw_n$, we have
	\begin{enumerate}[label=(\roman*)]
		\item up to a subsequence, 
		$\tw_n\to \tw_r$ in $C^\infty_{loc}(\C\setminus \{0\})$ as $n\to \infty$.
		\item There exist sequences $\delta_n^+\to 0^+$ and $d_n\in \R$ such that, up to a subsequence,
		$\tw_n(\delta_n \cdot)+d_n\to \tw_q$ 	in $C^\infty_{loc}(\C\setminus \{0\})$ as $n\to \infty$. 
		\item Given an $S^1$-invariant neighborhood $\mathcal{W}_3 \subset C^\infty(\R/\Z,S^3)$ of the loop $t\mapsto x_3(T_3t)$, there exists $R_3>>1$ such that  the loop $t\mapsto w_n(Re^{2\pi it})$ belongs to $\mathcal{W}_3$, for $R\geq R_3$ and large $n$.
		\item Given an $S^1$-invariant neighborhood $\mathcal{W}_2\subset C^\infty(\R/\Z,S^3)$ of the loop $t\mapsto x_2(T_2t)$, there exists $\epsilon_2>0$ and $R_2>1$ such that the loop $t\mapsto w_n(Re^{2\pi it})$ belongs to $\mathcal{W}_2$, for $R_2\delta_n\leq R\leq \epsilon_2$ and large $n$.
		\item Given an $S^1$-invariant neighborhood $\mathcal{W}_1\subset C^\infty(\R/\Z,S^3)$ of the loop $t\mapsto x_1(T_1t)$, there exists $\epsilon_1>0$ such that the loop $t\mapsto w_n(\rho e^{2\pi it})$ belongs to $\mathcal{W}_1$, for $\rho\leq \epsilon_1\delta_n$ and large $n$.
		\item Given any neighborhood $\mathcal{V}\subset \mathcal{R}_2$ of $w_r(\C\setminus\{0\})\cup w_q(\C\setminus \{0\})\cup P_1 \cup P_2\cup P_3$, we have $w_n(\C\setminus \{0\})\subset \mathcal{V}$, for large $n$.  		
	\end{enumerate} 
	A similar claim holds for any sequence $\tau_n\to 0^+$ with $\tw_r$ replaced with a cylinder $\tw_r'$ with the same asymptotics as $\tw_r$ and $\tw_q$ replaced with a cylinder $\tw_q'$ with the same asymptotics as $\tw_q$.

	After a reparametrization and $\R$-translation of $\tw_n$, we can assume that
	\begin{align}
		\int_{\C\setminus\D}w_n^*d\lambda=\frac{\gamma}{2},~\forall n\in\N  \label{eq:germination-cilindro-exemplo1}\\
		c_n(2)=0,~\forall n\in \N\label{eq:germinating-cilindro-exemplo-2},
	\end{align}
	where $\gamma$ is defined by \eqref{eq:defi-gamma-exemplo}.
	Observe that  $|d\tw_n(z)|$ is uniformlly bounded in $n\in \N$ and $z\in \C\setminus \{0\}$. 
	Otherwise, arguing as in the proof of \eqref{eq:nao-tem-energia-suficiente-exemplo}, we would find sequences $z_n\in \C$, $r_n\to 0^+$ and a subsequence of $\tw_n$, still denoted $\tw_n$, such that 
	$$T_2=\int_\C v^*d\lambda \leq \lim_{n\to \infty}\int_{B_{r_n}(z_n)}w_n^*d\lambda\leq \lim_{n\to \infty}\int_\C w_n^*d\lambda=T_3-T_1,$$
	contradicting the hypothesis $T_3<2T_1$. 
	We conclude that there exists a finite energy $\tj$-holomorphic cylinder 
	$\tw_r=(c_r,w_r):\C\setminus \{0\}\to \R\times S^3$
	such  that, up to subsequence, still denoted by $\tw_n$, 
	$\tw_n\to \tw_r~\text{ in }C^\infty_{loc}(\C\setminus\{0\},\R\times S^3).$ 
	
	The finite energy cylinder $\tw_r$ is nonconstant, $z=0$ is  a negative puncture and $z=\infty$ is a positive puncture. Indeed, for any $r>0$, we have 
		$\int_{\partial B_r(0)}w_r^*\lambda=\lim_{n\to \infty}\int_{\partial B_r(0)}w_n^*\in(T_1,T_3),$ 
	where $\partial B_r(0)$ is oriented counterclockwise. 
	
	It follows from \eqref{eq:germination-cilindro-exemplo1} that  
	$\int_{\C\setminus\{0\}}w_r^*d\lambda \geq\frac{\gamma}{2}>0$. 
	Using Lemma \ref{le:cylinders-small-area-exemplo} as in the proof of \textit{Claim I}, we conclude that $\tw_r$ is asymptotic to $P_3$ at the puncture $z=\infty$. 
	Any asymptotic limit of $\tw_r$ at $z=0$ is not linked to $P_3$ and has period $<T_3$. Thus, $\tw_r$ is asymptotic to either $P_2$ or $P_1$ at $z=0$. 
	Since the family \eqref{eq:familia-de-cilindros-exemplo} is maximal, we conclude that $\tw_r$ is asymptotic to $P_2$ at the negative puncture $z=0$. 
	
	Now we proceed as in the soft rescaling done in \S \ref{se:soft-rescaling-cilindros}. Arguing as in \S \ref{se:soft-rescaling-cilindros}, we find a sequence 
	$\delta_n>0$ satisfying 
	\begin{equation}\label{eq:existencia-seq-tau-n-cilindro-exemplo}
		\int_{{B_{\delta_n}(0)}\setminus\{0\}}w_n^*d\lambda=(T_2-T_1)-\frac{\gamma}{2}
	\end{equation}
and such that $\liminf \delta_n=0$.
	Thus, passing to a subsequence still denoted by $\tw_n$, we can assume that $\delta_n\to 0$.
	
	Fix $\epsilon>0$ such that 
	\begin{equation}\label{eq:epsilon-pequeno-exemplo}
		\int_{\partial B_\epsilon(0)}w_r^*\lambda\leq T_2+\frac{\gamma}{4}
	\end{equation}
	and define the sequence of $\tj$-holomorphic maps $\tv_n=(b_n,v_n):B_{\frac{\epsilon}{\delta_n}}(0)\setminus\{0\}\to \R\times S^3$ by
	\begin{equation}\label{eq:defi-vn-cilindros-exemplo}
		\tv_n(z)=(c_n(\delta_nz)-c_n(2\delta_n),w_n(\delta_n z)).
	\end{equation}
	Using \eqref{eq:existencia-seq-tau-n-cilindro-exemplo} and \eqref{eq:epsilon-pequeno-exemplo}, we conclude that, for large $n$,
	\begin{equation}\label{eq:vn-germinating-cilindros-exemplo}
		\begin{aligned}
		\int_{B_{\frac{\epsilon}{\delta_n}}(0)\setminus\D}v_n^*d\lambda
		&=\int_{B_{\epsilon}(0)\setminus\{0\}}w_n^*d\lambda- \int_{B_{\delta_n}(0)\setminus \{0\}}w_n^*d\lambda\\
		&=\int_{\partial B_\epsilon(0)}w_n^*\lambda-T_1-\left(T_2-T_1-\frac{\gamma}{2}\right)\\
		&\leq T_2+\frac{\gamma}{2}-T_1 - \left(T_2-T_1-\frac{\gamma}{2}\right)=\gamma
		\end{aligned}
	\end{equation} 
	Note that $E(\tv_n)\leq T_3$ and $\tv_n(2)\in \{0\}\times S^3$.
	 Moreover, $|d\tv_n(z)|$ is uniformly bounded on $n\in \N$ and $z\in \C\setminus\{0\}$. Otherwise, arguing as in the proof of \eqref{eq:nao-tem-energia-suficiente-exemplo}, we would find sequences $z_n\to z\in \C\setminus\{0\}$, $r_n\to 0^+$ and a subsequence of $\tv_n$, still denoted by $\tv_n$, such that 
	$$T_2\leq \lim_{n\to \infty}\int_{B_{r_n}(z_n)}v_n^*d\lambda\leq \lim_{n\to \infty}\int_{B_\epsilon(0)\setminus \{0\}}w_n^*d\lambda\leq T_2+\frac{\gamma}{2}-T_1,$$
	a contradiction.
	Therefore, there exists a finite energy $\tj$-holomorphic cylinder 
	$$\tw_q:\C\setminus\{0\}\to \R\times S^3$$
	such that, up to a subsequence, 
	$\tv_n\to \tw_q ~\text{ in } C^\infty_{loc}(\C\setminus \{0\}).$ 
	
	For $r\geq1$ and large $n$, we have 
	\begin{equation}
		T_2-\frac{\gamma}{2}\leq \int_{\partial B_{\delta_nr}(0)}w_n^*\lambda\leq \int_{\partial \D}w_n^*\lambda=T_3-\frac{\gamma}{2}.
	\end{equation}
	Here we have used \eqref{eq:germination-cilindro-exemplo1} and \eqref{eq:existencia-seq-tau-n-cilindro-exemplo}. 
	Therefore, we have
	\begin{equation}
		\int_{\partial B_r(0)}w_q^*\lambda=\lim_{n\to \infty}\int_{\partial B_r(0)}v_n^*\lambda=\lim_{n\to \infty}\int_{\partial B_{\delta_nr}(0)}w_n^*\lambda\in\left[T_2-\frac{\gamma}{2},T_3-\frac{\gamma}{2}\right].
	\end{equation}
	We conclude that $\tw_q$ is asymptotic to $P_2$ at the positive puncture $z=\infty$. 

	Using \eqref{eq:existencia-seq-tau-n-cilindro-exemplo}, we conclude that 
	$$\int_{\C\setminus \{0\}}w_q^*d\lambda\geq \int_{\C\setminus\D}w_q^*d\lambda =T_2-\lim_{n\to \infty}\int_{\partial \D}v_n^*\lambda=T_2-\lim_{n\to \infty}\int_{\partial B_{\delta_n}(0)}w_n^*d\lambda=\frac{\gamma}{2}>0.$$	
	Thus, any asymptotic limit of $\tw_q$ at its negative puncture $z=0$ is an orbit that is not linked to $P_3$ and has period $<T_2$. 
	We conclude that $\tw_q$ is asymptotic to $P_1$ at $z=0$. 
	This completes the proof of (i) and (ii).
	
	An application of Lemma \ref{le:cylinders-small-area-exemplo}, similar to the proof of Proposition \ref{pr:familia-de-cilindros-se-aproxima-limite}, proves (iii), (iv) and (v). Assertion  (vi) is a consequence of (i)-(v).

		Following the same arguments as in the proof of Propositions  \ref{pr:tw=tu'}-\ref{pr:self-linking-1}, we conclude that
	\begin{itemize}
		\item Up to reparametrization and $\R$-translation,  we have $\tw_r=\tu_r$ and $\tw'_r=\tu_r'$, where $\tu_r$ and $\tu_r'$ are  
		given by {Claim I};
		\item Up to reparametrization  and $\R$-translation we have $\tw_q=\tw'_q$;
		\item The images of the family $\{w_\tau\},\tau \in (\tau_-,+\infty)$, $u_r$, $u_r'$, $w_q$, $x_1$, $x_2$ and $x_3$ determine a singular foliation of the region $\mathcal{R}_2=\overline{S^3\setminus \mathcal{R}_1}$. 
		\item ${\rm sl}(P_i)=-1$, $i=1,2,3$. 
	\end{itemize}
It follows that there exists a $3-2-1$ foliation with binding orbits $P_1$, $P_2$ and $P_3$.
The existence of a homoclinic to $P_2$ is proved as in \S \ref{se:conclusao-prova}.
\end{proof}

\section{Proof of Proposition \ref{pr:exemplo-prova}}\label{se:proof-prop-exemplo}
In this section we prove Proposition \ref{pr:exemplo-prova}, which is restated below.
\begin{proposition}
	For sufficiently small $\epsilon$, the contact form $\lambda=\lambda_0|_S$, where $S=H^{-1}(\frac{1}{2})$, and the Reeb orbits $P_1$, $P_2$, and $P_3$, defined by \eqref{eq:defi-orbitas}-\eqref{eq:periodo-exemplo}, satisfy the hypotheses of Theorem \ref{pr:example}. Therefore, there exists a $3-2-1$ foliation adapted to $\lambda$ with binding orbits $P_1$, $P_2$, and $P_3$. 
\end{proposition}

The proof of Proposition \ref{pr:exemplo-prova} follows immediately from \eqref{eq:t1-t2-t3}, \eqref{eq:t3menort1} and 
Lemmas \ref{le:exemplo-indices}-\ref{le:exemplo-nao-existe-esfera} below.

\begin{lemma}\label{le:exemplo-indices}
	Assuming that $\epsilon$ is sufficiently small, the orbits $P_1$, $P_2$ and $P_3$ are nondegenerate and their  Conley-Zehnder indices are respectively $1,~2$ and $3$. 
\end{lemma}
\begin{proof}
	First we define a trivialization of $\xi=\ker \lambda$. 
	Consider the matrices $A_0={\rm{Id}}$,
	\begin{equation*}
	A_1=\left[\begin{matrix}
			0&0&0&1\\
			0&0&1&0\\
			0&-1&0&0\\
			-1&0&0&0
		\end{matrix}\right],~
		A_2=\left[\begin{matrix}
			0&0&-1&0\\
			0&0&0&1\\
			1&0&0&0\\
			0&-1&0&0
		\end{matrix}\right],~
		A_3=
		\left[\begin{matrix}
			0&-1&0&0\\
			1&0&0&0\\
			0&0&0&-1\\
			0&0&1&0
		\end{matrix}\right].
	\end{equation*}
	For each $z\in S$ we have an ortonormal frame of $T_z\R^4\simeq \R^4$ given by $\{X_i(z)\},~i=0,\dotsc,3$, where
	\begin{equation}
		X_i:=A_i \left(\frac{\nabla H(z)}{\|\nabla H(z)\|}\right),~i=0,1,2,3.
	\end{equation}
	Note that $X_3(z)=\frac{X_H(z)}{\|\nabla H(z)\|}$ and $T_zS={\rm span}\{X_1,X_2,X_3\}$.
	
	The contact structure $\xi=\ker \lambda$ is isomorphic to ${\rm span}\{X_1,X_2\}$, as a tangent hyperplane distribution, via the projection 
	$\pi_{12}:TS\to {\rm span}\{X_1,X_2\}$ along $X_3$. 
	Let $\bar{X}_i\in \xi$ be the vector field determined by 
	\begin{equation}\label{eq:definition-xbar}
		\pi_{12}(\bar{X}_i)=X_i,~i=1,2.
	\end{equation}
	We can define a symplectic trivialization of $\xi$ by 
	\begin{equation}
		\Psi:S\times \R^2\to \xi;~~ \Psi(z,\alpha_1,\alpha_2)=\alpha_1\bar{X}_1(z)+\alpha_2 \bar{X}_2(z).
	\end{equation}
	Note that along the orbits $P_i$, $i=1,2,3$, ${\rm span}\{X_1,X_2\}$ coincides with the plane $(x_2,y_2)$. Indeed, along the orbits, we have $P=Q=0$ and
	\begin{equation}
		\begin{aligned}
			X_1(z)&=\frac{1}{x_1^2+y_1^2}(0,0,-y_1,-x_1)\\
			X_2(z)&=\frac{1}{x_1^2+y_1^2}(0,0,x_1,-y_1).
		\end{aligned}
	\end{equation}

	We can define another trivialization along the orbits $P_i$, $i=1,2,3$, by
	\begin{equation}
		\rho_i:\R\times \R^2\to x_i^*\xi;~~(t,\beta_1,\beta_2)\mapsto \left(\pi_{12}^{-1}\right)_{x_i(t)}(\beta_1(0,0,1,0)+\beta_2(0,0,0,1)).
	\end{equation} 
	It is easy to check that $\rho_i$,~$i=1,2,3$, is a symplectic trivialization. 
	A simple calculation shows that, in the trivialization $\rho_i$,~$i=1,2,3$, $dR_\lambda(\gamma_i(t))$ is represented by 
	\begin{equation}
		dR_\lambda(\gamma_i(t))=	\left[
		\begin{matrix}
			0&-h(\gamma_i(t))\dfrac{\partial P}{\partial y_2}(p_i,0)\\
			h(\gamma_i(t))\dfrac{\partial Q }{\partial x_2}(p_i,0)&0
		\end{matrix}
		\right] =: \left[
		\begin{matrix}
			0&k_1^i\\
			k_2^i&0
		\end{matrix}
		\right]
	\end{equation}
	which is a constant linear map.  
	Note that, for sufficiently small $\epsilon>0$, we have 
	$0<|k^i_1|,|k_2^i|<<1, ~~ i=1,2,3$.

	Let $\varphi^t$ be the flow of $R_\lambda$. 
	The linearized flow $d\varphi^t(x(0))$ along a trajectory $x(t)$ is the solution of the equation
	\begin{equation}\label{eq:edo-fluxo-linearizado}
		\dfrac{d}{dt}d\varphi^t(x(0))=dR_\lambda(x(t))\cdot d\varphi^t(x(0)).
	\end{equation}

		 For $i=1$, we have  $k_1^1>0$ and $k_2^1<0$. Solving equation \eqref{eq:edo-fluxo-linearizado} we find 
		\begin{equation}
			d\varphi^t_{x_1(0)}=\left[\begin{matrix}
				\cos(\sqrt{-k_1^1k_2^1}t)&\frac{\sqrt{k_1^1}}{\sqrt{-k_2^1}}\sin(\sqrt{-k_1^1k_2^1}t)\\
				-\frac{\sqrt{-k_2^1}}{\sqrt{k_1^1}}\sin(\sqrt{-k_1^1k_2^1}t)&\cos(\sqrt{-k_1^1k_2^1}t)
			\end{matrix}\right].
		\end{equation}
		Thus, if $\epsilon$ is sufficiently small, $P_1$ is nondegenerate and for any $z\in \R^2$, $z(t):=d\varphi^t_{x_1(0)}z$ has winding number $-1<\Delta(z)<0$, where $\Delta(z)$ is defined by \eqref{eq:winding-cz}.  We conclude that $\mu(P_1,\rho_1)=-1$.
		 	
		 For $i=2$, we have $k_1^2>0$ and $k_2^2>0$. Solving equation \eqref{eq:edo-fluxo-linearizado} we find 
		\begin{equation}
			d\varphi^t_{x_2(0)}=\left[\begin{matrix}
				\cosh(\sqrt{k_1^2k_2^2}t)&\frac{\sqrt{k_1^2}}{\sqrt{k_2^2}}\sinh(\sqrt{k_1^2k_2^2}t)\\
				\frac{\sqrt{k_2^2}}{\sqrt{k_1^2}}\sinh(\sqrt{k_1^2k_2^2}t)&\cosh(\sqrt{k_1^2k_2^2}t)
			\end{matrix}\right].
		\end{equation}
		Note that 
		\begin{equation}
			\begin{aligned}
				d\varphi^t_{x_2(0)}\left(\frac{\sqrt{k_1^2}}{\sqrt{k_2^2}},1\right)&=e^{\sqrt{k_1^2k_2^2}t}\left(\frac{\sqrt{k_1^2}}{\sqrt{k_2^2}},1\right),~~\forall t\in \R,\\
				d\varphi^t_{x_2(0)}\left(-1,\frac{\sqrt{k_2^2}}{\sqrt{k_1^2}}\right)&=e^{-\sqrt{k_1^2k_2^2}t}\left(-1,\frac{\sqrt{k_2^2}}{\sqrt{k_1^2}}\right),~~\forall t\in \R.
			\end{aligned} 
		\end{equation}
		We conclude that $P_2$ is nondegenerate and $\mu(P_2,\rho_2)=0$.
		
		For $i=3$, we have $k_1^3<0$ and $k_2^3>0$. Solving equation \eqref{eq:edo-fluxo-linearizado} we find 
		\begin{equation}
			d\varphi^t_{x_3(0)}=\left[\begin{matrix}
				\cos(\sqrt{-k_1^3k_2^3}t)&-\frac{\sqrt{-k_1^3}}{\sqrt{k_2^3}}\sin(\sqrt{-k_1^3k_2^3}t)\\
				\frac{\sqrt{k_2^3}}{\sqrt{-k_1^3}}\sin(\sqrt{-k_1^3k_2^3}t)&\cos(\sqrt{-k_1^3k_2^3}t)
			\end{matrix}\right]
		\end{equation}
		Thus, if $\epsilon$ is sufficiently small, $P_3$ is nondegenerate and for any $z\in \R^2$, $z(t):=d\varphi^t_{x_3(0)}z$ has winding number $0<\Delta(z)<1$.  We conclude that $\mu(P_3,\rho_3)=1$. 

	Now, the Conley-Zehnder index of the orbit $P_i$, $i=1,2,3$, satisfies the formula
	$$\mu(P_i)=\mu(P_i,\Psi)=\mu(P_i,\rho_i)+2\wind((\rho_i)_t(1,0),\Psi)$$
	and 
	\begin{equation}
		\begin{aligned}
			\wind((\rho_i)_t(1,0),\Psi)
			&=\deg\left((\pi_{12}^{-1})_{\gamma_i(t)}(0,0,1,0),\Psi\right)\\
			&=\deg\left(\R/T_i\Z\ni t\mapsto \Psi^{-1}_{\gamma_i(t)}(\pi_{12}^{-1})_{\gamma_i(t)}(0,0,1,0)\right)\\
			&=\deg\left(\R/T_i\Z\ni t\mapsto \left(-r_i\sin\left(\frac{2}{r_i^2}t\right),r_i\cos\left(\frac{2}{r_i^2}t\right)\right)\right)\\
			&=1.
		\end{aligned}
	\end{equation}
	We conclude that $\mu(P_1)=1$, $\mu(P_2)=2$ and $\mu(P_3)=3$.
\end{proof} 

\begin{lemma}\label{le:exemplo-plano-cilindro}
	There exists $J\in \mathcal{J}(\lambda)$ such that the almost complex structure $\tj=(\lambda,J)$ admits a finite energy plane $\tu:\C\to \R\times S$ asymptotic to $P_3$ at its positive puncture $z=\infty$ and a finite energy cylinder $\tv:\C\setminus \{0\}\to \R\times S$ asymptotic to $P_3$ at its positive puncture $z=\infty$ and $P_1$ at its negative puncture $z=0$.
\end{lemma}
\begin{proof}
	Our proof follows \cite[\S 5]{dePS2013}.
	Define a $d\lambda$-compatible complex structure $J:\xi\to \xi$ by 
	\begin{equation}\label{eq:definition-J}
		J \bar{X}_1=\bar{X}_2
	\end{equation}
	where $\bar{X}_i$, $i=1,2$ is defined by \eqref{eq:definition-xbar}.
	Let $\tilde{J}=(\lambda,J)$ be the almost complex structure on $\R\times S$ defined by \eqref{eq:J-til}.
	
	First we look for the $\tj$-holomorphic plane. 
	Our candidate for the plane $\tu$ will project into the surface 
	\begin{equation}\label{eq:surface-image}
		\{(x_1,y_1,x_2,0)\in \R^4~|~x_2>p_3\}\cap S.
	\end{equation}
	We write $\tu(s,t)=\tu(e^{2\pi(s+it)})$ for $(s,t)\in \R\times S^1$. 
	First we assume that $\tu=(a(s,t),u(s,t)):\R\times S^1\to \R\times S$ satisfies 
	\begin{equation}\label{eq:defi-u}
		u(s,t)=(f(s)\cos2\pi t,f(s)\sin 2\pi t,g(s),0),
	\end{equation}
	where $f:\R\to (0,+\infty)$ and $g:\R\to \R\setminus \{0\}$ are smooth functions to be determined. 
	
	Recall that $\tu$ is $\tj$-holomorphic if and only if it satisfies
	\begin{align}
		\pi u_s+J\pi u_t=0  \label{eq:jhol1}\\  
		da\circ i=-u^*\lambda   \label{eq:jhol2},
	\end{align}
	where $\pi:TS\to \xi$ is the projection along the Reeb vector field $R_\lambda$. 
	
	Now we find conditions on $f(s)$ and $g(s)$ so that $u(s,t)$ defined by \eqref{eq:defi-u} satisfies \eqref{eq:jhol1}.
	First note that 
	\begin{equation}\label{eq:derivadasu}
		\begin{aligned}
			\pi u_s&=u_s-\lambda(u_s)R_\lambda\\
			&=(f'(s)\cos 2\pi t, f'(s)\sin 2\pi t,g'(t),0)\\
			\pi u_t&=u_t-\lambda(u_t)R_\lambda\\
			&=(-2\pi f(s)\sin 2\pi t,2\pi f(s) \cos 2\pi t,0,0)\\
			&~~~~~~~~-\pi f(s)^2h(-f(s)\sin 2\pi t,f(s)\cos 2\pi t,0, Q).
		\end{aligned}
	\end{equation}
	Restricting the frame $\{\bar{X}_1,\bar{X}_2\}$ to the surface \eqref{eq:surface-image} we find
	\begin{equation}\label{eq:framex}
		\begin{aligned}
			\bar{X}_1&=X_1-\lambda(X_1)R_\lambda\\
			&=(0,Q,-y_1,-x_1)-h(x_1Q-x_2x_1)(-y_1,x_1,0,Q)\\
			\bar{X}_2&=X_2-\lambda(X_2)R_\lambda\\
			&=(-Q,0,x_1,-y_1)-h(y_1Q-x_2y_1)(-y_1,x_1,0,Q).
		\end{aligned}
	\end{equation}
	Using \eqref{eq:derivadasu} and \eqref{eq:framex} we find
	\begin{equation}
		\begin{aligned}
			\bar{X}_1(u(s,t))&=-\frac{y_1}{g'(s)}\pi u_s +x_1\frac{(1+h(Q-x_2)Q)}{h\pi f(s)^2Q}\pi u_t\\
			\bar{X}_2(u(s,t))&=\frac{x_1}{g'(s)}\pi u_s+y_1\frac{(1+h(Q-x_2)Q)}{h\pi f(s)^2Q}\pi u_t.
		\end{aligned}
	\end{equation}
	Now using the definition of $J$ \eqref{eq:definition-J} and equation \eqref{eq:jhol1}, we find the differential equation
	\begin{equation}\label{eq:edo1}
		g'(s)=\frac{-h\pi f(s)^2Q}{1+h(Q-x_2)Q}.
	\end{equation}
	Since the image of $u$ is the surface \eqref{eq:surface-image}, we have
	\begin{equation}\label{eq:relacao-f-e-g}
		\frac{1}{2}f(s)^2=\frac{1}{2}-\left(\frac{g(s)^4}{2}+\epsilon a g(s)^3+\epsilon^2cg(s)^2\right).
	\end{equation}
	Moreover, $h(u(s,t))=2(f(s)^2+g(s)Q(g(s),0))^{-1}$.
	Thus, we may view \eqref{eq:edo1} as an ODE of the type
	\begin{equation}
		g'(s)=G(g(s)),
	\end{equation}
	where $G:[p_3,\bar{x}]\to \R$ is a smooth function. Here $\bar{x}$ is the unique positive solution of $H_2(x,0)=\frac{1}{2}$. 
	
	Note that $G$ vanishes in $g=p_3$ and in $g=\bar{x}$, and G is negative on $(p_3,\bar{x})$. 
	Thus, $g'(x)<0$ for $g(s)\in (p_3,\bar{x})$. Consequently $g(s)$ is strictly decreasing and satisfies 
	\begin{align}\label{eq:limite-g(s)}
		\lim_{s\to-\infty}g(s)=\bar{x}&&\lim_{s\to +\infty}g(s)=p_3.
	\end{align}
	It follows from \eqref{eq:relacao-f-e-g} and \eqref{eq:limite-g(s)} that 
	\begin{align}\label{eq:limite-f(s)}
		\lim_{s\to -\infty}f(s)=0&&\lim_{s\to +\infty}f(s)=r_3,
	\end{align}
	where $r_3$ is defined by \eqref{eq:defi-orbitas}.
	Using \eqref{eq:limite-g(s)} and \eqref{eq:limite-f(s)} we conclude that the loops $S^1\ni t\mapsto u(s,t)$ converge to $S^1\ni t \mapsto x_3(T_3t)$ in $C^\infty$ as $s\to +\infty$. 
	Moreover, the loops $S^1\ni t\mapsto u(s,t)$ converge to $(0,0,\bar{x},0)$ in $C^\infty$ as $s\to \infty$. 
	
	Now we define the function $a:\R\times S^1\to \R$  by
	\begin{equation}\label{eq:definicao-a}
		a(s,t)=\pi\int_0^sf(\tau)d\tau.
	\end{equation}
	Then  
	\begin{equation}
		\begin{aligned}
			a_s(s,t)&=\pi f(s)^2=\lambda(u_t)\\
			a_t(s,t)&=0=-\lambda(u_s).
		\end{aligned}
	\end{equation}
	
	Consider the map $\tu=(a,u):S^1\times \R\to \R\times S$ with $u$ defined by \eqref{eq:defi-u}, where $g$ is a solution of \eqref{eq:edo1} satisfying $g(0)\in (p_3,\bar{x})$, $f$ is defined by \eqref{eq:relacao-f-e-g}, and $a$ defined by \eqref{eq:definicao-a}.
	Since $\tu=(a,u)$ satisfies \eqref{eq:jhol1} and \eqref{eq:jhol2}, $\tu$ is $\tj$-holomorphic.
	Moreover, $\tu$ has finite energy. Indeed, using Stokes theorem and equations \eqref{eq:limite-f(s)} and \eqref{eq:definicao-a} we find
	\begin{equation}
		\begin{aligned}
			\int_{\R\times S^1} \tu^*d(\phi\lambda)&=\lim_{R\to+\infty}\left(\phi(a(R))\pi f(R)^2-\phi(a(-R))\pi f(-R)^2\right)\\
			&=\lim_{R\to +\infty}\phi(R)\pi r_3^2\\
			&=\lim_{R\to +\infty}\phi(R)T_3
		\end{aligned}
	\end{equation}
	where $\phi\lambda(a,u)=\phi(a)\lambda(u)$, $\phi:\R\to [0,1]$ is smooth and $\phi'\geq 0$. 
	We conclude that $E(\tu)=T_3$. 
	
	The mass $m(-\infty)$ of the singularity $z=-\infty$ satisfies 
	$$m(-\infty)=\lim_{s\to -\infty}\int_{\{s\}\times S^1}u^*\lambda=\lim_{s\to -\infty}\pi f(s)^2=0.$$
	This implies that $\tu$ has a removable singularity at $s=-\infty$. Removing it, we obtain a finite energy $\tj$-holomorphic plane $\tu:\C\to\R\times S^1$ asymptotic to $P_3$ at its positive puncture $z=\infty$. 
	
	{The construction of the cylinder $\tv:\C\setminus \{0\}\to \R\times S^1$ is completely analogous, but we start with a solution of \eqref{eq:edo1} satisfying $g(0)\in (p_1,p_3)$.}
	In this case, we obtain 
	\begin{align*}
		\lim_{s\to-\infty}g(s)=p_1&&\lim_{s\to +\infty}g(s)=p_3\\
		\lim_{s\to -\infty}f(s)=r_1&&\lim_{s\to +\infty}f(s)=r_3,
	\end{align*}
	and the singularity $z=-\infty$ is non-removable, since 
	$$m(-\infty)=\lim_{s\to -\infty}\int_{\{s\}\times S^1}v^*\lambda=\lim_{s\to -\infty}\pi f(s)^2=\pi r_1^2=T_1.$$
\end{proof}

\begin{lemma}\label{le:exemplo-linking}
	Assuming that $\epsilon$ is sufficiently small, if $P=(\gamma, T)\in\mathcal{P}(\lambda)$ satisfies $P\neq P_3,~T\leq T_3$  \text{ and } ${\rm lk}(P,P_3)=0$, then $P\in\{P_1,P_2\}$.
\end{lemma}
\begin{proof}
First we prove the following claim. 
	\paragraph{\textit{Claim I}}
	Let $\gamma:\R\to \R^{2n}$ be a nonconstant $T$-periodic solution of $\dot{\gamma}(t)=X_{H}(\gamma(t))$, where $H:\R^{2n}\to \R$.  Then $h T\geq 2\pi$, where $h:=\sup_{t\in [0,T]}|d^2H(\gamma(t))|$.

	To prove the claim, first we note that $\dot{\gamma}$ satisfies
	\begin{equation}\label{eq:fourier}
		\|\dot{\gamma}\|_{L^2}\leq \frac{T}{2\pi}\|\ddot{\gamma}\|_{L^2}.
	\end{equation}
	This follows from the fact that $\int_0^T\dot{\gamma}=0$ and from the representations of $\dot{\gamma}$ and $\ddot{\gamma}$ as  Fourier series. 
	Now we have
	\begin{equation}\label{eq:fourier-ineq}
		\|\ddot{\gamma}\|_{L^2}=\|(dX_H\circ\gamma)\dot{\gamma}\|_{L^2}
		\leq \left(\int_0^T|dX_H(\gamma(t))|^2\|\dot{\gamma}(t)\|^2 dt\right)^{\frac{1}{2}}
		\leq h\|\dot{\gamma}\|_{L^2}.
	\end{equation}
	We conclude from \eqref{eq:fourier} and \eqref{eq:fourier-ineq} that $h T\geq 2\pi$.

	Consider the function $\R^3\ni (\epsilon,x,y)\mapsto H_2^\epsilon(x,y)$, where 
	$$H_2^\epsilon(x,y)=(x^2+y^2)^2-\epsilon (x^2+y^2)x-\frac{\epsilon}{2}(x^2+y^2).$$
	Then $\R^3\ni (\epsilon,x,y)\mapsto d^2(H_2^\epsilon)(x,y)$ is continuous and $d^2(H_2^0(0,0))=0$. Hence there exists $r>0$ such that, for $0<\epsilon<r$ and $(x,y)\in B_r(0)\subset \R^2$, 
	\begin{equation}\label{eq:estimate-hessian}
		|d^2H_2^\epsilon(x,y)|<\frac{1}{2}.
	\end{equation}
	For $\epsilon>0$ sufficiently small, we also have
	\begin{equation}\label{eq:a-epsilon}
		A_\epsilon:=\{(x,y)\in \R^2~|~H_2^\epsilon(x,y)\leq H_2^\epsilon(p_1,0)\}\subset B_r(0).
	\end{equation}
	Indeed, using the fact that $(\epsilon,x,y)\mapsto H_2^\epsilon(x,y)$  is a proper function, we know that the set $A_\epsilon$ is bounded by a uniform constant for every $0<\epsilon<r$. Thus, for sufficiently small $\epsilon$ and $(x,y)\in A_\epsilon$, we have
		\begin{equation}
			\begin{aligned}
				\|(x,y)\|^4&=(x^2+y^2)^2\\
				&\leq H_2^\epsilon(p_1,0)+\epsilon(x^2+y^2)x+\frac{\epsilon}{2}(x^2+y^2)\\
				&\leq r^4.
			\end{aligned}
		\end{equation}

	Fix $\epsilon>0$ so that \eqref{eq:a-epsilon} is satisfied and consider the Hamiltonian function $H=H_\epsilon$ defined by \eqref{eq:hamiltoniana}.
	Let $P=(\gamma, T)\in \mathcal{P}(\lambda)\setminus \{P_1,P_2,P_3\}$ be a simple orbit  and let $z(t)$ be a periodic solution of $\dot{z}(t)=X_H(z(t))$ that is a reparametrization of $\gamma$. 
	Then $z=z_1\times z_2$, where 
		$\dot{z_1}(t)=X_{H_1}(z_1(t))$, $\dot{z_2}(t)=X_{H_2}(z_2(t))$ and $z_2(t)$ is nonconstant. Now we prove the following claim. 

	\paragraph{\textit{Claim II}}If $z_2$ satisfies 
	\begin{equation}\label{eq:constant-C}
		H_2\circ z_2 	\equiv C\leq H_2(p_1,0), 
	\end{equation}
then the period $T$ of $P$ satisfies  $T<T_3$.
	
	Using Claim I and equation \eqref{eq:estimate-hessian} we conclude that the period $\tilde{T}_2$ of $z_2$ satisfies
	$$4\pi\leq 2h\cdot \tilde{T}_2<\tilde{T}_2.$$
	Note that the orbit 
	$z_1(t)=(\sqrt{1-2C}\cos(t+t_0),\sqrt{1-2C}\sin(t+t_0))$ is nonconstant, since $C\leq H_2(p_1,0)<\frac{1}{2}$, 
	and $2\pi$-periodic. This implies that the period $\tilde{T}$ of $z$ satisfies 
	$\tilde{T}=k2\pi$,  
	for some  $k>2$.
	The period $T$ of $\gamma$ satisfies 
	\begin{equation}
		\begin{aligned}
			T&=\int_{\gamma(\R)}\lambda\\
			&=\int_0^{\tilde{T}}\frac{1}{2}z_1^*(x_1dy_1-y_1dx_1)+\int_0^{\tilde{T}} \frac{1}{2} z_2^*(x_2dy_2-y_2dx_2)\\
			& >2 \pi (1-2C)+{\rm area}(R_2)\\
			&\geq2\pi ( 1-2H_2(p_1,0))\\
			&=2T_1
		\end{aligned}
	\end{equation}
	where $R_2$ is  the region in the plane $(x_2,y_2)$ limited by the image of $z_2$. Here we have used \eqref{eq:periodo-exemplo} and \eqref{eq:constant-C}.
	Taking a smaller $\epsilon$ if necessary and using  \eqref{eq:t3menort1}, we have 
	$$T>2T_1>T_3,$$
	proving the claim.
	
	Let $P=(\gamma, T)\in \mathcal{P}(\lambda)\setminus \{P_1,P_2,P_3\}$ be such that $T\leq T_3$. We claim that ${\rm lk}(P,P_3)\neq 0$.
	By Claim II, we know that the solution $z_2$ does not satisfy \eqref{eq:constant-C}, that is, $H_2\circ z_2\equiv C>H_2(p_1,0)$.
		The energy levels $H_2=C$, for $C>H_2(p_1,0)$ are diffeomorphic to $S^1$ and bound a region containing $(p_3,0)$ in its interior. 
		Then $P$ intersects the disk
		$$D:=\{(x_1,y_1,x_2,0)\in \R^4~|~x_2\geq p_3\}\cap S.$$
		Note that $\partial D=P_3$.
		Moreover $D$ is transverse to the Reeb vector field $R_\lambda$, since $X_{H_2}(x_2,0)=(0,Q)$. 
		We conclude that ${\rm lk}(P,P_3)\neq 0$, which proves the lemma.
	 
%
\end{proof}


\begin{lemma}\label{le:exemplo-nao-existe-esfera}
	There is no $C^1$-embedding $\psi:S^2\to S^3$ such that $\psi({S^1\times \{0\}})=P_2$ 
	and each hemisphere of $\psi(S^2)$ is a strong transverse section.
\end{lemma}
\begin{proof}
	Let $W^+(P_2)$ and $W^-(P_2)$ be the stable and unstable manifolds of $P_2$ respectively. 
	First note that 
	$$W^+(P_2)=W^-(P_2)=H_1^{-1}\left(\frac{1}{2}-H_2(p_2,0)\right)\times H_2^{-1}(H_2(p_2,0)).$$
	Moreover, 
	$W^\pm(P_2)\setminus P_2$ 
	consists of two connected components diffeomorphic to cylinders
	\begin{align}
		C_1=H_1^{-1}\left(\frac{1}{2}-H_2(p_2,0)\right)\times \gamma^1(\R)&&C_2=H_1^{-1}\left(\frac{1}{2}-H_2(p_2,0)\right)\times \gamma^2(\R),
	\end{align}
	where $\gamma^1:\R\to\R^2$ and $\gamma^2:\R\to \R^2$ are  solutions  of $\dot{\gamma}=X_{H_2}(\gamma)$ such that 
	$$H^{-1}_2(H_2(p_2,0))=\{(p_2,0)\}\cup \gamma_2^1(\R)\cup \gamma_2^2(\R), $$  
	{and $\gamma^1(\R)$ is contained in the interior of the bounded region $R_2\subset\R^2$ with boundary $\gamma^2(\R)\cup\{(p_2,0)\}$.}
	
	Define $\mathcal{T}_i:=C_i\cup P_2,~i=1,2.$ 
	Then $\mathcal{T}_i$ is a $2$-torus topologically embedded in $S$ and divides $S$ into two closed regions with boundary $\mathcal{T}_i$. One of these regions, that we call $\mathcal{R}_i^1$, has holomogy generated by $P_2$, and $P_2$ is contractible in $\mathcal{R}_i^2:=\overline{S\setminus \mathcal{R}_i^1}$.
	Since the disk 
	$$D:=\{(x_1,y_1,x_2,0)\in \R^4~|~x_2\leq p_2\}\cap S$$
	has boundary $P_2$ and projects into the complement of $R_2$ in the plane $(x_2,y_2)$, we know that $\mathcal{R}^2_i$, $i=1,2$, also projects into the complement of $R_2$ in the plane $(x_2,y_2)$. This implies that $\mathcal{R}^1_i$ projects into $R_2$, for $i=1,2$. Moreover, we have $\mathcal{R}_1^1\subset \mathcal{R}_2^1$.
	

	Following \S \ref{se:relative-position}, let $\{v^-(t),v^+(t)\}$ be the positive basis of $\xi_{x_2(t)}$ where $v^-(t)$ is an eigenvector of  $d\varphi^T|_{\xi(x_2(t))}$ associated to the eigenvalue $\beta>1$ and $v^+(t)$ is an eigenvector of  $d\varphi^T|_{\xi(x_2(t))}$ associated to the eigenvalue $\beta^{-1}$.
	Let $\rm{(I)}$ and $\rm{(III)}$ be the open quadrants between  $\R v^-(t)$ and $\R v^+(t)$ following the positive orientation and $\rm{(II)}$ and $\rm{(IV)}$ the open quadrants between  $\R v^+(t)$ and $\R v^-(t)$.
	
	Let $\R/\Z\ni t\mapsto\eta(t)\in \xi_{x_2(T_2t)}$ be a section of $\xi$ along $x_2(T_2\cdot)$ such that $\{\eta, R_\lambda(T_2\cdot)\}$ generates $d\psi(TS^2)$ along $x_2(T_2\cdot)$. 
	We claim that 
	\begin{equation}\label{eq:relative-position-example}
		\eta(t)\in (II)\cup (IV), ~\forall t\in \R/\Z.
	\end{equation}
	By Proposition \ref{pr:secoes-fortes-wind-1} it is enough to show that 
	\begin{align}
		\wind(\eta,\Psi)=1 \label{eq:wind=1-exemplo}\\
		d\lambda(\eta(t),\mathcal{L}_{R_\lambda}\eta(t))<0,~\forall t\in \R/\Z.\label{eq:transversal-forte-positiva-exemplo}
	\end{align}	
	where $\Psi$ is any global symplectic trivialization of $\xi$. 
	The proof of \eqref{eq:wind=1-exemplo} is completely analogous to the proof of Lemma \ref{le:wind=1} 
	and the proof of \eqref{eq:transversal-forte-positiva-exemplo} is similar to the proof of \eqref{eq:dlambda-negativo}. 
	
	%
	Suppose, by contradiction,  that $\psi:S^2\to S^3$ is a $C^1$ embedding such that $\psi({S^1\times \{0\}})=P_2$ and each hemisphere is a strong transverse section.
	Using \eqref{eq:relative-position-example}, we conclude that one of the closed hemispheres of $\Psi(S^2)$, that we denote by $H\subset S^2$, intersects $\mathcal{R}_1^1$. Since $P_2$ is not contractible in $\mathcal{R}_1^1$, the image of $H$ can not be contained in $\mathcal{R}_1^1$ and we conclude that $\psi(S^2)\cap W^\pm(P_2)\neq \emptyset$. 
	
	We claim that $\psi(S^2)\cap W^\pm(P_2)$ consists of a disjoint union of embedded circles. To prove this claim, note that by \eqref{eq:relative-position-example}, there exists an open neighborhood $U$ of $P_2$ in $W^\pm(P_2)$ such that $U\cap \psi(S^2)=P_2$. Thus, $F:=W^\pm(P_2)\setminus U$ is a closed subset of $ S$ such that $\psi(S^2)\cap (W^\pm(P_2)\setminus P_2)\subset F$. 
	Moreover, $\psi|_{S^2\setminus (S^1\times \{0\})}$ is transverse to $W^\pm(P_2)\setminus P_2$.
	This implies that $\psi^{-1}(F)=\psi^{-1}(\psi(S^2)\cap (W^\pm(P_2)\setminus P_2))$ is a one dimensional submanifold of $S^2$ that is a closed subset of $S^2$, proving our claim. 

	At least one of the connected components of $\psi(S^2)\cap W^\pm(P_2)$ is homologous to $P_2$ both in $\mathcal{R}_1^1$ and in $\mathcal{T}_1$. Let $L\subset S^2$ be such that $\psi(L)$ is one of these components. 
	Then $\psi(L)$ and $P_2$ divide $T_1$ into two closed regions. Let $A$ be one of these regions, with orientation induced by $P_2$. 
	Then
	\begin{equation}\label{eq:contradicao-int-0-exemplo}
		0=\int_Ad\lambda=\int_{P_2}\lambda-\int_{\psi(L)}\lambda.
	\end{equation}
	Consider $H$ with the orientation induced by the orientation of $S^1\times \{0\}$. It follows that 
	\begin{equation}
		0<T_2=\int_{\R/\Z}x_2(T_2\cdot)^*\lambda=\int_{H}\psi^*d\lambda
	\end{equation}
	Since $\psi|_{S^2\setminus (S^1\times\{0\})}$ is transverse to the Reeb vector field, we have $\psi^*d\lambda>0$ in $H\setminus(S^1\times \{0\})$. Let $B\subset H$ be the region bounded by $S^1\times \{0\}$ and $L$. Then
	\begin{equation}
		0<\int_B\psi^*d\lambda=\int_{P_2}\lambda-\int_{\psi(L)}\lambda,
	\end{equation}
	contradicting \eqref{eq:contradicao-int-0-exemplo}. This proves the lemma.
\end{proof}

\section*{Acknowledgments}
This work originated from the author’s Ph.D. thesis, written under the supervision
of Prof. Pedro Salomão at the University of São Paulo. 
The author wishes to express her gratitude to Pedro Salomão for suggesting the problem, for many stimulating conversations, and for his support during the preparation of this paper. The author would like to thank Alexsandro Schneider, Ana Kelly de Oliveira, Naiara de Paulo, and Seongchan Kim for pointing out mistakes in previous versions of the paper and the reviewer for the helpful feedback.

\bibliographystyle{amsplain}

\bibliography{biblografia_artigo_3_2_1_abreviada}
%
	
\end{document}